\documentclass[a4paper,8pt]{article}
\usepackage[english]{babel}
\usepackage{geometry}
\usepackage[T1]{fontenc}
\usepackage[latin1]{inputenc}
\usepackage{amsmath,amssymb,mathrsfs,amsthm}
\usepackage{soul}
\usepackage{epsfig}
\usepackage{hyperref}
\usepackage{graphicx}
\usepackage{datetime}
\usepackage{fancyhdr}
\fancypagestyle{plain}{
\lhead{}
\rhead{}
\lfoot{Built \today{} at \currenttime}

}

\usepackage{tocloft}

\newtheorem{theorem}{Th\'eor\`eme}[section]
\newtheorem{proposition}[theorem]{Proposition}

\newtheorem{lemma}[theorem]{Lemme}

\newtheorem{defprop}[theorem]{D\'efinition-Proposition}
\newtheorem{Corollaire}[theorem]{Corollaire}
\newtheorem{definition}[theorem]{D\'efinition}

\newtheorem{remarque}[theorem]{Remarque}

\newcommand{\vc}{\|\cdot\|}
\newcommand{\C}{\mathbb{C}}

\newcommand{\Z}{\mathbb{Z}}

\newcommand{\s}{\mathbb{S}^1}
\newcommand{\lra}{\longrightarrow}
\newcommand{\al}{\alpha}
\newcommand{\la}{\lambda}

\newcommand{\R}{\mathbb{R}}
\newcommand{\cl}{\mathcal{C}^\infty}
\newcommand{\p}{\mathbb{P}}
\newcommand{\eps}{\varepsilon}

\newcommand{\vf}{\varphi}
\newcommand{\si}{\sigma}

\newcommand{\N}{\mathbb{N}}
\newcommand{\z}{\overline{z}}
\newcommand{\pt}{\partial}
\newcommand{\dif}{\frac{\pt}{\pt z}}

\title{Spectre du Laplacien  singulier associ\'e aux m\'etriques canoniques sur $\p^1$}
\date{\today, \currenttime}
\author{Mounir Hajli}
\begin{document}

\maketitle

\begin{abstract}
Nous  construisons
 un op\'erateur diff\'erentiel associ\'e \`a une classe de m\'etriques singuli\`eres sur les fibr\'es en droites sur
$\p^1$. Cet op\'erateur
\'etend la notion du Laplacien classique (voir, par exemple  \cite{heat}). Nous \'etablissons
 qu'il admet un  spectre positif, discret et infini,
 qu'on calcule  explicitement.

  Ce calcul sera une application d'une   th\'eorie des s\'eries de Fourier-Bessel g\'en\'eralis\'ee que nous
 d\'eveloppons dans ce texte.
\end{abstract}

\begin{center}
  \Large{Introduction}
\end{center}

Soit $X$ une vari\'et\'e k\"ahl\'erienne compacte munie une forme de K\"ahler  $\omega_X$
de classe $\cl$ et
 $\overline{E}=(E,h_E)$ un fibr\'e vectoriel holomorphe muni de $h_E$,
 une m\'etrique  hermitienne de classe $\cl$.  \`A cette donn\'ee on attache un op\'erateur diff\'erentiel
 $\Delta_{\overline{E}}^q$ appel\'e l'op\'erateur
 Laplacien
 agissant sur l'espace des $(0,q)$-formes diff\'erentielles de classe $\cl$ \`a coefficients dans
  $E$ pour $q=0,1,\ldots,n$. Alors,  il est bien connu que cet op\'erateur admet un spectre positif, infini et
  discret.\\

 Le calcul explicite du spectre est un prob\`eme int\'eressant.
  Il y a peu de cas pour les quels on sait calculer explicitement le spectre de $\Delta_{\overline{E}}^\ast$.

   Lorsque
  $X=\p^n$ est muni de la m\'etrique de Fubini-Study et $\overline{E}$ est le fibr\'e en droites hermitien
  trivial,  Ikeda et Taniguchi calculent explicitement
   le spectre du Laplacien  $\Delta^q_{\overline{E}}$, voir \cite[th\'eor\`eme 5.2]{Spectra}.
    Ces calculs ont \'et\'e utilis\'es par Gillet et Soul\'e afin
   de calculer explicitement la torsion analytique holomorphe correspondante et par suite \'etablir
   un th\'eor\`eme de Riemann-Roch arithm\'etique, voir \cite[th\'eor\`eme 2.1.1 et  \S 2.3.2]{GSZ}.  L'id\'ee d'Ikeda
   et Taniguchi  consiste \`a utiliser la compatibilit\'e de la structure de $\p^n$
   ,vu comme vari\'et\'e homog\`ene, avec la m\'etrique de Fubini-Study et d'appliquer ensuite
    la th\'eorie des repr\'esentations.
   Plus concr\`etement, si
    $X=G/K$ est une vari\'et\'e riemannienne homog\`ene telle que $G$ agit transitivement sur $X$, et
   $g$ une m\'etrique riemannienne $G$-invariante sur $X$.
   Alors, $\Delta$ le Laplacien associ\'e  est $G$-invariant, et donc ses espaces propres sont des $G$-modules.
   D'apr\`es Ikeda
et Taniguchi, le calcul du spectre de $\Delta$
se ram\`ene donc \`a un probl\`eme de la th\'eorie des representations.\\

\begin{center}
  \Large{Motivation}
\end{center}

Dans ce texte, nous nous int\'eressons  \`a une classe de m\'etriques  singuli\`eres sur les fibr\'es en droites  sur $\p^1$,
l'espace projectif complexe de dimension $1$,  appel\'ees les m\'etriques canoniques. Ce sont des m\'etriques continues mais pas $\cl$ en g\'en\'eral, d\'etermin\'ees uniquement par la structure combinatoire de $\p^1$, vu comme vari\'et\'e torique.  Pour tout $m\in \N^\ast$, la m\'etrique canonique sur $\mathcal{O}(m)$ est donn\'ee comme suit:

\begin{equation}\label{defcanonique}
 h_{{}_{\overline{\mathcal{O}(m)}_\infty}}(s,s)(x)=\frac{|s(x)|^2}{\max(|x_0|,|x_1|)^{2m}}
\end{equation}
o\`u $x=[x_0:x_1]\in \p^1$ et $s$ est une section locale holomorphe de $\mathcal{O}(m)$ autour de $x$. Comme
 $T\p^1\simeq \mathcal{O}(2)$, alors on peut munir $\p^1$ de la forme de K\"ahler singuli\`ere suivante:
\[
\omega_\infty:=\frac{i}{2\pi}\frac{dz\wedge d\z}{\max(1,|z|)^4},
\]
et on note par $h_{{}_{\overline{T\p^1}_\infty}}$ la m\'etrique sur $T\p^1$ correspondante.\\

 \`A priori, il  n'est pas clair  si la th\'eorie classique du Laplacien  peut \^etre \'etendue
  \`a ce genre
   de m\'etriques.\\

   Le but de cet article est double: Nous montrons qu'on peut associer aux m\'etriques canoniques sur $\p^1$ une
th\'eorie d'op\'erateurs singuliers g\'en\'eralisant formellement la th\'eorie classique du Laplacien, et nous
\'etablissons ensuite que ces op\'erateurs poss\'edent des propri\'et\'es analogues au cas classique.

Plus
concr\`etement,  nous montrons   qu'on peut  associer \`a $\omega_\infty$ et \`a $h_{{}_{\overline{\mathcal{O}(m)}_\infty}}$ un op\'erateur Laplacien singulier que nous noterons
    par $\Delta_{\overline{\mathcal{O}(m)}_\infty}$ qui \'etend la notion classique
    du Laplacien. Nous \'etablissons qu'il   poss\`ede  un spectre, not\'e
    $\mathrm{Spec}(\Delta_{\overline{\mathcal{O}(m)}_\infty})$, que nous calculons explicitement.    On notera,
     comme
 dans le cas classique,
  que  $\mathrm{Spec}(\Delta_{\overline{\mathcal{O}(m)}_\infty})$ est discret, infini et positif.
\begin{remarque}
\rm{ Nous utiliserons  ces calculs  afin de  d\'efinir une fonction Z\^eta associ\'ee \`a $\Delta_{\overline{\mathcal{O}(m)}_\infty}
 $ et \`a son \'etude (voir
  \cite{Mounir2}), aussi pour d\'efinir une torsion analytique holomorphe canonique associ\'ee aux m\'etriques canoniques, avec applications \`a la
 g\'eom\'etrie
 d'Arakelov  (voir \cite{Mounir3}).}
 \end{remarque}

   \begin{center}
     \Large{R\'esultats et organisation de l'article}
   \end{center}

 Soit $(\cdot,\cdot)_{L^2,\infty}$
le produit hermitien sur $A^{0,0}(\p^1,\mathcal{O}(m))$ associ\'e  \`a   $\omega_\infty$ et \`a $h_{{\overline{\mathcal{O}(m)}}_\infty}$(voir $\S$\ref{x30}). On note par $\mathcal{H}^{(m)}$ la compl\'etion de cet espace par
 $(\cdot,\cdot)_{L^2,\infty}$.  Comme les m\'etriques sont singuli\`eres alors la construction classique
du Laplacien (voir rappel \S \ref{x29}) n'est plus valable ici.

 Malgr\'e cela, nous construisons (voir \S\ref{x30})
 un op\'erateur diff\'erentiel singulier $\Delta_{\overline{\mathcal{O}(m)}_\infty}$, en associant \`a tout
\'el\'ement de $A^{0,0}(\p^1,\mathcal{O}(m))$ de la forme $f\otimes z^k$, l'\'el\'ement
$\Delta_{\overline{\mathcal{O}(m)}_\infty}(f\otimes
z^k)$  d\'efini sur $\p^1\setminus \s$ comme suit:

\[
\Delta_{\overline{\mathcal{O}(m)}_\infty}(f\otimes
z^k)=-h_{{}_{\overline{T\p^1}_\infty}}\bigl(\dif,\dif\bigr)^{-1}h_{{}_{\overline{\mathcal{O}(m)}_\infty}}\bigl(
z^k,z^k\bigr)^{-1}\frac{\pt}{\pt z}\Bigl(h_{{}_{\overline{\mathcal{O}(m)}_\infty}}(z^k,z^k)\frac{\pt f}{\pt \z}
\Bigr)\otimes z^k.
\]
Observons que lorsque les m\'etriques sont $\cl$, alors on v\'erifie que le Laplacien poss\`ede cette expression
locale.
Nous v\'erifions ensuite que  d\'efinition s'\'etend par lin\'earit\'e pour tout \'el\'ement $\xi \in
A^{0,0}(\p^1,\mathcal{O}(m))$. Contrairement au cas classique, on n'a pas
$\Delta_{\overline{\mathcal{O}(m)}_\infty}\bigl(A^{0,0}(\p^1,\mathcal{O}(m))\bigr)= A^{0,0}(\p^1,\mathcal{O}(m))$, cela
est d\^u bien \'evidemment au caract\`re singulier des m\'etriques. Mais nous
\'etablissons dans le th\'eor\`eme ci-dessous que $\Delta_{\overline{\mathcal{O}(m)}_\infty}$ est \`a valeurs dans $\mathcal{H}^{(m)}$, et qu'il poss\`ede quelques propri\'et\'es analogues au cas classique:

\begin{defprop}[voir proposition \ref{x31}]
On a $\Delta_{\overline{\mathcal{O}(m)}_\infty}:A^{(0,0)}(\p^1,\mathcal{O}(m))\lra \mathcal{H}^{(m)}$
est un op\'erateur lin\'eaire, on l'appelle le Laplacien canonique associ\'e \`a $\omega_\infty$ et \`a  $h_{\overline{\mathcal{O}(m)}_\infty}$. Soient $\xi,\eta\in A^{(0,0)}(\p^1,\mathcal{O}(m))$, on a:
\begin{enumerate}
\item \[
\ker \Delta_{\overline{\mathcal{O}(m)}_\infty}= H^0(\p^1,\mathcal{O}(m)).\]
\item
\[
\bigl(\xi,\Delta_{\overline{\mathcal{O}(m)}_\infty}\eta\bigr)_{L^2,\infty}=
\bigl(\Delta_{\overline{\mathcal{O}(m)}_\infty}\xi,\eta\bigr)_{L^2,\infty}
\]
\item
\[
\bigl(\xi,\Delta_{\overline{\mathcal{O}(m)}_\infty}\xi\bigr)_{L^2,\infty}\geq
0.
\]
\end{enumerate}

\end{defprop}

Comme dans le cas classique, $\Delta_{{\overline{\mathcal{O}(m)}}_\infty}$ admet une extension maximale auto-adjointe
et positive:
\begin{theorem}[voir th\'eor\`eme \ref{z14}]
Pour tout $m\in \N$, l'op\'erateur $\Delta_{{\overline{\mathcal{O}(m)}}_\infty}$ admet une extension maximale
auto-adjointe et
positive.
\end{theorem}

La suite de l'article est d\'edi\'ee \`a l'\'etude des propri\'et\'es spectrales de
$\Delta_{\overline{\mathcal{O}(m)}_\infty}$. Bien \'evidemment la th\'eorie classique du Laplacien n'est plus
applicable ici. \\

Soient $n\in \Z$ et  $m\in \N$. Nous introduisons  la fonction $L_{m,n}$
d\'efinie sur $\C^\ast$ comme suit:
\[
L_{m,n}(z)=-z^m\frac{d}{dz}\bigl(z^{-m} J_n(z)J_{n-m}(z)\bigr) \footnote{ $J_n$ d\'esigne la fonction de Bessel d'ordre $n$.}
\]
et on note par $\mathcal{Z}_{m,n}$ l'ensemble des z\'eros de $L_{m,n}$.  Nous verrons  que $L_{m,n}$ et $\mathcal{Z}_{m,n}$
  jouent un r\^ole fondamental dans l'\'etude de l'op\'erateur $\Delta_{\overline{\mathcal{O}(m)}_\infty}$.
   Notre principal r\'esultat de cet article
s'\'enonce comme suit:
\begin{theorem}[voir th\'eor\`eme \ref{x333}]
Pour tout $m\in\N$, $\Delta_{\overline{\mathcal{O}(m)}_\infty}$ admet un spectre discret, positif et infini, et on a
{{} \[
\mathrm{Spec}(\Delta_{\overline{\mathcal{O}(m)}_\infty})=\Bigl\{0\Bigr\}\bigcup \Bigl\{\frac{\la^2}{4} \Bigl|\, \exists n\in \N,\, \la\in \mathcal{Z}_{m,n}  \Bigr\},
\]}
En plus, on a
\begin{enumerate}
\item Lorsque  $m$ est pair, alors  la multiplicit\'e de $\frac{\la^2}{4}$ est \'egale  \`a $2$ si $\la\in
    \mathcal{Z}_{m,n}$ si $n\geq m+1 $ ou $0\leq n\leq \frac{m}{2}-1$,
     et de multiplicit\'e $1$ si $\la\in \mathcal{Z}_{m,\frac{m}{2}}$.
\item Si $m$ est impair, alors $\frac{\la^2}{4}$ est de multiplicit\'e $2$ si $n\geq m+1$, et vaut
$1$ si $0\leq n\leq m $.
\end{enumerate}\end{theorem}

Le cas $m$=$0$, c-\`a-d l'\'etude du $\Delta_{\overline{\mathcal{O}}_\infty}$, constitue une situation relativement
 facile \`a trait\'ee. En effet, nous d\'emontrons que  le r\'esultat ci-dessus pour $m=0$ r\'esulte
 du th\'eor\`eme \ref{x15}  dont la preuve est
une cons\'equence de la th\'eorie classique des s\'eries
de Fourier Bessel, (voir \cite[\S 18]{Watson}, ou le paragraphe \ref{z10} pour un bref rappel sur cette th\'eorie). En
d'autres termes, nous avons montr\'e que
l'\'etude des propri\'et\'es spectrales de   $\Delta_{\overline{\mathcal{O}}_\infty}$ est
\'equivalente \`a la th\'eorie des s\'eries de Fourier-Bessel classique.\\

\noindent\textbf{Id\'ee de la preuve:}  Nous \'etablissons   que l'\'etude de l'op\'erateur $\Delta_{\overline{\mathcal{O}(m)}_\infty}$
   est \'equivalente \`a une th\'eorie
 des s\'eries de Fourier-Bessel g\'en\'eralis\'ee que nous allons d\'evelopper dans cet article.\\

\noindent\textbf{Strat\'egie  de la preuve:} La preuve du th\'eor\`eme ci-dessus sera une cons\'equence
 d'une suite de r\'esultats. Ces derniers sont r\'epartis en deux sections \S \ref{D1} et \S \ref{V1}. \\

  Le but de la section \ref{D1} consiste \`a construire explicitement
une famille de vecteurs propres pour $\Delta_{\overline{\mathcal{O}(m)}_\infty}$. Nous commen\c{c}ons tout
d'abord
par \'etudier les fonctions $L_{m,n}$ et les ensembles $\mathcal{Z}_{m,n}$ (voir lemme \ref{zerosimple}). Apr\`es,
nous associons  \`a tout $\la\in \mathcal{Z}_{m,n}$ un \'el\'ement
$\vf_{n,\la}^{(m)}\in\mathcal{H}^{(m)}$ (voir \ref{z2}), et
nous \'etablissons que la famille
$\{\vf_{n,\la}^{(m)}|\,n\in \Z, \la\in \mathcal{Z}_{m,n}\}$ est orthogonale par
rapport \`a $(,)_{L^2,\infty}$,  nous d\'eduisons alors que $\mathcal{Z}_{m,n}$ est un sous-ensemble discret
 de $\R$\footnote{Notons que lorsque $m=0$, on a $\mathcal{Z}_{0,n}=\{\la\in \C| J_n(\la)J_n'(\la)=0\}$. Il est
 bien connu
que les z\'eros des fonctions de Bessel d'ordre entier et leurs d\'eriv\'ees sont r\'eels. Notre r\'esultat
\ref{vecteurpropreL} 4. peut \^etre vu comme une g\'en\'eralisation: Les z\'eros de $L_{m,n}$ sont r\'eels.} (voir
th\'eor\`eme \ref{vecteurpropreL}). Le r\'esultat majeur de cette section est le th\'eor\`eme suivant,
  dans lequel on montre que les fonctions $L_{m,n}$ g\'en\`erent des vecteurs propres pour le Laplacien $\Delta_{\overline{\mathcal{O}(m)}_\infty}$:
\begin{theorem}[voir th\'eor\`eme \ref{laplacevect1}]
Soit $m\in \N$, on a pour tout $n\in \Z$,  $\la\in \mathcal{Z}_{m,n}$ et $k\in \{0,1,\ldots,m\}$:

\[
\bigl(1\otimes z^k, \Delta_{\overline{\mathcal{O}(m)}_\infty} \xi\bigr)_{L^2,\infty}=0,
\]
\[
\bigl(\vf_{n,\la}^{(m)}, \Delta_{\overline{\mathcal{O}(m)}_\infty} \xi\bigr)_{L^2,\infty}=\frac{\la^2}{4}\bigl(\vf_{n,\la}^{(m)}, \xi\bigr)_{L^2,\infty},
\]
et
\[
\bigl(\vf_{{}_{-n+m,\la}}^{(m)}, \Delta_{\overline{\mathcal{O}(m)}_\infty} \xi\bigr)_{L^2,\infty}=\frac{\la^2}{4}\bigl(\vf_{{}_{-n+m,\la}}^{(m)}, \xi\bigr)_{L^2,\infty},
\]
pour tout $\xi\in A^{(0,0)}(\p^1,\mathcal{O}(m))$. En particulier,
 \[
\Bigl\{0\Bigr\}\bigcup \Bigl\{\frac{\la^2}{4} \Bigl|\, \exists n\in \N,\, \la\in \mathcal{Z}_{m,n}  \Bigr\}\subset \mathrm{Spec}(\Delta_{\overline{\mathcal{O}(m)}_\infty}).
\]

\end{theorem}

Nous terminons cette section par une formule reliant la norme $L^2$ de $\vf_{n,\la}^{(m)}$ et la d\'eriv\'ee
de $L_{m,n}$ en $\la$ (voir th\'eor\`eme \ref{deriveeLnorme}). Ce r\'esultat est crucial pour la suite (voir page
31).\\

La section \ref{V1} constitue
 le noyau  de ce travail, \`a savoir que
 l'op\'erateur $\Delta_{\overline{\mathcal{O}(m)}_\infty}$ admet un spectre discret, positif et
 infini, et qu'il est compl\'etment d\'etermin\'e par les ensembles $\mathcal{Z}_{m,n}, \, n\in \Z$. Plus
 concr\'etement, nous \'etabissons dans cette section le th\'eor\`eme:
 \begin{theorem}[cf. Th\'eor\`eme \ref{x333}]
Pour tout $m\in\N$, $\Delta_{\overline{\mathcal{O}(m)}_\infty}$ admet un spectre discret, positif et infini, et on a
{{} \[
\mathrm{Spec}(\Delta_{\overline{\mathcal{O}(m)}_\infty})=\Bigl\{0\Bigr\}\bigcup \Bigl\{\frac{\la^2}{4} \Bigl|\, \exists n\in \N,\, \la\in \mathcal{Z}_{m,n}  \Bigr\},
\]}
En plus, on a
\begin{enumerate}
\item Lorsque  $m$ est pair, alors  la multiplicit\'e de $\frac{\la^2}{4}$ est \'egale  \`a $2$ si $\la\in
    \mathcal{Z}_{m,n}$ si $n\geq m+1 $ ou $0\leq n\leq \frac{m}{2}-1$,
     et de multiplicit\'e $1$ si $\la\in \mathcal{Z}_{m,\frac{m}{2}}$.
\item Si $m$ est impair, alors $\frac{\la^2}{4}$ est de multiplicit\'e $2$ si $n\geq m+1$, et vaut
$1$ si $0\leq n\leq m $.
\end{enumerate}
\end{theorem}

 Ce th\'eor\`me sera une cons\'equence du r\'esultat suivant:
 \begin{theorem}[voir th\'eor\`eme \ref{x16}]
 \begin{equation}
\Bigl\{1\otimes 1,1\otimes z,\ldots,1\otimes z^m\Bigr\}\bigcup\Bigl\{\vf_{\nu,\la}^{(m)}\,\Bigl|\, \nu\in\Z,\, \la\in \mathcal{Z}_{m,\nu}\Bigr\},
\end{equation}
est une base hilbertienne pour   $\mathcal{H}^{(m)}$.
\end{theorem}
Ce th\'eor\`eme affirme que les vecteurs propres calcul\'es
th\'eor\`eme   \ref{laplacevect1} sont les uniques vecteurs propres possibles. Le th\'eor\`eme \ref{x15} traite
le cas $m=0$, nous  montrons qu'il   d\'ecoule  de la th\'eorie classique des s\'eries de Fourier-Bessel.\\

\noindent \textbf{\textsc{La th\'eorie des s\'eries de Fourier-Bessel g\'en\'eralis\'ee}} Dans \S \ref{z9},
nous d\'eveloppons une th\'eorie des s\'eries de Fourier-Bessel g\'en\'eralis\'ee qui \'etend la th\'eorie
classique. Comme application, nous \'etablissons les r\'esultats cit\'es plus haut.

 Soit $E$ le sous-espace vectoriel des fonctions complexes sur $\R^+$ d\'efini comme suit: $f\in E$ si les fonctions
suivantes
$x\in [0,1]\mapsto f(x)$ et $x\in ]0,1]\mapsto f(\frac{1}{x})$ sont continues et born\'ees. Soit
$m\in \N$. On munit $E$  de la norme
scalaire $\vc_m$
suivante:
\[
\|f\|^2_m:=\int_0^1 |f(x)|^2xdx+\int_1^\infty |f(x)|^2\frac{dx}{x^{3+2m}}\quad \forall\, f\in E,
\]
et nous notons  par $\mathcal{E}_m$ la compl\'etion de $E$ pour cette norme.  Nous montrons par exemple que $\mathcal{E}_0
$ est isom\'etrique \`a  $L^2([0,1])\oplus L^2([0,1])$ \footnote{$L^2([0,1])$ d\'esigne
 l'espace des fonctions complexes carr\'e int\'egrables pour la mesure $x\,dx$ sur $[0,1]$} (Cet espace
joue un
rôle important dans la th\'eorie des s\'eries de Fourier-Bessel) et que la suite $(\mathcal{E}_m)_{m\in \N}$ forme une
suite strictement croissante pour l'inclusion (voir proposition \ref{lespaceE}). Nous expliquons dans \ref{z12}, le
lien entre $\mathcal{E}_m$ et $\mathcal{H}^{(m)}$.

Nous \'etablissons que pour tout $m\in \N$, il existe une th\'eorie des s\'eries de Fourier-Bessel g\'en\'eralis\'ee
sur $\mathcal{E}_m$. Par d\'efinition, la th\'eorie classique des s\'eries de Fourier-Bessel
est associ\'ee \`a $\mathcal{E}_0$ (voir la remarque \ref{z13}). Avec les notations de \S\ref{z9},
 le th\'eor\`eme \ref{x17} \'etablit l'existence d'une telle th\'eorie:
\begin{theorem}[voir th\'eor\`eme \ref{x17}]
La famille suivante:
\[
\bigl\{ {\boldsymbol\vf}_{\nu,k}^{(m)}|\, k\in \N \bigr\}\quad\bigl(\text{resp.}\,\bigl\{1,x,\ldots,x^m  \bigr\}\cup\bigl\{ {\boldsymbol\vf}_{\nu,k}^{(m)}|\, k\in \N \bigr\}\bigr),
\]
est une base hilbertienne dans $\mathcal{E}_m$ lorsque $\nu\leq -1$ ou $\nu\geq m+1$ (resp. si $\nu \in \{0,1,\ldots,m\}$).
\end{theorem}
La preuve de ce r\'esultat suit la d\'emarche dans \cite[\S 18]{Watson}, dont un aper\c{c}u est donn\'e au \S
\ref{z10}.
 Nous  adaptons  les techniques de la th\'eorie des s\'eries de
Fourier-Bessel  (voir \cite[\S.18]{Watson}) \`a notre situation,  par exemple, les
 fonctions $L_{m,n}$  joueront ici le rôle des fonctions de Bessel.

 En utilisant le lien existant entre $\mathcal{E}_m$ et $\mathcal{H}^{(m)}$ (remarque \ref{z12}), nous
 d\'eduisons le th\'eor\`eme \ref{x16} et par suite le th\'eor\`eme \ref{x333}.\\

\noindent\textbf{Remerciements}: Cet article est une partie de ma th\`ese (voir \cite{these})
 sous la direction de Vincent Maillot. Je le remercie pour
ses conseils et son aide  lors de la pr\'eparation de ce travail. Je tiens  aussi \`a remercier G\'erard Freixas pour  les longues
discussions math\'ematiques qu'on a eu autour de ce travail.

\tableofcontents

\section{Rappel sur la th\'eorie classique du Laplacien}\label{x29}

Dans cette section, on rappelle  quelques \'el\'ements de la th\'eorie classique des op\'erateurs Laplaciens (voir  \cite{heat}).

Soit $(X,h_X)$ une vari\'et\'e k\"ahl\'erienne compacte de dimension $n$ et $\overline{E}:=(E,h_E)$ un fibr\'e  hermitien muni
d'une m\'etrique de classe $\cl$ sur $X$. Pour $p,q=0,1,\ldots,n$, on note par $\Omega^{(p,q)}(X,E)$ (resp.
$A^{(p,q)}(X,E)$) le fibr\'e vectoriel des $(p,q)-$formes de classe $\cl$ \`a coefficients dans $E$ (resp. l'espace des
sections globales de $\Omega^{(p,q)}(X,E)$). La m\'etrique de $h_X$ sur $T^{(1,0)}X$ induit par dualit\'e une  m\'etrique
hermitienne $\Omega^{(1,0)}X$ et par conjugaison une m\'etrique sur  $\Omega^{(0,1)}(X)$.  En prenant les diff\'erents
produits ext\'erieurs de cette m\'etrique et en les tensorisant avec la m\'etrique de $E$,   on construit un produit
hermitien ponctuel $\bigl(s(x),t(x)\bigr)$ pour deux sections $s$ et $t$
 de
$A^{(0,q)}(X,E)=A^{(0,q)}(X)\otimes_{\mathcal{C}^\infty(X)}A^{(0,0)}(X,E)$. Soit $\omega_0$ la forme de K\"ahler
normalis\'ee, donn\'ee dans chaque carte locale $(z_\al)$ sur $X$ par
\[
 \omega_0=\frac{i}{2\pi}\sum_{\al,\beta}h\Bigl(\frac{\pt}{\pt z_\al},\frac{\pt}{\pt z_\beta}\Bigr)dz_\al\wedge d\z_\beta.
\]
Le produit scalaire $L^2$ de deux sections $s,t\in A^{(0,q)}(X,E)$ est d\'efini par la formule suivante:
\begin{equation}\label{z1}
 (s,t)_{L^2}=\int_X(s(x),t(x))\frac{\omega_0^n}{n!}.
\end{equation}
On dispose du complexe suivant appel\'e complexe de Dolbeault:
\[
0\longrightarrow A^{(0,0)}(X,E)\overset{\overline{\pt}_E^0}{\longrightarrow} A^{(0,1)}(X,E)\overset{\overline{\pt}_E^1}{\longrightarrow}\ldots \overset{\overline{\pt}_E^{n-1}}{\longrightarrow}A^{(0,n)}(X,E)\longrightarrow 0.
\]
L'op\'erateur de Cauchy-Riemann $\overline{\pt}_E$,  admet un adjoint formel
pour la m\'etrique $L^2$, c'est \`a dire un op\'erateur
\[
 \overline{\pt}_E^{\ast}:A^{(0,q+1)}(X,E)\longrightarrow A^{(0,q)}(X,E)
\]
qui v\'erifie
\[
 (s,\overline{\pt}_E^{\ast}t)_{L^2}= (\overline{\pt}_Es,t)_{L^2}
\]
pour tout $s\in A^{(0,q)}(X,E)$ et $t\in A^{(0,q+1)}(X,E)$.

Pour  $q=0,\ldots,n$,  l'op\'erateur
\begin{equation}\label{x28}
\Delta_{\overline{E}}^q=\overline{\pt}_E^{q-1}\overline{\pt}_E^{q\ast}+\overline{\pt}_E^{q+1\ast}\overline{\pt}_E^{q},
\end{equation}
agissant sur
 $A^{(0,q)}(X,E)$ est appel\'e l'op\'erateur Laplacien. On sait (voir par exemple \cite{heat}, \cite{Ma}) que
$\Delta_{\overline{E}}^q$ est un op\'erateur positif, autoadjoint et
 qu'il admet un spectre  positif, discret et infini. En plus,
  les vecteurs propres associ\'es sont dans $A^{0,q}(X,E)$.


\section{Le  Laplacien canonique}\label{x30}
Dans ce paragraphe nous introduisons la notion du Laplacien canonique associ\'e \`a des m\'etriques
canoniques sur $\p^1$. Nous montrerons qu'il admet une extension maximale autoadjointe, ce r\'esultat sera une
 combinaison des th\'eor\`emes \ref{laplacevect1}, \ref{x16}  et le lemme \ref{extensionMaximale}.
  Cette notion \'etend celle du Laplacien associ\'e aux m\'etriques de classe $\cl$.\\

Soit $(\p^1,\omega_\infty)$ l'espace projectif complexe de dimensionn $1$, muni de la forme K\"ahler suivante
$\omega_\infty$ donn\'ee sur $\C$ par:
 \[
\omega_\infty=\frac{i}{2\pi } \frac{dz\wedge d\z}{\max(1,|z|^2)^2}.
\]
C'est la m\'etrique induite par la m\'etrique canonique de $\mathcal{O}(2)$ via l'isomorphisme  $\mathrm{T}\p^1 \simeq
\mathcal{O}(2)$. On l'appellera la forme de K\"ahler canonique de $\p^1$. On note par
$\overline{\mathcal{O}(m)}_\infty=(\mathcal{O}(m),h_{{\overline{\mathcal{O}(m)}_\infty}})$ muni
de sa m\'etrique canonique.

Cette donn\'ee induit un produit hermitien $L^2_\infty$
 sur $A^{(0,0)}(\p^1,\mathcal{O}(m))$ d\'efinie
comme suit:
\[
\bigl(\xi,\xi'\bigr)_{L^2,\infty}=\int_{\p^1}\,h_{{\overline{\mathcal{O}(m)}_\infty}}(\xi,\xi')\omega_\infty,
\]
pour tous $\xi,\xi'\in A^{(0,0)}(\p^1,\mathcal{O}(m))$. On note
par $\mathcal{H}^{(m)}$ le complet\'e de $A^{(0,0)}(\p^1,\mathcal{O}(m))$ par rapport \`a ce produit hermitien.\\

Nous allons  d\'efinir un op\'erateur lin\'eaire  $\Delta_{\overline{\mathcal{O}(m)}_\infty}$
d\'efini sur $A^{(0,0)}(\p^1,\mathcal{O}(m))$ \`a valeurs dans
$\mathcal{H}^{(m)}$ associ\'e aux m\'etriques canoniques, et qui \'etend la notion classique du
Laplacien.\\

Soit $\xi\in A^{(0,0)}(\p^1,\mathcal{O}(m))$. Comme $\mathcal{O}(m)$ est engendr\'e par ses sections globales, et
par une partition de l'unit\'e, alors  $\xi$ peut s'\'ecrire comme
 une combinaison lin\'eaire d'\'el\'ements de la forme $f\otimes z^k$
o\`u $f\in  A^{(0,0)}(\p^1)$ et $z^k$ est
 le  polyn\^ome monomial de degr\'e $k\leq m$. On commence d'abord par poser $\Delta_{\overline{\mathcal{O}(m)}_\infty}(f\otimes
z^k)$, la fonction \`a coefficients dans $\mathcal{O}(m)$  d\'efinie sur $\p^1\setminus \s$ comme suit:
\begin{equation}\label{z3}
\Delta_{\overline{\mathcal{O}(m)}_\infty}(f\otimes
z^k)=-h_{{}_{\overline{T\p^1}_\infty}}\bigl(\dif,\dif\bigr)^{-1}h_{{}_{\overline{\mathcal{O}(m)}_\infty}}\bigl(
z^k,z^k\bigr)^{-1}\frac{\pt}{\pt z}\Bigl(h_{{}_{\overline{\mathcal{O}(m)}_\infty}}(z^k,z^k)\frac{\pt f}{\pt \z}
\Bigr)\otimes z^k\footnote{Notons que
$h_{{\overline{\mathcal{O}(m)}_\infty}}(z^k,z^k)(z)=\frac{|z|^{2k}}{\max(1,|z|)^{2m}}$, $\forall z\in\C $, donc
  $\cl$ par morceaux et ses d\'eriv\'ees sont born\'ees localement.}\footnote{Lorsque les m\'etriques sont de classe $\cl$
  alors on v\'erifie que le Laplacien (voir \ref{x28}) poss\`ede la m\^eme expression locale. }
\end{equation}
On \'etend par lin\'earit\'e cette d\'efinition \`a tout \'el\'ement de $A^{(0,0)}(\p^1,\mathcal{O}(m))$.

Cet \'el\'ement repr\'esente un \'el\'ement de $\mathcal{H}^{(m)}$. En d'autres termes, $\Delta_{\overline{\mathcal{O}(m)}_\infty}(f\otimes
z^k)\in \mathcal{H}^{(m)}$. En effet, on a
\begin{align*}
 \bigl\|\Delta_{\overline{\mathcal{O}(m)}_\infty}(f\otimes
z^k)\bigr\|^2_{L^2,\infty}&=\frac{i}{2\pi}\int_{\p^1\setminus \s}h_{{}_{\overline{T\p^1}_\infty}}\bigl(\dif,\dif\bigr)^{-1}h_{{}_{\overline{\mathcal{O}(m)}_\infty}}\bigl(
z^k,z^k\bigr)^{-1}\Bigl|\frac{\pt}{\pt z}\Bigl(h_{{}_{\overline{\mathcal{O}(m)}_\infty}}(z^k,z^k)\frac{\pt f}{\pt \z}
\Bigr)\Bigr|^2dz\wedge d\z\\
&=\frac{i}{2\pi}\int_{\p^1\setminus \s}\max(1,|z|^4)\max(|z|^{-2k},|z|^{2m-2k})\Bigl|\frac{\pt}{\pt z}\Bigl(\frac{|z|^{2k}}{\max(1,|z|^{2m})}\frac{\pt f}{\pt \z}
\Bigr)\Bigr|^2dz\wedge d\z\\
&=\frac{i}{2\pi}\int_{|z|<1}|z|^{-2k}\Bigl|\frac{\pt}{\pt z}\Bigl(|z|^{2k}\frac{\pt f}{\pt
\z} \Bigr)\Bigr|^2dz\wedge d\z\\
&+\frac{i}{2\pi}\int_{|z|>1}|z|^{4+2m-2k}\Bigl|\frac{\pt}{\pt z}\Bigl(|z|^{2k-2m}\frac{\pt f}{\pt \z}
\Bigr)\Bigr|^2dz\wedge d\z,
\end{align*}
donc,
\begin{equation}\label{x2}
\bigl\|\Delta_{\overline{\mathcal{O}(m)}_\infty}(f\otimes
z^k)\bigr\|^2_{L^2,\infty}=\frac{i}{2\pi}\int_{|z|<1}|z|^{-2k}\Bigl|\frac{\pt}{\pt z}\Bigl(|z|^{2k}\frac{\pt f}{\pt
\z} \Bigr)\Bigr|^2dz\wedge d\z
+\frac{i}{2\pi}\int_{|z|>1}|z|^{4+2m-2k}\Bigl|\frac{\pt}{\pt z}\Bigl(|z|^{2k-2m}\frac{\pt f}{\pt \z}
\Bigr)\Bigr|^2dz\wedge d\z,
\end{equation}
et puisque $f$ est $A^{(0,0)}(\p^1)$, alors on v\'erifie que $\Bigl|\frac{\pt}{\pt z}\Bigl(|z|^{2k}\frac{\pt f}{\pt
\z} \Bigr)\Bigr|^2= O(|z|^{2k-1}) $ au voisinage de $z=0$, donc
\[
 \int_{|z|<1}|z|^{-2k}\Bigl|\frac{\pt}{\pt z}\Bigl(|z|^{2k}\frac{\pt f}{\pt
\z} \Bigr)\Bigr|^2dz\wedge d\z\leq \frac{i}{2\pi}\int_{|z|<1}O(|z|^{-1})dz\wedge d\z<\infty,
\]
En faisant le changement de variables $y=\frac{1}{z}$ dans la deuxi\`eme int\'egrale, on conclut avec le m\^eme
raisonnement que
\[
 \frac{i}{2\pi}\int_{|z|>1}|z|^{4+2m-2k}\Bigl|\frac{\pt}{\pt z}\Bigl(|z|^{2k-2m}\frac{\pt f}{\pt \z}
\Bigr)\Bigr|^2dz\wedge d\z<\infty.
\]
Donc,
\begin{equation}\label{z8}
 \Delta_{\overline{\mathcal{O}(m)}_\infty}(f\otimes
z^k)\in \mathcal{H}^{(m)}\quad \forall\,f\in A^{(0,0)}(\p^1),\,\forall\,k=0,\ldots,m.
\end{equation}

Si l'on consid\`ere $g \in A^{(0,0)}(\p^1)$ et $l\in \{0,1,\ldots,m\}$. On a,
\begin{align*}
(f\otimes z^k,\Delta_{\overline{\mathcal{O}(m)}_\infty}(g\otimes z^l))_{L^2,\infty}&=\int_{\p^1}
h_{\overline{\mathcal{O}(m)}_\infty}(z^l,z^l)^{-1}h_{\overline{\mathcal{O}(m)}_\infty}(z^k,z^l) f\frac{\pt}{\pt
\z}\Bigl(h_{{}_{\overline{\mathcal{O}(m)}_\infty}}(z^l,z^l)\frac{\pt \overline{g}}{\pt z}\Bigr)\omega_\infty\\
&=\frac{i}{2\pi}\int_{|z|\leq 1}\frac{z^k}{z^l}f\frac{\pt}{\pt
\z}\Bigl(|z|^{2l}\frac{\pt \overline{g}}{\pt z}\Bigr)dz\wedge \z
+\frac{i}{2\pi}\int_{|z|\geq 1}\frac{z^k}{z^l}f\frac{\pt}{\pt
\z}\Bigl(\frac{|z|^{2l}}{|z|^{2m}}\frac{\pt \overline{g}}{\pt z}\Bigr)\frac{dz\wedge \z}{|z|^4}\\
&=\frac{i}{2\pi}\int_{|z|\leq 1}f\frac{\pt}{\pt
\z}\Bigl(z^k\z^l\frac{\pt \overline{g}}{\pt z}\Bigr)dz\wedge \z
+\frac{i}{2\pi}\int_{|z|\geq 1}f\frac{\pt}{\pt
\z}\Bigl(\frac{z^k \z^l}{|z|^{2m}}\frac{\pt \overline{g}}{\pt z}\Bigr)\frac{dz\wedge \z}{|z|^4}.
\end{align*}
En utilisant la formule de Stokes, on peut d\'eduire que
\begin{equation}\label{z6}
(f\otimes z^k,\Delta_{\overline{\mathcal{O}(m)}_\infty}(g\otimes z^l))_{L^2,\infty}=\frac{i}{2\pi}
\int_{|z|\leq 1}z^k
\z^l \frac{\pt f}{\pt
\z} \frac{\pt \overline{g}}{\pt z}dz\wedge \z
+\frac{i}{2\pi}\int_{|z|\geq 1}\frac{z^k \z^l}{|z|^{2m}} \frac{\pt f}{\pt
\z}\frac{\pt \overline{g}}{\pt z}\frac{dz\wedge \z}{|z|^4}=\int_{\p^1}h_{\overline{\mathcal{O}(m)}_\infty
}(\frac{\pt f}{\pt
\z}z^k,\frac{\pt g}{\pt \z}z^l )\omega_\infty.
\end{equation}
Par cons\'equent,
\begin{equation}\label{z5}
(f\otimes z^k,\Delta_{\overline{\mathcal{O}(m)}_\infty}(g\otimes
z^l))_{L^2,\infty}=(\Delta_{\overline{\mathcal{O}(m)}_\infty}(f\otimes z^k),g\otimes z^l)_{L^2,\infty}.
\end{equation}

\begin{defprop}\label{x31}
On a $\Delta_{\overline{\mathcal{O}(m)}_\infty}:A^{(0,0)}(\p^1,\mathcal{O}(m))\lra \mathcal{H}^{(m)}$
est un op\'erateur lin\'eaire, on l'appelle le Laplacien canonique associ\'e \`a $\omega_\infty$ et \`a
$h_{\overline{\mathcal{O}(m)}_\infty}$. Soient $\xi,\eta\in A^{(0,0)}(\p^1,\mathcal{O}(m))$, on a:
\begin{enumerate}
\item
\[\ker \Delta_{\overline{\mathcal{O}(m)}_\infty}= H^0(\p^1,\mathcal{O}(m)).\]
\item
\[
\bigl(\xi,\Delta_{\overline{\mathcal{O}(m)}_\infty}\eta\bigr)_{L^2,\infty}=
\bigl(\Delta_{\overline{\mathcal{O}(m)}_\infty}\xi,\eta\bigr)_{L^2,\infty}
\]
\item
\[
\bigl(\xi,\Delta_{\overline{\mathcal{O}(m)}_\infty}\xi\bigr)_{L^2,\infty}\geq
0
\]
\end{enumerate}

\end{defprop}
\begin{proof}
Soit $\xi\in A^{(0,0)}(\p^1,\mathcal{O}(m))$. Il existe $f_0,f_1,\ldots,f_m\in A^{0,0}(\p^1)$
    tels que $\xi=\sum_{k=1}^m f_k\otimes z^k$. On v\'erifie facilement que sur $\p^1\setminus \s$,
    \[
    \Delta_{\overline{\mathcal{O}(m)}_\infty}\xi=\sum_{k=0}^m
  \Delta_{\overline{\mathcal{O}(m)}_\infty}(f_k\otimes z^k).
    \]
    D'apr\`es  \ref{z8}, on conclut que
\begin{equation}\label{Dfsi}
 \Delta_{\overline{\mathcal{O}(m)}_\infty}\xi\in \mathcal{H}^{(m)}.
\end{equation}
    En utilisant \ref{z6}, nous avons
    \begin{equation}\label{z7}
    \bigl(\xi,\Delta_{\overline{\mathcal{O}(m)}_\infty}\xi\bigr)_{L^2,\infty}=\sum_{0\leq k,j\leq m}
    \bigl(f_k\otimes z^k,\Delta_{\overline{\mathcal{O}(m)}_\infty}(f_j\otimes z^j)\bigr)_{L^2,\infty}
    =\sum_{0\leq k,j\leq m} \int_{\p^1}h_{\overline{\mathcal{O}(m)}_\infty
}(\frac{\pt f_k}{\pt
\z} z^k,\frac{\pt f_j}{\pt \z} z^j)\omega_\infty.
    \end{equation}

\leavevmode
\begin{enumerate}
\item Si $\xi\in \ker \Delta_{\overline{\mathcal{O}(m)}_\infty}$. Donc, par ce qui pr\'ec\`ede, et par continuit\'e de
la m\'etrique $h_{\overline{\mathcal{O}(m)}_\infty
}$:
\[
0=\sum_{0\leq k,j\leq m} h_{\overline{\mathcal{O}(m)}_\infty
}(\frac{\pt f_k}{\pt
\z}z^k,\frac{\pt f_j}{\pt \z} z^l)=h_{\overline{\mathcal{O}(m)}_\infty
}(\sum_{k=0}^m \frac{\pt f_k}{\pt
\z}z^k,\sum_{k=0}^m \frac{\pt f_k}{\pt
\z}z^k).
\]
    Donc, $\sum_{k=0}^m \frac{\pt f_k}{\pt
\z}z^k=0$ pour tout $z\in \C$. Par cons\'equent, $z\mapsto \sum_{k=0}^mf_k z^k $ est holomorphe sur $\C$. Comme les fonctions $f_k$ sont born\'ees sur $\p^1$, alors par un argument classique d'analyse complexe,
il existe $a_0,\ldots,a_m$ des constantes telles que
 $\sum_{k=0}^m f_k z^k=\sum_{k=0}^m a_k z^k, \forall\, z\in \C$. On v\'erifie que
 $\|\xi-\sum_{k=0}^m a_k\otimes z^k \|_{L^2,\infty}=0$. On conclut que
 \[
 \ker \Delta_{\overline{\mathcal{O}(m)}_\infty}= H^0(\p^1,\mathcal{O}(m)).
 \]
\item Par lin\'earit\'e, cela r\'esulte du \ref{z5}.
\item C'est une cons\'equence imm\'ediate de \ref{z7}.
\end{enumerate}
\end{proof}


On se propose dans la suite d'\'etablir que l'op\'erateur $\Delta_{{\overline{\mathcal{O}(m)}}_\infty}$ admet
un spectre discret, infini et positif.  Nous allons  adopter la d\'efinition fonctionnelle pour parler d'un
vecteur propre, c-\`a-d que $\al $ est une valeur propre de  $\Delta_{{\overline{\mathcal{O}(m)}}_\infty}$ s'il
existe $\vf\in \mathcal{H}^{(m)}$,  tel que
\begin{equation}\label{faible}
 (\vf,\Delta_{{\overline{\mathcal{O}(m)}}_\infty} \xi)_{L^2,\infty}=\al(\vf,\xi)_{L^2,\infty},\quad \forall\, \xi\in
A^{0,0}(\p^1,\mathcal{O}(m)).
\end{equation}
On note par
$\mathrm{Spec}(\Delta_{{\overline{\mathcal{O}(m)}}_\infty})$ l'ensemble des valeurs propres.

\begin{theorem}\label{z14}
Pour tout $m\in \N$, l'op\'erateur $\Delta_{{\overline{\mathcal{O}(m)}}_\infty}$ admet une extension maximale
auto-adjointe et
positive. Si l'on note par $\mathrm{Dom}(\Delta_{{\overline{\mathcal{O}(m)}}_\infty})$ le domaine de
$\Delta_{{\overline{\mathcal{O}(m)}}_\infty}$, alors
\[
\mathrm{Dom}(\Delta_{{\overline{\mathcal{O}(m)}}_\infty})=\Bigl\{\xi\in \mathcal{H}^{(m)}|\,\sum_{n\in \Z,\la\in
\mathcal{Z}_{m,n}}\frac{\la^4}{16}\bigl|\bigl(\xi,\vf_{n,\la}^{(m)}\bigr)\bigr|^2<\infty \Bigr\}.
\]
\end{theorem}

\begin{proof} Montrons que $A^{0,0}(\p^1,\mathcal{O}(m))\subset
\mathrm{Dom}(\Delta_{{\overline{\mathcal{O}(m)}}_\infty})$. Soit $\xi\in A^{0,0}(\p^1,\mathcal{O}(m))$. D'apr\`es
 le th\'eor\`eme \ref{x16}, il existe $a_0,\ldots,a_m\in \C$ et $(a_{n,\la})_{n\in \Z,\la\in \mathcal{Z}_{m,n}}$
 une suite de nombres complexes telles que $\xi=\sum_{k=0}^m a_k (1\otimes z^k)+\sum_{n\in \Z,\la\in
  \mathcal{Z}_{m,n}} a_{n,\la} \vf_{n,\la}^{(m)}$ dans $\mathcal{H}^{(m)}$. Or, on a montr\'e dans \ref{x31}, que
  $\Delta_{{\overline{\mathcal{O}(m)}}_\infty}\xi\in \mathcal{H}^{(m)}$. Il existe alors  $b_0,\ldots,b_m\in \C$ et $(b_{n,\la})_{n\in \Z,\la\in \mathcal{Z}_{m,n}}$
 une suite de nombres complexes telles que $\Delta_{{\overline{\mathcal{O}(m)}}_\infty}
 \xi=\sum_{k=0}^m b_k (1\otimes z^k)+\sum_{n\in \Z,\la\in
  \mathcal{Z}_{m,n}} b_{n,\la} \vf_{n,\la}^{(m)}$ dans $\mathcal{H}^{(m)}$. En utilisant le th\'eor\`eme
  \ref{laplacevect1}, on obtient que $b_0=\cdots=b_m=0$ et $ b_{n,\la}=\frac{\la^2}{4}  a_{n,\la}$ pour tout
  $n\in \Z$ et $\la\in \mathcal{Z}_{m,n}$. Par cons\'equent, $\xi\in
  \mathrm{Dom}(\Delta_{{\overline{\mathcal{O}(m)}}_\infty})$.

  On applique le  lemme \ref{extensionMaximale} pour conclure.
\end{proof}

\section{Sur le Spectre de $\Delta_{\overline{\mathcal{O}(m)}_\infty}$  (I) }\label{D1}

Cette section constitue la premi\`ere partie dans notre \'etude de l'op\'erateur  $\Delta_{\overline{\mathcal{O}(m)}_\infty}$.
Nous introduisons une  famille de fonctions  $(L_{m,n})_{n\in \Z}$ et nous  expliquerons son  lien
avec l'\'etude de l'op\'erateur $\Delta_{\overline{\mathcal{O}(m)}_\infty}$. \`A cette famille, nous associons une
 famille
 d'\'el\'ements de $\mathcal{H}^{(m)}$ orthogonale pour le produit $(,)_{L^2,\infty}$ (voir th\'eor\`eme \ref{vecteurpropreL}). Le r\'esultat majeur de cette section \'etablit que cette famille  g\'en\`ere
 une partie du spectre de  $\Delta_{\overline{\mathcal{O}(m)}_\infty}$ (th\'eor\`eme \ref{laplacevect1}).\\

Soit $m$ un entier positif.
\begin{definition}
 Pour tout $\nu\in \Z$, on consid\`ere $L_{m,\nu}$, la fonction d\'efinie par:
\[
 L_{m,\nu}(z)=-z^m\frac{d}{dz}\bigl(z^{-m} J_\nu(z)J_{\nu-m}(z) \bigr),\quad \forall\, z\in \C,
\]
et on pose
\begin{equation}\label{Zeros1}
\mathcal{Z}_{m,\nu}:=\Bigl\{\lambda\in \C^\ast \,\Bigl|\, L_{m,\nu}(\la)=0 \Bigr\}\footnote{Lorsque $m=0$, on a
$\mathcal{Z}_{0,\nu}$ est form\'e par les z\'eros de $J_\nu$ et ceux de $J_\nu'$.}.
\end{equation}

\end{definition}

 Notons que pour $m=0$,  $L_{0,\nu}=-2 J_\nu J_\nu'$.\\

Soit $\nu\in \Z$.
 L'ensemble des z\'eros de $J_\nu$, et plus g\'en\'eralement ceux de $J_\nu+H J_\nu'$ (avec $H$ une constante r\'eelle)
  a \'et\'e \'etudi\'e de mani\`ere intense dans la litt\'erature (voir par exemple \cite{Watson}). On montre par exemple,
  que ces z\'eros sont simples et r\'eels, et que $\{r\mapsto J_\nu(\la r)|\,J_\nu(\la)=0\}$ forme un sys\`eme orthogonal dans
  $L^2([0,1], rdr)$ (voir \cite[\S 15]{Watson}).

Notre but ici  est d'\'etendre
cette th\'eorie \`a cette nouvelle classe de fonctions $L_{m,\nu}$ pour tout $m\geq 1$ et $\nu\in \Z$.
 Par exemple,  on va \'etablir que les z\'eros de $L_{m,\nu}$ sont r\'eels et simples (voir 4. th\'eor\`eme
 \ref{vecteurpropreL} et th\'eor\`eme \ref{deriveeLnorme}).

 Nous  commençons d'abord par \'etudier  les propri\'et\'es de
$L_{m,\nu}$  et
de $\mathcal{Z}_{m,\nu}$:
\begin{lemma}\label{zerosimple}
 $\forall\, m\in \N$,  $\forall\,\nu\in\Z$, $L_{m,\nu}$ est une fonction analytique sur $\C$ entier qui v\'erifie les propri\'et\'es suivantes:

\begin{enumerate}
\item \begin{equation}\label{Lexpression}
L_{m,\nu}(z)=J_{\nu+1}(z)J_{\nu-m}(z) -J_\nu(z)J_{\nu-m-1}(z),\quad \forall\, z\in \C,
\end{equation}
\item
\begin{equation}\label{Lsymetrie}
 L_{m,-\nu}(z)=(-1)^mL_{m,\nu+m}(z)\quad \forall\, z\in \C,
\end{equation}
\item
\begin{equation}\label{Lparite}
 L_{m,\nu}(-z)=(-1)^{m+1}L_{m,\nu}(z)\quad \forall\, z\in \C,
\end{equation}

\item
On a pour $z$ assez petit,
\begin{equation}\label{L1ordre}
 L_{m,\nu}(z)=-\frac{z^{2\nu-m-1}}{2^{2\nu-m-1}n! (n-m-1)!}+o(z^{2\nu-m-1})\quad \text{si}\,\, \nu\geq m+1,
\end{equation}
et
\begin{equation}\label{L2ordre}
L_{m,\nu}(z)=(-1)^{m+\nu}\frac{m+2}{(m-\nu+1)!(\nu+1)!}z^{m+1}+o(z^{m+1})\quad \text{si}\,\, 0\leq \nu\leq m.
\end{equation}
En particulier, $\mathcal{Z}_{m,\nu}$ est un sous-ensemble infini et discret de $\C^\ast$   et on a

\item
\begin{equation}\label{x3}
\mathcal{Z}_{m,-\nu}=\mathcal{Z}_{m,\nu+m},
\end{equation}
\item Si $m\neq 0$,
 \begin{equation}\label{x4}
 \mathcal{Z}_{m,\nu}\cap  \{\la\in \R|\, J_{\nu-m}(\la)J_\nu(\la)=0 \}=\emptyset.
\end{equation}
\item Si $m\neq 0$
\[
 \mathcal{Z}_{m,\nu}\neq \mathcal{Z}_{0,\nu}.
\]
\end{enumerate}
\end{lemma}
\begin{proof}
En utilisant les relations de r\'ecurrences des fonctions de Bessel, voir appendice \ref{paragraphe},
\begin{enumerate}
\item
nous montrons que pour tout
$z\in\C^\ast$:
\begin{align*}
\frac{d}{dz}\bigl(z^{-m}J_\nu(z)&J_{\nu-m}(z)\bigr)=\frac{d}{dz}\Bigl(z^{-\nu}J_\nu(z)z^{\nu-m}J_{\nu-m}(z)\Bigr)\\
&=\frac{d}{dz}(z^{-\nu}J_\nu(z))z^{\nu-m}J_{\nu-m}(z)+z^{-\nu}J_\nu(z)\frac{d}{dz}(z^{\nu-m}J_{\nu-m}(z))\\
&=-z^{-\nu}J_{\nu+1}(z)z^{\nu-m}J_{\nu-m}(z)+z^{-\nu}J_\nu(z)z^{\nu-m}J_{\nu-m-1}(z)\quad\text{par}\, \ref{B1}\;\text{et}\;\ref{B5}\\
&=-z^{-m}J_{\nu+1}(z)J_{\nu-m}(z)+z^{-m}J_\nu(z)J_{\nu-m-1}(z).
\end{align*}
Par suite,
\[
 L_{m,\nu}(z)=J_{\nu+1}(z)J_{\nu-m}(z)-J_\nu(z)J_{\nu-m-1}(z),\quad \forall\, z\in \C,\, \forall\, \nu\in\Z.
\]
\item
De cette identit\'e, nous d\'eduisons que pour tout $z\in \C$ et $\nu\in\Z$:
\begin{align*}
 L_{m,-\nu}(z)&=J_{(-\nu)+1}(z)J_{(-\nu)-m}(z)-J_{(-\nu)}(z)J_{(-\nu)-m-1}(z)\\
&=(-1)^{m+1}J_{n-1}(z)J_{\nu+m}(z)-(-1)^{m+1}J_\nu(z)J_{\nu+m+1}(z)\quad \text{car}\; J_{-p}=(-1)^p J_p\\
&=(-1)^{m}\Bigl(J_{(\nu+m)+1}(z)J_{(\nu+m)-m}(z)-J_{(\nu+m)}(z)J_{(\nu+m)-m-1}(z)   \Bigr)\\
&=(-1)^m L_{m,\nu+m}(z).
\end{align*}
\item
Comme $J_p(-z)=(-1)^p J_p(z),\;\forall\, z\in \C$ et  $\forall p\in \N$, alors  la fonction $z(\in \C^\ast)\mapsto
z^{-m}J_\nu(z)J_{\nu-m}(z)$ est  paire. On en d\'eduit que
\[
 L_{m,\nu}(-z)=(-1)^{m+1} L_{m,\nu}(z)\quad \forall\, z\in \C,\, \forall\, \nu\in \Z.
\]
\item
 Evaluons maintenant l'ordre de $L_{m,\nu}$ en $z=0$ en fonction de $\nu\in \N$. D'apr\`es l'appendice \ref{paragraphe}, on a lorsque $p\geq 0$,$
 J_p(z)=\frac{1}{2^p p!}z^p+o(z^{p})$ pour tout $z$ assez petit.
\begin{enumerate}
\item Si $\nu \geq m+1$. On peut \'ecrire, en utilisant les identit\'es pr\'ec\'edentes et le d\'eveloppement ci-dessus, que  pour tout
$z$ assez petit:
\begin{align*}
 L_{m,\nu}(z)&=J_{\nu+1}(z)J_{\nu-m}(z)-J_\nu(z)J_{\nu-(m+1)}(z)\\
&=\frac{z^{2\nu-m+1}}{2^{2\nu-m+1}(\nu+1)!(\nu-m)!}-\frac{z^{2\nu-m-1}}{2^{2\nu-m-1}\nu! (\nu-m-1)!}+o(z^{2\nu-m-1})\\
&=-\frac{z^{2\nu-m-1}}{2^{2\nu-m-1}\nu ! (\nu-m-1)!}+o(z^{2\nu-m-1}).
\end{align*}

\item Si $ 0\leq \nu\leq m$. Comme $L_{m,\nu}(z)=J_{\nu+1}(z)J_{\nu-m}(z)-J_\nu(z)J_{\nu-m-1}(z)$, et puisque $J_{-p}=(-1)^pJ_p$ alors pour tout $z\in\C$ et en particulier pour $z$ assez petit, nous avons:
\begin{align*}
 L_{m,\nu}(z)&=(-1)^{m-n}J_{\nu+1}(z)J_{m-n}(z)-(-1)^{m+1-n}J_\nu(z)J_{m+1-n}(z)\\
&=(-1)^{m-n}\bigl( J_\nu(z)J_{m+1-n}(z)-J_{\nu+1}(z)J_{m-n}(z)\bigr)\\
&=(-1)^{m+\nu}\Bigl( \frac{z^{m+1}}{2^{m+1}n! (m+2-\nu)!}+ \frac{z^{m+1}}{2^{m+1}(\nu+1)! (m+1-\nu)!}+o(z^{m+1})   \Bigr)\\
&=(-1)^{m+\nu}\frac{m+2}{(m-\nu+1)!(\nu+1)!}z^{m+1}+o(z^{m+1}).
\end{align*}

\end{enumerate}

Nous d\'eduisons que dans tous les cas, $L_{m,\nu}$ est une fonction analytique avec un z\'ero d'ordre au moins $m+1$ en $z=0$.  Donc
 $ \mathcal{Z}_{m,\nu}$ est discret et  non vide. On sait que les fonctions de Bessel d'ordre sup\'erieur \`a $\frac{1}{2}$ ont une infinit\'e de z\'eros et qui sont en plus tous r\'eels,
voir par exemple \cite[\S 15]{Watson}. Nous pouvons donc affirmer \`a l'aide du th\'eor\`eme de accroissements finis, que la fonction
$z\mapsto \frac{d}{dz}(z^{-m}J_\nu(z)J_{\nu-m}(z))$ restreinte \`a $\R^\ast$, s'annule entre deux z\'eros de $J_\nu$ (resp. de $J_{\nu-m}$). On conclut que  $ \mathcal{Z}_{m,\nu}$ est un sous-ensemble discret et infini de $\C^\ast$.\\
\item
Le point (5.) se d\'eduit du (2.).
\item
Montrons par l'absurde que
\[
 \mathcal{Z}_{m,\nu}\cap \{\la\in \C|\, J_{\nu-m}(\la)=0 \}= \emptyset.
\]
Rappelons que $\{\la\in \C|\, J_{\nu-m}(\la)=0 \}=\{\la\in \R|\, J_{\nu-m}(\la)=0 \}$, puisque  les z\'eros de $J_p$ sont r\'eels lorsque $|p|\in \N$). Soit   $\la$ un \'el\'ement de $ \mathcal{Z}_{m,\nu}\cap \{\la\in \C|\, J_{\nu-m}(\la)=0 \}$, alors on obtient:
\[
 0=L_{m,\nu}(\la)=J_{\nu+1}(\la)J_{\nu-m}(\la)-J_\nu(\la)J_{\nu-m-1}(\la)=-J_\nu(\la)J_{\nu-m-1}(\la),
\]
donc,
\[
 J_\nu(\la)J_{\nu-m-1}(\la)=0.
\]
D'apr\`es \cite[\S 15.22]{Watson},  deux fonction de $\{J_{\nu-m-1},J_{\nu-m}, J_\nu\}$\footnote{Comme on a suppos\'e que $m\geq 1$, alors cet ensemble est de cardinal $3$} n'ont pas de z\'eros non nuls communs, lorsque $\nu>-1$. Donc on a n\'ecessairement
\[
 J_\nu(\la)=0.\footnote{Notons que $|\nu-m-1|\in \N$, donc on peut supposer  que $\nu-m-1\geq -1$. l'autre cas se traite de la m\^eme mani\`ere en utilisant la formule $J_{-p}=(-1)^pJ_p$.}
\]
Montrons que cela est impossible, pour cela on va montrer que si $\nu >1$ alors  les d\'eriv\'ees sup\'erieures de $J_\nu$ s'\'ecrivent en fonction de $J_\nu$ et $J_\nu'$, plus pr\'ecis\'ement on montre facilement par r\'ecurrence que pour tout $k\geq 1$, il existe $P_k$ et $Q_k$ deux fractions rationnelles avec au plus un pôle en z\'ero, telles que
\begin{equation}\label{recurrenceJJJJ}
 J_\nu^{(k)}(z)=P_k(z) J_\nu(z)+Q_k(z) J_\nu'(z),\quad \forall\, z\in \C\setminus\{0\}.
\end{equation}
Pour $k=1$, l'\'egalit\'e pr\'ec\'edente est triviale. Si $k=2$, c'est une cons\'equence imm\'ediate de  l'\'equation de Bessel cf. \ref{besselequation} et on a $P_2(z)=-\bigl(1-\frac{\nu^2}{z^2} \bigr)$ et $Q_2(z)=-\frac{1}{z}$.\\

On sait que
\[
 \frac{d}{dz}\bigl( z^{-\nu}J_\nu(z)\bigr)=-z^{-\nu}J_{\nu+1}(z),\quad\forall\, z\in \C\setminus\{0\},
\]
donc
\[
 \frac{d}{dz}\biggl( z^{-1}\frac{d}{dz}\bigl( z^{-\nu}J_\nu(z)\bigr)\biggr)=\frac{d}{dz}\bigl(z^{-(\nu+1)}J_{\nu+1}(z)\bigr)=z^{-\nu-1}J_{\nu+2}(z),
\]
et on montre par r\'ecurrence que
\[
 \frac{d}{dz}\biggl(z^{-1}\frac{d}{dz}\Bigl(z^{-1}\frac{d}{dz}\cdots \Bigl(z^{-1}\frac{d}{dz}\bigl( z^{-\nu}J_\nu(z)\Bigr) \biggr)=(-1)^{k}z^{-\nu-k+1}J_{\nu+k}(z).
\]
Alors, il existe des fonctions rationnelles $R_0,\ldots,R_k$, ayant un \'eventuel pôle en $z=0$, telles que:
\[
J_{\nu+k}(z)=\sum_{j=0}^k R_j(z) J_\nu^{(j)}(z)\quad \forall\, z\in \C^\ast.
\]
De cette \'egalit\'e et par \ref{recurrenceJJJJ}, on d\'eduit l'existence deux fractions rationnelles $P$ et $Q$ avec un \'eventuel pôle en $z=0$ telles que:
\[
J_{\nu+k}(z)=P(z)J_\nu(z)+Q(z)J_\nu'(z)\quad \forall \, z\in \C^\ast.
\]
Donc si l'on prend $\nu=\nu-m$ et $k=m$, on obtient:
\[
 P(z)J_{\nu-m}(z)+Q(z)J_{\nu-m}'(z)=J_\nu(z), \quad \forall\, z\neq 0.
\]
Or on sait (par hypoth\`ese) que $J_{\nu-m}(\la)=0$ et $J_\nu(\la)=0$ donc
\[
 J_{\nu-m}'(\la)=0,
\]
ce qui impossible car les z\'eros non nuls des fonctions de Bessel sont simples.

Or, $\mathcal{Z}_{m,-\nu+m}=\mathcal{Z}_{m,\nu}$ d'apr\`es \ref{x3}, et comme $J_{-\nu}=J_\nu$, on conclut que
\[
\mathcal{Z}_{m,\nu}\cap  \{\la\in \R|\, J_{\nu-m}(\la)J_\nu(\la)=0 \}=\emptyset.
\]
\item Cela d\'ecoule directement du $(6.)$.
\end{enumerate}
\end{proof}

Dans la suite, on construit \`a partir de la famille $(\mathcal{Z}_{m,n})_{n\in \Z}$ une famille $\bigl\{\vf_{n,\la}\,\bigl|\, n\in \Z,\, \la\in \mathcal{Z}_{m,n} \bigr\}$ d'\'el\'ements de $\mathcal{H}^{(m)}$
  orthogonale pour le produit $L^2_\infty$.
  Soient $n\in \Z$ et  $\la\in \mathcal{Z}_{m,n}$. On consid\`ere  $f_{n,\la}^{(m)}$, la fonction d\'efinie sur $\R^+
  \times \R$
  comme
suit:
\begin{enumerate}
\item Si $m\geq 1$, on pose:
\begin{equation}\label{expressionvecteurpropre}
f_{n,\la}^{(m)}(r,\theta) :=
\begin{cases}
 J_n(\la  r)e^{in\theta} & \text{si } r\in [0,1[,\;  \theta\in \R,\\
\frac{J_n(\la)}{J_{n-m}(\la)}r^{m}J_{n-m}(\frac{\la}{r}) e^{in\theta}& \text{si } r\in  ]1,\infty[,\;
 \theta\in \R, \\
\end{cases}
\end{equation}
(Notons que cette fonction est bien d\'efinie au voisinage $0$, puisque on a
 $J_{|n-m|}(x)=O(x^{|n-m|})$ pour $x$ assez
petit (voir \ref{ordreenzero})).
\item Si $m=0$, on pose:
\begin{enumerate}
\item Si $\la\in \mathcal{Z}_{0,n}$ avec $J_n'(\la)=0$,
\begin{equation}
f_{n,\la}^{(0)}(r,\theta):=
\begin{cases}
 J_n(\la  r)e^{in\theta} & \text{si } r\in [0,1[,\;  \theta\in \R,\\
J_n(\frac{\la}{r}) e^{in\theta}& \text{si } r\in  ]1,\infty[,\; \theta\in \R, \\
\end{cases}
\end{equation}
\item Si $\la\in \mathcal{Z}_{0,n}$ avec $J_n(\la)=0$
\begin{equation}
f_{n,\la}^{(0)}(r,\theta) :=
\begin{cases}
 -J_n(\la  r)e^{in\theta} & \text{si } r\in [0,1[,\; \theta\in \R,\\
J_n(\frac{\la}{r}) e^{in\theta}& \text{si } r\in  ]1,\infty[,\;  \theta\in \R, \\
\end{cases}
\end{equation}
\end{enumerate}
\end{enumerate}
On v\'erifie que $f_{n,\la}^{(m)}$ d\'efinit une fonction continue sur $\C$, qu'on va noter par la m\^eme notation
et on pose
 \begin{equation}\label{z2}
\vf_{n,\la}^{(m)}:=f_{n,\la}^{(m)}\otimes 1.
\end{equation}
o\`u $1$ d\'esigne la section globale $1$ de $\mathcal{O}(m)$.
\begin{theorem}\label{vecteurpropreL}
 $\forall\, m\in \N$, $\forall\, n\in \Z,\;\forall \la\in \mathcal{Z}_{m,n}$. On a
\begin{enumerate}

\item
\[
      \vf_{n,\la}^{(m)}\in \mathcal{H}^{(m)},
     \]
\item
\[
 \bigl \| \vf_{{}_{n,\la}}^{(m)}\bigr\|^2_{L^2,\infty}=\frac{J_{n}(\la)^2}{2}\Bigl(\frac{J'_{n}(\la)^2}{J_{n}(\la)^2}+\frac{J_{n-m}'(\la)^2}{J_{n-m}(\la)^2} \Bigr)+\frac{J_{n}(\la)^2}{2}\Bigl(2-\frac{(n-m)^2+n^2}{\la^2}\Bigr).
\]

\item \[
\bigl(\vf_{{}_{n,\la}}^{(m)},\vf_{{}_{n',\la'}}^{(m)}\bigr)_{L^2,\infty}=\delta_{n,n'}\delta_{\la,\la'}\bigl(
\vf_{{}_{n,\la}}^{(m)},\vf_{{}_{n,\la}}^{(m)}\bigr)_{L^2,\infty},\quad \forall\, n,n'\in \Z,\; \forall\, \la\in \mathcal{Z}_{m,n},\,\la'\in \mathcal{Z}_{m,n'}\footnote{$\delta_{\ast,\ast'}$ par d\'efinition \'egale \`a 0 si $\ast\neq \ast'$, $1$ sinon.}.
\]

\item $\mathcal{Z}_{m,n}$ est un sous-ensemble discret et infini de $\R^\ast$.
\end{enumerate}
\end{theorem}

\begin{proof}
\leavevmode
\begin{enumerate}
\item Comme $f_{n,\la}^{(m)}$ est continue sur $\p^1$.
Alors, par compacit\'e,
 on peut approcher $\vf_{n,\la}^{(m)}$ uniform\'ement
    par une suite d'\'el\'ements de $A^{0,0}(\p^1,\mathcal{O}(m))$. On d\'eduit que $ \vf_{n,\la}^{(m)}\in \mathcal{H}^{(m)}$.

\item
Soit $n \in \Z$ et $\la\in \mathcal{Z}_{m,n}$. Calculons la norme $L^2_\infty$ de $\vf_{n,\la}^{(m)}$:
\begin{align*}
\bigl\| \vf_{n,\la}^{(m)}\bigr\|^2_{L^2,\infty}&=\int_{x\in \p^1}|f_{n,\la}(x)|^2h_{\overline{\mathcal{O}(m)}_\infty}(1,1)(x)\omega_\infty\\
&=\int_{\R^+}|f_{n,\la}(r,\theta)|\frac{1}{\max(1,|r|^{2m})}\frac{rdr}{\max(1,|r|^4)}\\
&=\int_0^1 J_n(\la r)^2 rdr+\frac{J_n(\la)^2}{J_{n-m}(\la)^2}\int_1^\infty J_{n-m}(\frac{\la}{r})^2\frac{rdr}{r^4}\\
&=\int_0^1 J_n(\la r)^2 rdr+\frac{J_n(\la)^2}{J_{n-m}(\la)^2}\int_0^1 J_{n-m}(\la r)^2rdr \\
&=\frac{1}{2}\Bigl(J_n'(\la)^2+(1-\frac{n^2}{\la^2})J_n(\la)^2  \Bigr)+\frac{J_n(\la)^2}{J_{n-m}(\la)^2}\Bigl(J_{n-m}'(\la)^2+(1-\frac{(n-m)^2}{\la^2})J_{n-m}(\la)^2  \Bigr)\quad \text{par}\; \ref{encoreeq}\\
&=\frac{J_{n}(\la)^2}{2}\Bigl(\frac{J'_{n}(\la)^2}{J_{n}(\la)^2}+\frac{J_{n-m}'(\la)^2}{J_{n-m}(\la)^2} \Bigr)+\frac{J_{n}(\la)^2}{2}\Bigl(2-\frac{(n-m)^2+n^2}{\la^2}\Bigr).
\end{align*}
(Notons que le dernier terme est bien d\'efini puisque $J_{n-m}(\la)J_n(\la)\neq 0$ d'apr\`es le
lemme \ref{zerosimple}).

\item
\begin{enumerate}
\item
Supposons d'abord que $m\geq 1$.
Soient $n,n'\in \Z$, $\la\in \mathcal{Z}_{m,n}$ et $\la'\in \mathcal{Z}_{m,n'}$. On a
{\allowdisplaybreaks
\begin{align*}
\bigl(\vf_{n,\la}^{(m)},\vf_{n',\la'}^{(m)}\bigr)_{L^2,\infty}&=\int_{\p^1}f_{{}_{n,\la}}^{(m)}
\overline{f_{{}_{n',\la'}}^{(m)}}
\|1\|^2_{\infty}\omega_\infty\\
&=\delta_{n,n'}\int_{r\leq 1}J_n(\la r)J_n(\la' r)rdr+\delta_{n,n'}\int_{r\geq 1}\frac{J_n(\la) J_n(\la')}{J_{n-m}(\la) J_{n-m}(\la')}J_{n-m}(\frac{\la }{r})J_{n-m}(\frac{\la'}{ r})\frac{rdr}{r^4}\\
&=\delta_{n,n'}\int_{r\leq 1}J_n(\la r)J_n(\la' r)rdr+\delta_{n,n'}\int_{r\leq 1}\frac{J_n(\la) J_n(\la')}{J_{n-m}(\la) J_{n-m}(\la')}J_{n-m}(\la r)J_{n-m}(\la r)rdr\\
\end{align*}}
Si $\la\neq \la'$, on a d'apr\`es \ref{intTTT}:
{\allowdisplaybreaks
\begin{align*}
\bigl(\vf_{n,\la}^{(m)},\vf_{n',\la'}^{(m)}\bigr)_{L^2,\infty}&=\frac{1}{\la'^2-\la^2}\Bigl(\la J_n(\la')J_n'(\la)-\la' J_n(\la)J_n'(\la')\Bigr)\\
&+\frac{1}{\la'^2-\la^2}\frac{J_n(\la) J_n(\la')}{J_{n-m}(\la) J_{n-m}(\la')}\Bigl(\la J_{n-m}(\la')J_{n-m}'(\la)-\la' J_{n-m}(\la)J_{n-m}'(\la')\Bigr)\quad \text{si}\; \la\neq \la'\\
&=\frac{1}{\la'^2-\la^2}\frac{\la J_n(\la')}{J_{n-m}(\la)}\Bigl(J_n'(\la)J_{n-m}(\la)+J_n(\la)J'_{n-m}(\la)\Bigr)\\
&-\frac{1}{\la'^2-\la^2}\frac{\la' J_n(\la)}{J_{n-m}(\la')}\Bigl(J_n'(\la')J_{n-m}(\la')+J_n(\la')J'_{n-m}(\la')\Bigr)\\
&=\frac{1}{\la'^2-\la^2}\frac{\la J_n(\la')}{J_{n-m}(\la)}\frac{m}{\la}J_n(\la)J_{n-m}(\la)-\frac{1}{\la'^2-\la^2}\frac{\la' J_n(\la)}{J_{n-m}(\la')}\frac{m}{\la'}J_n(\la')J_{n-m}(\la')\\
&=\frac{m J_n(\la')J_n(\la)}{\la'^2-\la^2}(1-1)\\
&=0.
\end{align*}}

(On a utilis\'e l'\'egalit\'e suivante $J_n'(\la)J_{n-m}(\la)+J_n(\la)J_{n-m}'(\la)=\frac{m}{\la}J_n(\la)J_{n-m}(\la) $ qui r\'esulte du fait $\la\in \mathcal{Z}_{m,n}$).
\item le cas $m=0$. Il suffit de traiter le cas de $\vf_{n,\la}^{(0)}$ et $\vf_{n,\la'}^{(0)}$ avec $J_n(\la)=0$
et $J_n'(\la')=0$. On a,
{\allowdisplaybreaks
\begin{align*}
\bigl(\vf_{n,\la}^{(0)},\vf_{n,\la'}^{(0)}\bigr)_{L^2,\infty}&=\int_{\p^1}f_{{}_{n,\la}}^{(0)}
\overline{f_{{}_{n,\la'}}^{(0)}}
\|1\|^2_{\infty}\omega_\infty\\
&=-\int_{r\leq 1}J_n(\la r)J_n(\la' r)rdr+\int_{r\geq 1}J_n(\frac{\la }{r})J_n(\frac{\la'}{ r})\frac{rdr}{r^4}\\
&=-\int_{r\leq 1}J_n(\la r)J_n(\la' r)rdr+\int_{r\leq 1}J_n(\la r)J_n(\la r)rdr\\
&=0.
\end{align*}}
\end{enumerate}

\item

On a d\'ej\`a montr\'e que $\mathcal{Z}_{m,n}$  est discret et infini (voir lemme \ref{zerosimple}).  Montrons que $ \mathcal{Z}_{m,n}\subset\R$. On proc\`ede par l'absurde: Soit $\la\in \mathcal{Z}_{m,n}$
et supposons que $\overline{\la}\neq \la$. On a $L_{m,n}(\overline{\la})=\overline{L_{m,n}(\la)}=0$, (puisque $L_{m,n}$ est une fonction analytique \`a coefficients r\'eels). Si l'on note par $f_{n,\overline{\la}}$ la fonction d\'efinie par

\[
f_{n,\overline{\la}}^{(m)} :=
\begin{cases}
 J_n(\overline{\la} r)\exp(i n\theta)  & \text{si } r \leq 1,\\
\frac{J_n(\overline{\la})}{J_{n-m}(\overline{\la})}r^{m}J_{n-m}(\frac{\overline{\la}}{r}) \exp(i n\theta)& \text{si } r>1,\\
\end{cases}
\]
et on pose
\[
 \vf_{n,\overline{\la}}^{(m)}:=f_{n,\overline{\la}}^{(m)}\otimes 1.
\]
Alors on a,
{\allowdisplaybreaks
\begin{align*}
\bigl(\vf_{n,\la}^{(m)},\vf_{{}_{n,\overline{\la}}}^{(m)})_{L^2,\infty}&=\int_{\p^1}f_{n,
\la}\overline{f_{n,\overline{\la}}}\|1\|^2_{\infty}\omega_\infty\\
&=\int_{r\leq 1}J_n(\la r)J_n(\overline{\la} r)rdr+\int_{r\geq 1}\frac{J_n(\la) J_n(\overline{\la})}{J_{n-m}(\la) J_{n-m}(\overline{\la})}J_{n-m}(\frac{\la }{r})J_{n-m}(\frac{\overline{\la}}{ r})\frac{rdr}{r^4}\\
&=\int_{r\leq 1}J_n(\la r)J_n(\overline{\la} r)rdr+\int_{r\leq 1}\frac{J_n(\la) J_n(\overline{\la})}{J_{n-m}(\la) J_{n-m}(\overline{\la})}J_{n-m}(\la r)J_{n-m}(\overline{\la} r)rdr\\
&=\frac{1}{\overline{\la}^2-\la^2}\Bigl(\la J_n(\overline{\la})J_n'(\la)-\overline{\la} J_n(\la)J_n'(\overline{\la})\Bigr)\\
&+\frac{1}{\overline{\la}^2-\la^2}\frac{J_n(\la) J_n(\overline{\la})}{J_{n-m}(\la) J_{n-m}(\overline{\la})}\Bigl(\la J_{n-m}(\overline{\la})J_{n-m}'(\la)-\overline{\la} J_{n-m}(\la)J_{n-m}'(\overline{\la})\Bigr)\quad\text{par}\, \ref{intTTT}\\
&=\frac{1}{\overline{\la}^2-\la^2}\frac{\la J_n(\overline{\la})}{J_{n-m}(\la)}\Bigl(J_n'(\la)J_{n-m}(\la)+J_n(\la)J'_{n-m}(\la)\Bigr)\\
&-\frac{1}{\overline{\la}^2-\la^2}\frac{\overline{\la} J_n(\la)}{J_{n-m}(\overline{\la})}\Bigl(J_n'(\overline{\la})J_{n-m}(\overline{\la})+J_n(\overline{\la})J'_{n-m}(\overline{\la})\Bigr)\\
&=\frac{1}{\overline{\la}^2-\la^2}\frac{\la J_n(\overline{\la})}{J_{n-m}(\la)}\frac{m}{\la}J_n(\la)J_{n-m}(\la)-\frac{1}{\overline{\la}^2-\la^2}\frac{\overline{\la} J_n(\la)}{J_{n-m}(\overline{\la})}\frac{m}{\overline{\la}}J_n(\la)J_{n-m}(\overline{\la})\\
&=\frac{m J_n(\overline{\la})J_n(\la)}{\overline{\la}^2-\la^2}(1-1)\\
&=0.
\end{align*}}
Notons que la deuxi\`eme \'egalit\'e nous donne que $\bigl(\vf_{n,\la}^{(m)},\vf_{{}_{n,\overline{\la}}})_{L^2,\infty} $ est un r\'eel strictement positif (puisque  c'est une somme de deux int\'egrales positives non nulles et que $\la\neq 0$ par hypoth\`ese). Cela aboutit \`a une contradiction. On conclut que
\[
 \overline{\la}=\la.
\]

\end{enumerate}

\end{proof}

Dans ce th\'eor\`eme,  nous montrons   que les fonctions $L_{m,n}$ g\'en\`erent des vecteurs propres pour le Laplacien $\Delta_{\overline{\mathcal{O}(m)}_\infty}$.
\begin{theorem}\label{laplacevect1}
Soit $m\in \N$, on a pour tout $n\in \Z$,  $\la\in \mathcal{Z}_{m,n}$ et $k\in \{0,1,\ldots,m\}$:

\[
\bigl(1\otimes z^k, \Delta_{\overline{\mathcal{O}(m)}_\infty} \xi\bigr)_{L^2,\infty}=0,
\]
\[
\bigl(\vf_{n,\la}^{(m)}, \Delta_{\overline{\mathcal{O}(m)}_\infty} \xi\bigr)_{L^2,\infty}=\frac{\la^2}{4}\bigl(\vf_{n,\la}^{(m)}, \xi\bigr)_{L^2,\infty},
\]
et
\[
\bigl(\vf_{{}_{-n+m,\la}}^{(m)}, \Delta_{\overline{\mathcal{O}(m)}_\infty} \xi\bigr)_{L^2,\infty}=\frac{\la^2}{4}\bigl(\vf_{{}_{-n+m,\la}}^{(m)}, \xi\bigr)_{L^2,\infty},\footnote{Notons que la deuxi\`eme \'egalit\'e est bien d\'efinie, puisqu'on a montr\'e dans \ref{zerosimple} que $\mathcal{Z}_{m,n}=Z_{-n+m}$.}
\]
pour tout $\xi\in A^{(0,0)}(\p^1,\mathcal{O}(m))$ . En particulier,
 \[
\Bigl\{0\Bigr\}\bigcup \Bigl\{\frac{\la^2}{4} \Bigl|\, \exists n\in \N,\, \la\in \mathcal{Z}_{m,n}  \Bigr\}\subset \mathrm{Spec}(\Delta_{\overline{\mathcal{O}(m)}_\infty}).
\]

\end{theorem}
\begin{proof}
La premi\`ere assertion est une cons\'equence imm\'ediate du \ref{x31}.

Notons que $\vf_{n,\la}^{(m)}$ et $\vf_{-n+m,\la}^{(m)}$ sont lin\'eairement d\'ependants si et seulement si $m$ est pair et que $n=\frac{m}{2}$. En effet, (par d\'efinition de $\vf_{n,\la}^{(m)}$, voir \ref{z2}), ces deux vecteurs sont li\'es s'il existe  $c\in \C$ tel que
\begin{align*}
J_n(\la r)e^{in\theta}&=c J_{-n+m}(\la r)e^{i(-n+m)\theta},\\
 \text{et}\;\frac{J_n(\la)}{J_{n-m}(\la)}J_{n-m}(\la r) e^{{in\theta}}&=c \frac{J_{-n+m}(\la)}{J_{-n}(\la)}J_{-n}(\la r) e^{{i(-n+m)\theta}}\quad \forall\,\, r\in [0,1]\;\forall\, \theta\in \R.
\end{align*}
Donc,
\begin{align*}
J_n(\la r)&=(-1)^{m-n}c J_{n-m}(\la r)e^{i(-2n+m)\theta}\,\text{et}\,
 J_{n-m}(\la r)=(-1)^{n-m}c \frac{J_{n-m}(\la)^2}{J_{n}(\la)^2}J_{n}(\la r) e^{{i(-2n+m)\theta}},
\end{align*}
pour tout $r\in [0,1]$ et pour tout $\theta\in \R$. On d\'eduit que  $m$ est pair et que $n=\frac{m}{2}$.\\

Afin de montrer que $\vf_{n,\la}^{(m)}$ est un vecteur propre, on aura besoin de l'expression locale de $\Delta_{\overline{\mathcal{O}(m)}_\infty}$  sur $\p^1\setminus \s $. Soit $f\in A^{0,0}(\p^1)$ et $k\in \{0,\ldots,m\}$.
D'apr\`es \ref{z3}, on a sur $\{z\in \C|\,|z|<1\}$,
{{} \begin{equation}\label{exp111}
\begin{split}
\Delta_{\overline{\mathcal{O}(m)}_\infty}(f\otimes z^k )=-|z|^{-2k}\frac{\pt}{\pt z}\bigl( |z|^{2k}\frac{\pt }{\pt\z}f  \bigr)\otimes z^ k
=-z^{-k} \frac{\pt^2}{\pt z\pt\z}(z^k f)\otimes z^k,
\end{split}
\end{equation}
}

et sur $\{|z|>1\}$,
{{} \begin{equation}\label{exp112}
 \begin{split}
\Delta_{\overline{\mathcal{O}(m)}_\infty}(f\otimes z^k )=-|z|^{2m-2k+4}\frac{\pt}{ \pt z} \bigl( |z|^{2k-2m} \frac{\pt}{ \pt \z}f \bigr)\otimes z^ k
=-\dfrac{|z|^4}{z^{k-m}}\frac{\pt^2}{\pt z\pt \z}( \frac{f}{z^{m-k}})\otimes z^k.
\end{split}
\end{equation}
}


Fixons $n\in \Z$ et soit $\la \in  \R^\ast \setminus\{x |\,J_{n-m}(x)=0\}$. Soient
$c,d\in \R$ deux constantes non nulles et posons $\widetilde{f}_{n,\la}$ la fonction sur $\C$
donn\'ee comme suit:

\begin{equation}
\widetilde{f}_{n,\la}(r,\theta) :=
\begin{cases}
 \widetilde{f}_{n,\la,-}(r,\theta):= c\, J_n(\la r)e^{i n\theta} & \text{si } r \leq 1,\theta\in \R\\
 \widetilde{f}_{n,\la,+}(r,\theta):=d\,\frac{J_n(\la)}{J_{n-m}(\la)}r^mJ_{n-m}(\frac{\la}{r})e^{in \theta} & \text{si } r>1,\, \theta\in \R. \\
\end{cases}
\end{equation}
et on  consid\`ere l'\'el\'ement $\widetilde{\vf}_{n,\la}^{(m)}$ d\'efini presque partout sur $\p^1$  par:
\begin{equation}\label{x1}
\widetilde{\vf}_{n,\la}^{(m)}:=\widetilde{f}_{n,\la}\otimes 1
\end{equation}

On va montrer l'\'equivalence suivante:
 \begin{align*}
\la \in \mathcal{Z}_{m,n}  \, \Leftrightarrow &\bigl(\widetilde{\vf}_{n,\la}^{(m)}, \Delta_{\overline{\mathcal{O}(m)}_\infty} \xi\bigr)_{\infty}=\frac{\la^2}{4} \bigl(\widetilde{\vf}_{n,\la}^{(m)},\xi\bigr)_{\infty}\, \forall\, \xi\in A^{0,0}(\p^1,\mathcal{O}(m))\\
&\,\text{avec}\,\begin{cases}
 c=d&, \text{lorsque}\, m\geq 1\\
  c=d& \text{si }\,J_n'(\la)=0\,\text{et}\, c=-d\, \text{si}\, J_n(\la)=0,\, \text{lorsque}\, m=0.\\
\end{cases}.
\end{align*}

Pour cela, nous allons \'etablir  que pour tout $n\in \Z$ et tout $\la \in  \R^\ast \setminus\{x |\,J_{n-m}(x)=0\}$,
on a:
\[
\begin{split}\label{x10}
(\widetilde{f}_{n,\la}\otimes 1,\Delta_{\overline{\mathcal{O}(m)}_\infty}(g\otimes
z^l))_{L^2,\infty}&=\frac{\la^2}{4}(\widetilde{f}_{n,\la}\otimes 1,g\otimes z^l)_{L^2,\infty}+\int_{|z|=1}
\bigl(\widetilde{f}_{n,\la,+}-\widetilde{f}_{n,\la,-}\bigr)d^c (\z^l\overline{g})\\
&+\int_{|z|=1}
\bigl(d^c  \widetilde{f}_{n,\la,-} -d^c\widetilde{f}_{n,\la,+} )\z^l \overline{g},
\end{split}
\]
pour tout $g\in A^{0,0}(\p^1)$ et $l\in \{0,1,\ldots,m\}$.\\

On a
{\allowbreak
\[
 \begin{split}
  (\widetilde{f}_{n,\la}\otimes 1,\Delta_{\overline{\mathcal{O}(m)}_\infty}(g\otimes z^l))_{L^2,\infty}&=-\int_{|z|\leq 1}\widetilde{f}_{n,\la}\z^{-l}\frac{\pt^2}{\pt z\pt \z}( \z^l \overline{g})\z^l dz\wedge d\z\\
&-\int_{|z|\geq 1} \widetilde{f}_{n,\la}\z^{m-l}\frac{\pt^2}{\pt z\pt \z}( \frac{ \overline{g}}{\z^{m-l}})
z^{-m}\z^{l-m} dz\wedge d\z\\
&=-\int_{|z|\leq 1}\widetilde{f}_{n,\la}\frac{\pt^2}{\pt z\pt \z}( \z^l \overline{g}) dz\wedge d\z-\int_{|z|\geq 1} \widetilde{f}_{n,\la}z^{-m}\frac{\pt^2}{\pt z\pt \z}( \frac{ \overline{g}}{\z^{m-l}}) dz\wedge d\z.
 \end{split}
\]}

Par la formule de Green \ref{Green},
{{} \[
 \begin{split}
-\int_{|z|\leq 1}\widetilde{f}_{n,\la}\frac{\pt^2}{\pt z\pt \z}( \z^l \overline{g}) dz\wedge d\z&=-\int_{|z|\leq 1}\frac{\pt^2}{\pt z\pt \z}(  \widetilde{f}_{n,\la,-}) \z^l \overline{g} dz\wedge d\z-\int_{|z|=1} \widetilde{f}_{n,\la
,-} d^c(\z^l \overline{g})\\
&+\int_{|z|=1} d^c ( f_{\la,n,-} )\z^l \overline{g}\\
&=-\int_{|z|\leq 1}\frac{\pt^2}{\pt z\pt \z}(  \widetilde{f}_{n,\la,-}) \z^l \overline{g} dz\wedge d\z-\int_{|z|=1}  \widetilde{f}_{n,\la,-} d^c(\z^l \overline{g})\\
&+\int_{|z|=1} d^c ( \widetilde{f}_{n,\la,-} )\z^l \overline{g},
 \end{split}
\]}
et
 \[
 \begin{split}
  -\int_{|z|\geq 1} \widetilde{f}_{n,\la}z^{-m}\frac{\pt^2}{\pt z\pt \z}&( \frac{ \overline{g}}{\z^{m-l}}) dz\wedge d\z=\\
& -\int_{|z|\geq 1} \frac{\pt^2}{\pt z\pt \z}( \frac{ \widetilde{f}_{n,\la,+}}{z^{m}})\frac{\overline{g}}{\z^{m-l}}dz
\wedge d\z+\int_{|z|=1} \frac{ \widetilde{f}_{n,\la,+}}{z^{m}}d^c\bigl(\frac{\overline{g}}{\z^{m-l}}\bigr)\\
&-\int_{|z|=1} d^c\bigl(\frac{\widetilde{f}_{n,\la,+}}{z^{m}} \bigr)\frac{\overline{g}}{\z^{m-l}}\\
=&-\int_{|z|\geq 1} \frac{\pt^2}{\pt z\pt \z}( \frac{ \widetilde{f}_{n,\la,+}}{z^{m}})
\frac{\overline{g}}{\z^{m-l}}dz\wedge d\z+ \int_{|z|=1}  \widetilde{f}_{n,\la,+}d^c (\z^l\overline{g})\\
&+\int_{|z|=1} \z^l \overline{g}\widetilde{f}_{n,\la,+}\bigl(\frac{1}{z^m}d^c(\frac{1}{\z^m}) -\frac{1}{\z^m}d^c(\frac{1}{z^m})\bigr)
-\int_{|z|=1} d^c(\widetilde{f}_{n,\la,+}) {\z^l\overline{g}}\\
=&-\int_{|z|\geq 1} \frac{\pt^2}{\pt z\pt \z}( \frac{ \widetilde{f}_{n,\la,+}}{z^{m}})
\frac{\overline{g}}{\z^{m-l}}dz\wedge d\z+ \int_{|z|=1}  \widetilde{f}_{n,\la,+}d^c (\z^l\overline{g})-\int_{|z|=1}
d^c(\widetilde{f}_{n,\la,+}) {\z^l\overline{g}}
\end{split}
\]
(On a utilis\'e le fait que $ \frac{dz}{z}=-\frac{d\z}{\z}$
sur $\mathbb{S}^1$).

Observons  que par \ref{exp111} et \ref{exp112}, on  a sur $\p^1\setminus \s$:
 \begin{equation}\label{x9}
\Delta_{\overline{\mathcal{O}(m)}_\infty}(\widetilde{f}_{n,\la}\otimes z^k )=\frac{\la^2}{4} \widetilde{f}_{n,\la} \otimes z^k.
\end{equation}
En regroupant tout cela, et apr\`es des simplifications \'evidentes, il vient que:
{\allowdisplaybreaks \[
 \begin{split}
  (\widetilde{f}_{n,\la}\otimes 1,&\Delta_{\overline{\mathcal{O}(m)}_\infty}(g\otimes z^l))_{L^2,\infty}=\\
&-\int_{|z|\leq 1}\frac{\pt^2}{\pt z\pt \z}( \widetilde{f}_{n,\la,-}) \z^l \overline{g} dz\wedge d\z-\int_{|z|\geq 1} \frac{\pt^2}{\pt z\pt \z}( \frac{ \widetilde{f}_{n,\la,+}}{z^{m}})\frac{\overline{g}}{\z^{m-l}}dz\wedge d\z\\
&\int_{|z|=1}  \bigl(\widetilde{f}_{n,\la,+}-\widetilde{f}_{n,\la,-}\bigr)d^c (\z^l\overline{g})+\int_{|z|=1}
\bigl(d^c  \widetilde{f}_{n,\la,-} -d^c\widetilde{f}_{n,\la,+} )\z^l \overline{g}\\
&=\frac{\la^2}{4}\int_{|z|\leq 1} \widetilde{f}_{{}_{n,\la,-}}\z^l\overline{g}dz\wedge d\z+\frac{\la^2}{4} \int_{|z|\geq 1} \frac{\widetilde{f}_{n,\la,+}}{z^{m}|z|^4}\frac{\overline{g}}{\z^{m-l}}dz\wedge d\z \quad\text{par}
\,\ref{x9}\\
&+\int_{|z|=1}  \bigl(\widetilde{f}_{n,\la,+}-\widetilde{f}_{n,\la,-}\bigr)d^c (\z^l\overline{g})+\int_{|z|=1}
\bigl(d^c  \widetilde{f}_{n,\la,-} -d^c\widetilde{f}_{n,\la,+} )\z^l \overline{g}\\
&=\frac{\la^2}{4}(\widetilde{f}_{n,\la}\otimes 1,g\otimes z^l)_{L^2,\infty}+\int_{|z|=1}  \bigl(\widetilde{f}_{n,\la,+}-\widetilde{f}_{n,\la,-}\bigr)d^c (\z^l\overline{g})+\int_{|z|=1}
\bigl(d^c  \widetilde{f}_{n,\la,-} -d^c\widetilde{f}_{n,\la,+} )\z^l \overline{g}.
 \end{split}
\]}

Donc pour que {{}$ \frac{\la^2}{4}\in \mathrm{Spec}(\Delta_{\overline{\mathcal{O}(m)}_\infty})$} il faut et
il suffit que
 \begin{equation}\label{z4}
 \int_{|z|=1}  \bigl(\widetilde{f}_{n,\la,+}-\widetilde{f}_{n,\la,-}\bigr)d^c (\z^l\overline{g})+\int_{|z|=1}
\bigl(d^c  \widetilde{f}_{n,\la,-} -d^c\widetilde{f}_{n,\la,+} )\z^l \overline{g}=0\quad\forall\,
g\in A^{0,0}(\p^1),\,\forall\, l\in \{0,1,\ldots,m\}.
\end{equation}
Ce qui est \'equivalent \`a:
\begin{equation}\label{x11}
\widetilde{f}_{n,\la,+}=\widetilde{f}_{n,\la,-}\quad \text{et}\quad d^c( f_{n,\la,-} )= d^c(f_{n,\la,+})\quad
 \text{sur}\,\, \s.\footnote{Notons que l'\'equation \ref{z4} peut s'\'ecrire sous la forme suivante $\int_{\s}\phi_0
 d^c\phi-\phi d^c\phi_0=0$ o\`u $\phi_0=\widetilde{f}_{n,\la,+}-\widetilde{f}_{n,\la,-}$ et $\phi=\z^l \overline{g}$. Le
 fait que $\frac{\la^2}{4}\in \mathrm{Spec}(\Delta_{\overline{\mathcal{O}(m)}_\infty})$ revient \`a \'etablir que
 la fonction nulle est l'unique fonction sur $\s$, solution de
  l'\'equation fonctionnelle  $\int_{\s}\phi_0
 d^c\phi-\phi d^c\phi_0=0$ pour toute fonction $\phi$ de classe  $\cl$ au voisinage de $\s$. Ce qui est le cas,
 en effet, il suffit de consid\'erer la fonction $|z|\phi$, ce qui nous donne $\int_{\s}\phi_0\phi d^c|z|=0$ pour tout
 $\phi$. }
\end{equation}

Pour simplifier les notations on pose
{{} \[
 \psi_{-}(r)=cJ_n(\la r)\qquad \text{et}\qquad\psi_{+}(r)=d \frac{J_{n}(\la)}{J_{n-m}(\la)}r^{m}J_{n-m}(\frac{\la}{r}).
\]}
Donc, $\widetilde{f}_{n,\la,-}(z)=\psi_-(|z|)(\frac{z}{|z|})^n$ lorsque $|z|\leq 1$ et
$\widetilde{f}_{n,\la,+}(z)=\psi_+(|z|)(\frac{z}{|z|})^n$ si $|z|\geq 1$.
\begin{enumerate}
\item
Soit $0<|z|\leq 1$.  Par un simple calcul, on a $\frac{\partial \widetilde{f}_{n,\la,-}}{\partial
z }(z)=\frac{\partial}{\partial
z}\bigl(\bigl(\frac{z}{|z|}\bigr)^n\bigr)\psi_-(|z|)+\bigl(\frac{z}{|z|}\bigr)^n\frac{\overline{z}}{2|z|}
\frac{\partial \psi_-}{\partial r}(|z|)$ et $\frac{\partial \widetilde{f}_{n,\la,-}}{\partial\overline{ z} }(z)=
\frac{\partial}{\partial\overline{z}}\bigl(\bigl(\frac{{z}}{|z|}\bigr)^{n-k}\bigr)\psi_-(|z|)
+\bigl(\frac{z}{|z|}\bigr)^{n-k}\frac{{z}}{2|z|}\frac{\partial \psi_-}{\partial r}(|z|).$ Sur $\s$, on a alors

\begin{equation}\label{x12}
 d^c f_{n,\la,-}=n\psi_-(1)e^{i(n-1)\theta}d\theta+\frac{\partial \psi_-}{\partial r}(1) e^{in\theta} d\theta.
\end{equation}

Par un calcul semblable au pr\'ec\'edent, on obtient sur $\s$:
\begin{equation}\label{x13}
 d^c f_{n,\la,+}=n\psi_+(1)e^{i(n-1)\theta}d\theta+\frac{\partial \psi_+}{\partial r}(1) e^{in\theta} d\theta
\end{equation}



Donc \ref{x11} est \'equivalente \`a:
\begin{equation}\label{x24}
\psi_-(1)=\psi_+(1)\quad \text{et}\quad \frac{\partial \psi_-}{\partial r}(1)=\frac{\partial \psi_+}{\partial r}(1).
\end{equation}
On a
 $\frac{\partial \psi_+}{\partial r}(1)=d\frac{J_n(\la)}{J_{n-m}(\la)}\frac{d}{dr}(r^{m-k}J_{n-m}(\frac{\la}{r}) )_{|_{r=1}}=d m J_n(\la)-d \la\frac{J'_{n-m}(\la)}{J_{n-m}(\la)}J_n(\la)$, $\frac{\partial \psi_-}{\partial r}(1)=c\la J_n'(\la)$
et
$c J_n(\la)=d J_n(\la)$.
 Par suite, \ref{x24} devient:
  \[(c-d)J_n(\la)=0\quad \text{et}  \quad d m J_{n}(\la)-d\la\frac{J'_{n-m}(\la)}{J_{n-m}(\la)}J_n(\la)=c\la J_n'(\la).\]
  \begin{enumerate}
  \item
 Si $m\geq 1$, on en d\'eduit que $c=d$ (En effet, si $J_n(\la)=0$, l'\'equation ci-dessus
 donne que $J_n'(\la)=0$. Ce qui est impossible) par suite $\la\in \mathcal{Z}_{m,n}$.
 \item
 Si $m=0$, alors $\la\in \mathcal{Z}_{0,n}$ avec
 $c=d$ si   $\la$ v\'erifie $J_n'(\la)=0$, et $c=-d$ si $J-n(\la)=0$.
\end{enumerate}

\end{enumerate}
On a donc montr\'e que  $\la \in \mathcal{Z}_{m,n}$ si et seulement si
 \[
\bigl(\vf^{(m)}_{n,\la},\Delta_{\overline{\mathcal{O}(m)}_\infty}\xi \bigr)_{L^2,\infty}=\frac{\la^2}{4}\bigl(\vf^{(m)}_{n,\la},\xi\bigr)_{L^2,\infty}
 \quad \forall\, \xi\in A^{0,0}(\p^1,\mathcal{O}(m)).\]

 Si l'on remplace $n$ par $-n+m$, alors la derni\`ere assertion devient:
$\la \in \mathcal{Z}_{m,-n+m}$ si et seulement si
 \[
\bigl(\vf^{(m)}_{-n+m,\la},\Delta_{\overline{\mathcal{O}(m)}_\infty}\xi \bigr)_{L^2,\infty}=\frac{\la^2}{4}\bigl(\vf^{(m)}_{-n+m,\la},\xi\bigr)_{L^2,\infty}
 \quad \forall\, \xi\in A^{0,0}(\p^1,\mathcal{O}(m)).\]
Or $\mathcal{Z}_{m,-n+m}=\mathcal{Z}_{m,n}$ (voir \ref{x3}). Cela compl\`ete la preuve du th\'eor\`eme.
\end{proof}

Le th\'eor\`eme  suivant sera utilis\'e de mani\`ere crucial dans l'\'etude du spectre  du Laplacien $\Delta_{\overline{\mathcal{O}(m)}_\infty}$:
\begin{theorem}\label{deriveeLnorme}
Soit $\nu\in \Z$ et  $\la\in \mathcal{Z}_{m,\nu}$. On a
\[
L'_{m,\nu}(\la)=2\frac{J_{\nu-m}(\la)}{J_{\nu}(\la)}\bigl(\vf_{{}_{\nu,\la}}^{(m)},\vf_{{}_{\nu,\la}}^{(m)}
\bigr)_{L^2,\infty},
\]
en particulier $\la$ est un z\'ero simple de $L_{m,\nu}$.
\end{theorem}
\begin{proof}
Notons par $K_{m,\nu}$ la fonction d\'efinie sur  $\C^\ast$ comme suit:
\[
K_{m,\nu}(z)=\frac{d}{dz}\bigl(z^{-m}J_{\nu-m}(z)J_\nu(z) \bigr).
\]
On a $K'_{m,\nu}(z)=m(m+1)z^{-m-2}J_{\nu-m}(z)J_\nu(z)-2mz^{-m-1}\frac{d}{dz}\bigl( J_{\nu-m}(z)J_\nu(z)\bigr)+z^{-m}\frac{d^2}{dz^2}\bigl(J_{\nu-m}(z)J_\nu(z) \bigr)$.
Sachant que $J_\nu$ v\'erifie $J''_\nu(z)+\frac{1}{z}J'_\nu(z)+(1-\frac{\nu^2}{z^2})J_{\nu}(z)=0,\; \forall\, z\in \C^\ast$,   alors par un simple calcul, on a $
\frac{d^2}{dz^2}\bigl(J_{\nu-m}(z)J_\nu(z) \bigr)=-\frac{1}{z}\bigl(J_{\nu}'(z)J_{\nu-m}(z)+J_{\nu-m}'(z)J_\nu(z) \bigr)-\bigl(2-\frac{(\nu-m)^2+\nu^2}{z^2} \bigr)J_{\nu-m}(z)J_\nu(z)
+2J_{\nu}'(z)J_{\nu-m}'(z)$. En remplaçant dans l'expression de $K'_{m,\nu}$, on obtient:
\begin{equation}\label{x6}
\begin{split}
K_{m,\nu}'(z)&=z^{-m-2}\Bigl( \bigl(2m^2+2(\nu-m)^2+(2(\nu-m)+1)m-2z^2 \bigr)J_{\nu-m}(z)J_\nu(z)\\
&-(2m+1)zJ_{\nu}'(z)J_{\nu-m}(z)-(2m+1)zJ_\nu(z)J_{\nu-m}'(z)+2z^2J'_{\nu-m}(z)J'_\nu(z)\Bigr)\quad \forall\, z\in \C^\ast.
\end{split}
\end{equation}

Si l'on consid\`ere $\la\in \mathcal{Z}_{m,\nu}$ (c-\`a-d $\frac{d}{dz}(z^{-m} J_{\nu}(z)J_{\nu-m}(z))=0$). Cela
donne les
deux identit\'es suivantes:
\begin{equation}\label{x5}
\la\bigl(J_{\nu-m}(\la)J'_{n}(\la)+J_{\nu}(\la)J_{\nu-m}'(\la) \bigr)=mJ_{\nu}(\la)J_{\nu-m}(\la),
\end{equation}
et
\begin{equation}\label{x8}
\frac{J_{\nu}'(\la)^2}{J_{\nu}(\la)^2}+\frac{J_{\nu-m}'(\la)^2}{J_{\nu-m}(\la)^2}+
2\frac{J_{\nu}'(\la)J_{\nu-m}'(\la)}{J_{\nu}(\la)J_{\nu-m}(\la)}=\biggl(\frac{J_{\nu}'(\la)}{J_{\nu}(\la)}+\frac{J_{\nu-m}'(\la)}{J_{\nu-m}(\la)} \biggr)^2 =\frac{m^2}{\la^2}\footnote{Rappelons que  $J_{\nu}(\la)J_{\nu-m}(\la)\neq 0$ lorsque $\la\in
\mathcal{Z}_{m,n}$  (voir \ref{x4}). }
\end{equation}

En utilisant \ref{x5},  la formule \ref{x6} devient:

\begin{equation}\label{x7}
K'_{m,\nu}(\la)=\la^{-m-2}\Bigl(2\bigl((\nu-m)\nu-\la^2\bigr)J_{\nu-m}(\la)J_{\nu}(\la)+2\la^2J_{\nu-m}'(\la)J_{\nu}'(\la)  \Bigr).
\end{equation}

De \ref{x5} on en tire:
Ce qui nous permet d'\'ecrire que
\[
\begin{split}
K'_{m,\nu}(\la)&=\la^{-m-2}\Bigl(2\bigl((\nu- m)\nu
-\la^2\bigr)J_{\nu-m}(\la)J_{\nu}(\la)+2\la^2J_{\nu-m}'(\la)J_{\nu}'(\la)  \Bigr)\\
&=\la^{-m}J_{\nu-m}(\la)J_{\nu}(\la)\Biggl(2\Bigl(\frac{(\nu-m)^2+(\nu-m) m}{\la^2}-1 \Bigr)+2\frac{J_{\nu}'(\la)J_{\nu-m}'(\la)}{J_{\nu}(\la)J_{\nu-m}(\la)}\Biggr)\\
&=\la^{-m}J_{\nu-m}(\la)J_{\nu}(\la)\Biggl(2\Bigl(\frac{(\nu-m)^2+(\nu-m) m}{\la^2}-1 \Bigr)+\frac{m^2}{\la^2}-\frac{J_{\nu}'(\la)^2}{J_{\nu}(\la)^2}-\frac{J_{\nu-m}'(\la)^2}{J_{\nu-m}(\la)^2}\Biggr)
\quad\text{par}\, \ref{x8}\\
&=\la^{-m}J_{\nu-m}(\la)J_{\nu}(\la)\Biggl(-\frac{J_{\nu}'(\la)^2}{J_{\nu}(\la)^2}-\frac{J_{\nu-m}'(\la)^2}{J_{\nu-m}(\la)^2}+\Bigl(\frac{(\nu-m)^2+\nu^2}{\la^2}-2 \Bigr)\Biggr)
\end{split}
\]
Or on a d\'ej\`a montr\'e que (voir th\'eor\`eme \ref{vecteurpropreL}):
\[
\bigl(\vf_{\nu,\la}^{(m)},\vf_{\nu,\la}^{(m)})_{L^2,\infty}=\frac{J_{\nu}(\la)^2}{2}\Bigl(\frac{J'_{\nu}(\la)^2}{J_{\nu}(\la)^2}+\frac{J_{\nu-m}'(\la)^2}{J_{\nu-m}(\la)^2} \Bigr)+\frac{J_{\nu}(\la)^2}{2}\Bigl(2-\frac{(\nu-m)^2+\nu^2}{\la^2}\Bigr).
\]
Donc,
\[
K_{m,\nu}'(\la)=-2\la^{-m}\frac{J_{\nu-m}(\la)}{J_{\nu}(\la)}\bigl(\vf_{\nu,\la}^{(m)},\vf_{\nu,\la}^{(m)})_{L^2,\infty}.
\]
Par suite,
\[
L_{m,\nu}'(\la)=-m\la^{m-1}K_{m,\nu}(\la)-\la^m
K_{m,\nu}'(\la)=-\la^mK_{m,\nu}'(\la)=2\frac{J_{\nu-m}(\la)}{J_{\nu}(\la)}\bigl(\vf_{\nu,\la}^{(m)},
\vf_{\nu,\la}^{(m)})_{L^2,\infty}.
\]

Montrons maintenant que $\la$ est un z\'ero simple de $L_{m,\nu}$. Cela d\'ecoule  du \ref{x4} et de
la formule pr\'ec\'edente.
\end{proof}

\section{Sur le Spectre de $\Delta_{\overline{\mathcal{O}(m)}_\infty}$ (II) }\label{V1}

Dans ce paragraphe on se propose de prouver que les valeurs propres de  $\Delta_{\overline{\mathcal{O}(m)}_\infty}$ calcul\'ees  dans le th\'eor\`eme
 \ref{laplacevect1} sont les uniques valeurs propres possibles, en d'autres termes, on a:
\begin{theorem}\label{x333}
Pour tout $m\in\N$, $\Delta_{\overline{\mathcal{O}(m)}_\infty}$ admet un spectre discret, positif et infini, et on a
 \begin{equation}\label{x25}
\mathrm{Spec}(\Delta_{\overline{\mathcal{O}(m)}_\infty})=\Bigl\{0\Bigr\}\bigcup \Bigl\{\frac{\la^2}{4} \Bigl|\, \exists \nu\in \N,\, \la\in \mathcal{Z}_{m,\nu}  \Bigr\},
\end{equation}
En plus, on a
\begin{enumerate}
\item Lorsque  $m$ est pair, alors  la multiplicit\'e de $\frac{\la^2}{4}$ est \'egale  \`a $2$ si $\la\in
    \mathcal{Z}_{m,n}$ si $n\geq m+1 $ ou $0\leq n\leq \frac{m}{2}-1$,
     et de multiplicit\'e $1$ si $\la\in \mathcal{Z}_{m,\frac{m}{2}}$.
\item Si $m$ est impair, alors $\frac{\la^2}{4}$ est de multiplicit\'e $2$ si $n\geq m+1$, et vaut
$1$ si $0\leq n\leq m $.
\end{enumerate}
\end{theorem}

Dans le th\'eor\`eme \ref{laplacevect1}, nous avons
 associ\'e \`a tout \'el\'ement de $\mathcal{Z}_{m,\nu}$ un  vecteur propre  non nul pour le Laplacien
 $\Delta_{\overline{\mathcal{O}(m)}_\infty}$. Ce qui nous a permis de d\'eduire une partie du spectre.

  Afin,
 d'\'etablir  le th\'eor\`eme \ref{x30}, il
suffira alors d'\'etablir le th\'eor\`eme suivant:

\begin{theorem}\label{x16}

\begin{equation}\label{x14}
\Bigl\{1\otimes 1,1\otimes z,\ldots,1\otimes z^m\Bigr\}\bigcup\Bigl\{\vf_{\nu,\la}^{(m)}\,\Bigl|\, \nu\in\Z,\, \la\in \mathcal{Z}_{m,\nu}\Bigr\},
\end{equation}
est une base hilbertienne pour   $\mathcal{H}^{(m)}$.
\end{theorem}

\begin{remarque}\rm{ Notons que la famille $\Bigl\{1\otimes 1,1\otimes z,\ldots,1\otimes z^m\Bigr\}\bigcup\Bigl\{\vf_{\nu,\la}^{(m)}\,\Bigl|\, \nu\in\Z,\, \la\in \mathcal{Z}_{m,\nu}\Bigr\}$ est orthogonale
pour le produit $(,)_{L^2,\infty}$. En effet, nous avons \'etabli que $\Bigl\{\vf_{\nu,\la}^{(m)}\,\Bigl|\, \nu\in\Z,\, \la\in \mathcal{Z}_{m,\nu}\Bigr\}$ est orthogonale
  (voir th\'eor\`eme \ref{vecteurpropreL}), il reste \`a \'etablir que
$(1\otimes z^k,\vf_{\nu,\la}^{(m)})_{L^2,\infty}=0$, $\forall \,
k\in \{0,1,\ldots,m\}, \,\forall\,\nu\in\Z,\, \forall\,\la\in \mathcal{Z}_{m,\nu}$. D'apr\`es \ref{laplacevect1},
$(\vf_{\nu,\la}^{(m)},\Delta_{\overline{\mathcal{O}(m)}_\infty}(1\otimes z^k))_{L^2,\infty}=\frac{\la^2}{4}
(\vf_{\nu,\la}^{(m)},1\otimes z^k)_{L^2,\infty}$. Or, le terme \`a gauche est nul (voir proposition \ref{x31}).
On conclut en rappelant que $\la\neq 0$.

}

\end{remarque}
Nous commen\c{c}ons par montrer que le cas $m=0$ d\'ecoule de la th\'eorie des s\'eries de Fourier-Bessel classique
\cite[\S. 18]{Watson} ou le paragraphe \ref{z10} pour un bref rappel.

 Quant au cas $m\geq 1$, il sera l'application d'une th\'eorie de s\'eries de Fourier-Bessel
g\'en\'eralis\'ee adapt\'ee que nous allons d\'evelopper au paragraphe \ref{z9}.\\

Le cas  $m=0$ est facile \`a traiter. En effet, on verra que l'\'etude du Laplacien
$\Delta_{\overline{\mathcal{O}}_\infty}$ est \'equivalente \`a deux probl\`emes avec conditions au bord sur le
 disque unit\'e
classiques:
 \begin{theorem}[le cas $m=0$]\label{x15}
 \[
\Bigl\{1\Bigr\}\bigcup\Bigl\{\vf_{\nu,\la}^{(0)}\,|\, \nu\in\Z,\, \la\in \mathcal{Z}_{0,\nu}\Bigr\},
\]
est une base hilbertienne pour $\mathcal{H}^{(0)}$.
 \end{theorem}
 \begin{proof}
 Soit $f\in \mathcal{H}^{(0)}$. Supposons que:
 \[
 (f,1)_{L^2,\infty}=(f,\vf_{\nu,\la}^{(0)})_{L^2,\infty}=0\quad  \forall\,\nu\in\Z,\,\forall\, \la\in \mathcal{Z}_{0,\nu},
 \]
 et montrons qu'on a n\'ecessairement $f=0$ dans  $\mathcal{H}^{(0)}$.

Fixons $\nu\in \Z$. Soit $\la\in \mathcal{Z}_{0,\nu}$, on a (voir notations \ref{z2}):
\begin{align*}
(f,1)_{L^2,\infty}&=\int_0^1  \bigl(f_\nu(r)+ f_\nu(\frac{1}{r})\bigr)rdr\\
(f,\vf_{\nu,\la}^{(0)})_{L^2,\infty}&=\int_0^1  \bigl(f_\nu(r)+ f_\nu(\frac{1}{r})\bigr)J_\nu(\la r)rdr\quad
\text{si}\;J_\nu'(\la)=0\\
(f,\vf_{\nu,\la}^{(0)})_{L^2,\infty}&=\int_0^1  \bigl(f_\nu(r)-f_\nu(\frac{1}{r})\bigr)J_\nu(\la r)rdr\quad
\text{si}\;J_\nu(\la)=0,
\end{align*}
(o\`u on a pos\'e $f_\nu(r,\theta):=\int_0^{2\pi} f(r,\theta)e^{-i\nu\theta}d\theta$).

Il est bien connu que la famille $\{1,  J_\nu(\la r))|\,  J_\nu'(\la)=0\}$ (resp. $\{ J_\nu(\la r)|\,
J_\nu(\la)=0\}$) forme un syst\`eme total pour l'espace des fonctions $L^2$
 sur $[0,1]$ v\'erifiant la condition de Neumann  (resp. la condition de Dirichlet), voir par exemple
 \cite[p.381]{Whittaker}. On  d\'eduit que
$\int_0^1|f_\nu(r)|^2rdr=\int_0^1|f_\nu(\frac{1}{r})|^2rdr=0$ pour tout $\nu\in \Z$. On conclut que:
\[
\|f\|_{L^2,\infty}=0.
\]
(on a utilis\'e \ref{z11}, voir \ref{z12} pour les notations.)
 \end{proof}

 \vspace{1cm}
 \'Etablir le th\'eor\`eme \ref{x16} revient \`a  montrer que pour toute fonction  $f$ sur $\p^1$ v\'erifiant $\|f\otimes 1\|_{L^2,\infty}<\infty$, alors on a
\[
f\otimes 1=\sum_{k=0}^{m}a_k \bigl(1\otimes z^k\bigr)+\sum_{\nu\in \Z}\sum_{\la\in \mathcal{Z}_{m,\nu}}a_{\nu,\la}
\vf_{\nu,\la}^{(m)},
\]
dans $\mathcal{H}^{(m)}$ o\`u  $a_k=\frac{(f\otimes 1,1\otimes z^k)_{L^2,\infty}}{\|1\otimes z^k\|^2_{L^2,\infty}}$
pour $k\in \{0,1,\ldots,m\}$
et $a_{\nu,\la}=\frac{(f\otimes 1,\vf_{\nu,\la}^{(m)}\otimes 1)_{L^2,\infty}}{\|\vf_{\nu,\la}^{(m)}\otimes
1\|^2_{L^2,\infty}}$ $\forall\, \nu\in \Z,\, \forall\,\la \in \mathcal{Z}_{m,\nu}$.

Une autre mani\`ere \'equivalente est de montrer que $0$ est  l'unique \'el\'ement de $\mathcal{H}^{(m)}$ orthogonal \`a cette famille (voir \ref{base}). C-\`a-d, si:
\[
a_k=a_{\nu,\la}=0,\; \forall\, k\in \{0,1,\ldots,m\},\; \forall\, \nu\in \Z,\,\forall\,\la\in \mathcal{Z}_{m,\nu},
\]
alors,
\[
f\otimes 1=0\quad\text{dans}\quad \mathcal{H}^{(m)}.
\]

Il est  naturel de commencer  par \'etudier le terme suivant:
\[
 \bigl(f\otimes 1,\vf_{\nu,\la}^{(m)}  \bigr)_{L^2,\infty}.
\]
Observons que
\begin{align*}
\bigl(f\otimes 1,\vf_{\nu,\la}^{(m)}  \bigr)_{L^2,\infty}&=\int_0^1 \bigl(\int_0^{2\pi} f(r,\theta)e^{-i\nu\theta}d\theta\bigr)J_{\nu}(\la r)rdr\\
&+\frac{J_{\nu}(\la)}{J_{\nu-m}(\la)}\int_1^\infty \bigl(\int_0^{2\pi} f(r,\theta)e^{-i\nu\theta}d\theta \bigr)J_{\nu-m}(\frac{\la}{r})\frac{rdr}{r^{4+m}}\\
&=\bigl(e^{i\nu\theta} f_\nu \otimes 1,\vf_{\nu,\la}^{(m)}  \bigr)_{L^2,\infty},
\end{align*}
(o\`u on a not\'e par $f_\nu$ la fonction d\'efinie presque partout par $f_\nu(r,\theta):=\int_0^{2\pi} f(r,\theta)e^{-i\nu\theta}d\theta$).  Cela nous am\`ene \`a fixer le param\`etre $\nu$ et de montrer le r\'esultat suivant: Si pour tout $\la \in \mathcal{Z}_{m,\nu}$
on a $\bigl(e^{i\nu\theta}f_\nu \otimes 1,\vf_{\nu,\la}^{(m)}\otimes 1  \bigr)_{L^2,\infty}=0$ alors
 \[
 f_\nu\otimes 1=0.
 \]
 En effet, cela est \'equivalent \`a notre probl\`eme de d\'epart puisque $\bigl(f_\nu\otimes 1,\vf_{\nu',\la'}  \bigr)_{L^2,\infty}=0$  si $\nu\neq \nu'$
et qu'on dispose d'une correspondance bijective:  $f\mapsto (f_\nu)_{\nu\in \Z}$
(voir \ref{z12} pour plus de d\'etails).\\

\noindent{}{\large{\textbf{La th\'eorie des s\'eries de Fourie-Bessel g\'en\'eralis\'ee}\large} }:
Fixons $\nu\in \Z$.  On sait que  $\mathcal{Z}_{m,\nu}$ est un sous-ensemble discret de $\R^\ast$ (voir
 th\'eor\`eme
 \ref{vecteurpropreL}). On ordonne $\bigl\{\frac{\la^2}{4}\,|\, \la\in \mathcal{Z}_{m,\nu}  \bigr\}$  par
  ordre croissant, et on note par $\la_k$ le $k$-\`eme \'el\'ement de
 cet ensemble  et  par
  $\vf_{\nu,k}^{(m)}$, le vecteur $\vf_{\nu,\la}^{(m)}$ tel que $\la_k=\frac{\la^2}{4}$. Au paragraphe \ref{z9}
  nous introduisons un nouveau espace hilbertien $\mathcal{E}_m$. Nous expliquons dans
  \ref{z12}, son lien avec
  $\mathcal{H}^{(m)}$ et nous \'etablissons un r\'esultat pr\'ecis (voir th\'eor\`eme
   \ref{x17}) g\'en\'eralisant ainsi la th\'eorie
  classique des s\'eries de Fourier-Bessel. Ce th\'eor\`eme sera une cons\'equence du th\'eor\`eme \ref{x26} dont
  la preuve est une adaptation de la th\'eorie classique \`a notre situation.

  Le th\'eor\`eme \ref{x17} va nous permettre d'obtenir \ref{x16} comme un corollaire. En effet,
 soit $f$ une fonction d\'efinie presque partout sur $\p^1$, avec
$\| f\otimes 1\|_{L^2,\infty}<\infty$. Si
\[
\bigl(f\otimes 1,\vf_{\nu,\la}^{(m)} \bigr)_{L^2,\infty}= \bigl(f\otimes 1,1\otimes z^k)_{L^2,\infty}=0
\quad
\forall\, k\in \{0,1,\ldots,m\},\; \forall\, \nu\in \Z,\,\forall\,\la\in \mathcal{Z}_{m,\nu},
\]
Par les notations du \ref{z9}, en particulier \ref{x18}, nous avons:
\[
(f_\nu,
{\boldsymbol\vf}_{\nu,k}^{(m)})_m= \bigl(f_k,x^k)_{m}=0
\quad
\forall\, k\in \{0,1,\ldots,m\},\; \forall\, \nu\in \Z,\,\forall\,\la\in \mathcal{Z}_{m,\nu},
\]
Nous d\'eduisons le th\'eor\`eme  \ref{x16}, en appliquant \ref{x17}, combin\'e avec la remarque
\ref{z12}.

Afin de d\'emontrer ce r\'esultat, notre strat\'egie consiste
 \`a developper une  th\'eorie des s\'eries de Fourier-Bessel g\'en\'eralis\'ee
 adapt\'ee au cas $m\geq 1$. Nous commen\c{c}ons  par faire un rappel sur la  th\'eorie des s\'eries de Fourier-Bessel
 classique en suivant la pr\'esentation du \cite[\S. 18]{Watson}.
\subsection{Th\'eorie des s\'eries Fourier-Bessel classique}\label{z10}
Les fonctions de Bessel jouent un rôle analogue \`a celui des fonctions trigonom\'etriques dans la th\'eorie des s\'eries de Fourier. Dans cette premi\`ere partie on va rappeler la th\'eorie des s\'eries de Fourier-Bessel et on \'enoncera le r\'esultat principal de cette th\'eorie \`a savoir une condition suffisante pour qu'une fonction $f$ admet un d\'eveloppement en s\'erie en fonctions de Bessel.  On suivra la pr\'esentation faite dans \cite[\S 18]{Watson}.\\

Soit $f$ une fonction r\'eelle d\'efinie sur $[0,1]$ et on suppose que $\int_0^1|f(t)|t^\frac{1}{2}dt$ converge.
Soit $\nu$ un r\'eel sup\'erieur \`a $\frac{1}{2}$. On appelle \textit{s\'erie de fonctions de Bessel}, la fonction qui \`a tout $x$ r\'eel associe la somme suivante (lorsqu'elle converge):
\[
 \sum_{k=1}^\infty a_k J_\nu(j_kx),
\]
o\`u $(a_k)_{k\geq 1}$ est une suite de constantes ind\'ependante de $x$ et $(j_k)_{k\geq 1}$ est la suite des z\'eros
 positifs de $J_\nu$ ordonn\'ee par ordre croissant. Si les coefficients de cette s\'erie sont donn\'ees par
\[
 a_k=\frac{2}{J_{\nu+1}(j_k)^2}\int_0^1 f(t)J_\nu(j_k t)tdt,
\]
alors cette s\'erie est dite \textit{la s\'erie de Fourier-Bessel associ\'ee }\`a $f(x)$. Si, en plus, cette s\'erie converge vers $f(x)$ en tout point de l'intervalle ouvert $]0,1[$ alors on parle du \textit{d\'eveloppement de Fourier-Bessel} de $f$.

On peut aussi \'etudier la s\'erie: $ \sum_{k=1}^\infty b_k J_\nu(\la_k x)$ o\`u $\la_1,\la_2,\la_3,\ldots$
sont les z\'eros positifs non nuls de $
 z\mapsto z J_\nu'(z)+HJ_\nu(z)$ appel\'ee la \textit{s\'erie de Dini des fonctions de Bessel}.  Lorsque
  les coefficients $b_k$ sont donn\'es par la formule suivante: $
 ( (\la_k^2-\nu^2)+\la_k^2J_\nu'(\la_k)^2)b_k=2\la_k^2\int_0^1f(t)J_\nu(\la_kt)tdt$
alors on dit que c'est \textit{la s\'erie de Dini associ\'ee} \`a $f(x)$.
\begin{remarque}
\rm{Notons que Watson \'etudie le d\'eveloppement de fonctions en s\'erie de Fourier-Bessel et affirme dans \cite[p. 582]{Watson} que le cas g\'en\'eral des d\'eveloppements en s\'eries de Dini se traite de la m\^eme mani\`ere. Par analogie avec la th\'eorie des s\'eries de Fourier, il donne une condition suffisante pour l'existence du d\'eveloppement en s\'erie de Fourier-Bessel qu'on \'enoncera dans la suite et on en rappellera l'id\'ee de la preuve.}
\end{remarque}
Rappelons qu'une fonction r\'eelle $f$ sur un intervalle $[a,b]$ est dite \textit{\`a variation born\'ee} si:
\[
 V_a^b(f):=\sup_{\si \in \mathcal{S}([a,b])}V(f,\si)<\infty
\]
o\`u $\mathcal{S}([a,b]) $ est l'ensemble des subdivisions de $[a,b]$ de la forme $\si=(x_0=a,x_1,\ldots,x_n=b)$ et
 $V(f,\si)=\sum_{i=0}^{n-1}\bigl|f(x_{i+1})-f(x_i)\bigr|$.\\

D'apr\`es Watson, on dispose de quatre r\'esultats majeurs sur la th\'eorie des s\'eries de Fourier-Bessel, \`a savoir une condition suffisante pour  la convergence simple en un point int\'erieur de $[0,1]$, la convergence uniforme sur un sous-intervalle ferm\'e propre de $]0,1[$, le comportement de la s\'erie de Fourier-Bessel aux voisinages de $0$ (resp. de $1$) et enfin l'unicit\'e du d\'eveloppement en s\'erie de Fourier-Bessel. Ce dernier point s'interpr\`ete en langage moderne en disant que les fonctions de Bessel forment une une famille totale dans un espace hilbertien donn\'e.

Rappelons donc les principaux r\'esultats de la th\'eorie des s\'eries de Fourier-Bessel.\\

Le r\'esultat suivant nous fournit une condition suffisante pour  la convergence simple de la s\'erie de Fourier-Bessel:
\begin{proposition}\label{fourierbessel}
 Soit $f$ une fonction d\'efinie sur $[0,1]$ et on suppose que $\int_0^1|f(t)|t^\frac{1}{2}dt$ converge. Soit
\[
 a_k:=\frac{2}{J_{\nu+1}^2(j_k)}\int_0^1 f(t)J_\nu(j_kt)tdt,
\]
avec $\nu\geq -\frac{1}{2}$.

Soit $x\in ]a,b[$ avec $0<a<b<1$ et  que $f$ soit \`a variation born\'ee sur $[a,b]$. Alors la s\'erie
\[
 \sum_{k=1}^\infty a_k J_\nu(j_k x),
\]
  est convergente, de somme \'egale \`a $\frac{1}{2}\bigl(\underset{y\mapsto x^+}{\lim}f(y)+\underset{y\mapsto x^-}{\lim}f(y) \bigr)$\footnote{Comme $f$ est \`a variation born\'ee alors  on peut montrer (par l'absurde) que $f$ admet une limite \`a droite et \`a gauche de $x$}.
\end{proposition}
\begin{proof}
 Voir \cite[\S 18.24]{Watson}.
\end{proof}

En gardant les m\^emes hypoth\`eses et si l'on suppose en plus que $f$ est continue sur $[a,b]$ alors on a convergence uniforme. Pr\'ecis\'ement, on a le r\'esultat suivant
\begin{proposition}\label{fourierbessell}
 Si $f$ est continue et v\'erifie les hypoth\`eses de la proposition pr\'ec\'edente, alors la s\'erie de Fourier-Bessel
  associ\'ee \`a $f$ converge  uniform\'ement vers $f$ sur l'intervalle $[a+\delta,b-\delta]$, o\`u $\delta$ est un r\'eel
  positif non nul arbitraire.
\end{proposition}
\begin{proof}
Voir \cite[\S 18.25]{Watson}.
\end{proof}

Comme les sommes partielles de la s\'erie de Fourier-Bessel associ\'ee \`a une fonction $f$ continue sont nulles  en $x=1$.
Il est donc n\'ecessaire de supposer que  $\lim_{x\mapsto 1^-}f(x)=0 $ si l'on veut garantir la convergence uniforme
de la s\'erie de Fourier-Bessel au voisinage de $x=1$.  On montre dans \cite[\S 18.26]{Watson} que la continuit\'e de $f$ sur $[a,1]$, avec $a\geq 0$ et la condition $f(1)=0$ suffissent pour assurer la convergence uniforme sur un intervalle de la forme $[a+\delta,1]$ pour $\delta>0$.\\

Le dernier r\'esultat fondamental de la th\'eorie des s\'eries de Fourier-Bessel expos\'ee dans \cite[\S 18]{Watson} est l'unicit\'e du d\'eveloppement en s\'erie de Fourier-Bessel. On montre dans \cite[\S 18.6]{Watson}, sous la condition de la convergence absolue de $\int_0^1 t^\frac{1}{2}f(t)dt$, que $f$ est nulle presque partout si et seulement si tous les coefficients de la s\'erie de Fourier-Bessel associ\'ee  sont nuls. \\

Le point  cl\'e de la th\'eorie des s\'eries de Fourier-Bessel classique r\'eside en   l'\'etude de la suite de fonctions
suivante:
\[
 (x,t)(\in [0,1]^2)\longmapsto T_n(t,x)=\sum_{k=1}^n\frac{2J_\nu(j_k x)J_\nu(j_k t)}{J_{\nu+1}(j_k)^2}.
\]
Notons d'abord que:
\[
 \sum_{k=1}^n a_k J_\nu(j_k x)=\int_0^1 f(t)T_n(t,x)tdt\quad \forall\, n\in \N_{\geq 1}.
\]
 Les fonction $T_n$ jouent le m\^eme rôle que
\[
 (x,t)\mapsto \frac{\sin\bigl((n+\frac{1}{2})(x-t)\bigr)}{\sin(x-t)},\quad x\neq t
\]
pour la th\'eorie des s\'eries de Fourier.\\


Il est commode d'exprimer $T_n$ comme \'etant l'int\'egrale sur le contour d'un rectangle contenant seulement
$j_1,\ldots,j_n$ d'une fonction auxiliaire  m\'eromorphe dont les  uniques pôles sont les z\'eros de $J_\nu$ et ayant
 un comportement asymptotique ad\'equat. Ce dernier point   nous permettra d'\'etudier $T_n$ pour $n$ assez grand et
 de d\'eduire la plupart des r\'esultats de la th\'eorie classique.\\

Nous allons rappeler quelques points techniques du \cite[\S 18]{Watson} dont le but
de les adapter et les appliquer ult\'erieurement \`a notre situation. Dans les pages pp. 583-584, pour tous $0<t\neq x<1$, on introduit la fonction $g$ donn\'ee par
\[
 g(w)=\frac{2w}{t^2-x^2}\bigl( tJ_\nu(xw)J_{\nu+1}(tw)-xJ_\nu(tw)J_{\nu+1}(xw) \bigr),\quad \forall\, w\in \C,
\]
et on montre que le r\'esidu de $
 w\mapsto \frac{g(w)}{wJ_\nu (w)^2}$
en $w=j_k$ est \'egal \`a $
 2\frac{J_\nu(xj_k)J_\nu(tj_k)}{J_{\nu+1}(j_k)^2}\footnote{On commence par noter que le r\'esidu de $ w\mapsto \frac{g(w)}{wJ_\nu (w)^2}$ en $j_k$ est $ \frac{g'(j_k)}{j_k J_\nu^{'2}(j_k)}$. Cela r\'esulte du fait suivant: Si $\vf_1$ et $\vf_2$ sont deux fonctions holomorphes sur $\C$ et  $\la$ est un z\'ero simple de $\vf_2$ alors on peut montrer \`a l'aide d'un d\'eveloppement limit\'e convenable au voisinage de $\la$ que
\[
 \mathrm{Res}_j\Bigl(w\mapsto \frac{\vf_1(w)}{w\vf_2^2(w)}\Bigr)=\Bigl( \frac{\vf'_1(j)}{j \vf_2'(j)^2}-\frac{\vf_1(j)}{ \vf_2'(j)^2}\bigl( 1+j \frac{\vf_2''(j)}{\vf_2'(j)} \bigr) \Bigr).
\]
Mais on a ici $J_\nu(j)=0$ et $J_\nu$ v\'erifie $J_\nu''(z)+\frac{1}{z}J_\nu'(z)+(1-\frac{\nu^2}{z^2})J_\nu(z)=0$, donc $1+j \frac{J_\nu'(j)}{J_\nu''(j)}=0$.}.$

D'un autre cot\'e, on a $\forall\, w\in \C$:
\[
\begin{split}
 \frac{d}{dw}\bigl(w J_\nu (xw)&J_{\nu+1}(tw) \bigr)=\frac{d}{dw}\bigl((xw)^{-\nu}J_\nu(tw) (tw)^{\nu+1}J_{\nu+1}(tw)  \bigr)\frac{x^\nu}{t^{\nu+1}}\\
&=-\frac{x^{\nu+1}}{t^{\nu+1}}(xw)^{-\nu}J_{\nu+1}(xw)(tw)^{\nu+1}J_{\nu+1}(tw)+\frac{x^\nu}{t^\nu}(xw)^{-\nu}J_\nu(xw)(tw)^{\nu+1} J_\nu(tw)\;\;\text{par}\; \ref{B1}\,\text{et}\,\ref{B5}\\
&=-xw J_{\nu+1}(xw)J_{\nu+1}(tw)+twJ_\nu(xw)J_\nu(tw),
\end{split}
\]
et par sym\'etrie, on en d\'eduit que
\[
 \frac{d}{dw}\bigl(w J_\nu (tw)J_{\nu+1}(xw) \bigr)=-tw J_{\nu+1}(tw)J_{\nu+1}(xw)+xwJ_\nu(tw)J_\nu(xw),\quad \forall\, w\in \C.
\]
En regroupant tout cela on obtient:
\[
 g'(w)=2wJ_\nu(xw)J_\nu(tw),\quad\forall\, w\in \C.
\]
On conclut que le r\'esidu de $w\mapsto \frac{g(w)}{wJ_\nu(w)^2}$ en $j_k$ est
\[
 2\frac{J_\nu(xj_k)J_\nu(tj_k)}{J_{\nu+1}(j_k)^2}.
\]
On consid\`ere maintenant le rectangle de sommets $iB$, $-iB$, $A_n+iB$ et $A_n-iB$ o\`u $B>0$ et $A_n$ est un r\'eel tel que $j_n<A_n<j_{n+1}$. Remarquons enfin que $w\mapsto  \frac{g(w)}{wJ_\nu(w)^2}$ est d'ordre 1 en $w=0$ (on le v\'erifie en utilisant \ref{ordreenzero} ). Donc par cette discussion on peut affirmer que l'int\'egrale de $w\mapsto \frac{g(w)}{wJ_\nu(w)^2}$  le long de ce contour est \'egale \`a $T_n(t,x)$. \\

La suite   utilise de mani\`ere cruciale les propri\'et\'es asymptotiques des fonctions de Bessel. En fait, le d\'eveloppement asymptotique des fonctions de Bessel, voir \cite[\S 7.21]{Watson}, nous permet d'assurer que l'int\'egrale sur les cot\'es inf\'erieurs et sup\'erieurs du rectangle s'annulent si $B$ tends vers l'infini, aussi puisque les int\'egrands sont impairs et sans singularit\'es  sur $[-Bi,Bi]$ donc leur int\'egrale sur ce segment  est nulle.   On conclut qu'on a
\[
 T_n(t,x)=\frac{1}{2\pi i}\int_{A_n-\infty i}^{A_n+\infty i}\frac{g(w)}{wJ_\nu(w)^2}dw,\quad \forall\,\, 0<x+t<2,\, x\neq t.
\]
En utilisant les deux in\'egalit\'es suivantes:
\begin{equation}\label{besselineqtriv}
 \bigl|J_\nu (\theta w) \bigr|\leq c_1\frac{\exp(|Im(\theta w)|)}{\sqrt{|\theta w|}},\quad  \bigl|J_\nu ( w) \bigr|\geq c_2\frac{\exp(|Im(w)|)}{\sqrt{| w|}},
\end{equation}

(ils d\'ecoulent  de l'expression du d\'eveloppement asymptotique donn\'ee dans \cite[\S 7.21]{Watson}), avec $\theta$ est un r\'eel positif non nul fix\'e, on d\'eduit la plupart des in\'egalit\'es \'enonc\'ees dans la suite.

\subsection{Th\'eorie des s\'eries de Fourier-Bessel g\'en\'eralis\'ee}\label{z9}

Soit $m$ un entier positif. Soit $E$ le sous-espace vectoriel des fonctions complexes sur $\R^+$ d\'efini comme suit: $f\in E$ si les fonctions
suivantes
$x\in [0,1]\mapsto f(x)$ et $x\in ]0,1]\mapsto f(\frac{1}{x})$ sont continues et born\'ees. On munit $E$  de la norme
scalaire $\vc_m$
suivante:
\[
\|f\|^2_m:=\int_0^1 |f(x)|^2xdx+\int_1^\infty |f(x)|^2\frac{dx}{x^{3+2m}}\quad \forall\, f\in E,
\]
et nous notons  par $\mathcal{E}_m$ la compl\'etion de $E$ pour cette norme. Nous avons
alors les propri\'et\'es suivantes:
\begin{proposition}\label{lespaceE}
On a,
\begin{enumerate}
\item  $\mathcal{E}_0$ est isom\'etrique \`a $L^2([0,1])\oplus L^2([0,1])$\,{}\footnote{ Ici, $L^2([0,1])$ d\'esigne  l'espace des fonctions complexes carr\'e int\'egrables pour la mesure $x\,dx$ sur $[0,1]$. Cet espace
joue un
rôle important dans la th\'eorie des s\'eries de Fourier-Bessel.}.
\item $(\mathcal{E}_m)_{m\in \N}$ forme une suite d'espaces hilbertiens strictement croissante pour l'inclusion.
\item Soit $\nu\in \Z$. On consid\`ere  la fonction  $P_\nu$ d\'efinie presque partout sur $\R^+$ comme suit: $P_\nu(x)=x^{\nu}$ pour
tout $x\in \R^{+\ast}$. Alors $P_\nu\in \mathcal{E}_m$ si et seulement si $ \nu\in \{0,1,\ldots, m\}$.
\end{enumerate}
\end{proposition}
\begin{proof}

\begin{enumerate}
\item
On consid\`ere
l'application suivante
$\Phi:f(\in E)\mapsto \Phi(f)=(f_1,f_2)$ o\`u $f_1$ et $f_2$ sont deux fonctions d\'efinies presque partout sur $[0,1]$
donn\'ees par $f_1(x)=f(x)$ et
$f_2(x)=f(\frac{1}{x})$ pour  tout $x\in ]0,1]$. $\Phi$ d\'efinit alors une isom\'etrie entre   $\mathcal{E}_0$ et $L^2([0,1])^{\oplus
2}$.

\item Soit $m<m'\in \N$. Donc $\vc_{m'}\leq \vc_m$, par cons\'equent
$\mathcal{E}_m\subseteq \mathcal{E}_{m'}$. Soit $\al$ un
 r\'eel strictement positif et
 $g_\al$ l'\'el\'ement de $E$ d\'efini par: $g(x)=0$ si $x\in [0,1]$ et $g(x)=x^{-\al}$ si $x\geq 1$. En choisissant
$m<\al<m'$,
alors on v\'erifie $g_\al\in \mathcal{E}_{m'}\setminus \mathcal{E}_m$.
\item

Soit $\nu\in\Z$ et  on consid\`ere $P_\nu$, la fonction d\'efinie presque partout sur $\R^+$, par $P_\nu(x)=x^\nu$ pour tout $x\neq 0$.
On a
$\|P_\nu\|_m^2=\int_0^1x^{2\nu+1}dx+\int_1^\infty x^{2\nu-3-2m}dx$. Cette quantit\'e est finie si et seulement si
$\nu\in \{0,\ldots, m\}$.
\end{enumerate}

\end{proof}

\begin{remarque}\label{z12}\rm{ Soit $\xi \in A^{(0,0)}(\p^1,\mathcal{O}(m))$. Il existe une fonction complexe $g$ sur $\p^1$ telle que
$\xi=g\otimes 1$\footnote{Comme $\mathcal{O}(m)$ est engendr\'e par ses sections globales, alors par une partition
de l'unit\'e on peut \'ecrire $\xi=\sum_{j=0}^m g_j\otimes z^j$ o\`u $g_j\in A^{(0,0)}(\p^1)$. On pose alors
$g=\sum_{j=0}^m g_j z^j$.}. On v\'erifie que
\begin{align*}
\|\xi\|_{L^2,\infty}^2&=\frac{i}{2\pi}\int_{|z|\leq 1}|g(z)|^2 dz\wedge
d\z+\frac{i}{2\pi}\int_{|z|\geq 1}|g(z)|^2\frac{dz\wedge d\z}{|z|^{4+2m}}\\
&=2\int_0^1 \frac{1}{2\pi}\bigl(\int_0^{2\pi} |g(xe^{i\theta})|^2d\theta\bigr)xdx+2\int_1^\infty
 \frac{1}{2\pi}\bigl(\int_0^{2\pi} |g(xe^{i\theta})|^2d\theta\bigr)\frac{dx}{x^{3+2m}}\\
&=2\sum_{\nu\in \Z}\int_0^1|g_\nu(x)|^2 xdx+ 2\sum_{\nu\in \Z}\int_1^\infty  |g_\nu(x)|^2\frac{dx}{x^{3+2m}}\\
&=2\sum_{\nu\in \Z}\|g_\nu\|_m^2,
\end{align*}
La 3\`eme \'egalit\'e  r\'esulte de la th\'eorie des s\'eries de Fourier appliqu\'ee \`a la fonction $\theta\mapsto g(x e^{i\theta})$ qui est carr\'e
int\'egrable pour tout $x\neq 0$ (Ce dernier point provient du fait que  $g$ est continue sur
$\p^1\setminus\{0,\infty\}$). On peut alors \'ecrire que:
\begin{equation}\label{z11}
\|g\otimes 1\|_{L^2,\infty}^2=2\sum_{\nu\in \Z}\|g_\nu\|_m^2.
\end{equation}
Donc si $g\otimes 1\in \mathcal{H}_m$ alors $g_\nu\in \mathcal{E}_m$ pour tout $\nu\in \Z$. R\'eciproquement,
si l'on se donne $(g_\nu)_{\nu\in \Z}$ une suite d'\'el\'ements de $\mathcal{E}_m$ telle que
$\sum_{\nu\in \Z}\|g_\nu\|_m^2<\infty$, alors cette suite d\'efinit
un \'el\'ement de $\mathcal{H}_m$.}
\end{remarque}

Fixons $\nu\in \Z$.  Comme $\mathcal{Z}_{m,\nu}$ est un sous-ensemble discret de $\R^\ast$ alors on peut ordonner l'ensemble $\bigl\{\frac{\la^2}{4}\,|\, \la\in \mathcal{Z}_{m,\nu}  \bigr\}$ et on note par $\la_k$ la $k$-\`eme
valeur de cet ensemble pour l'ordre croissant.  On pose  ${\boldsymbol\vf}_{\nu,k}^{(m)}$ la fonction sur $\R^+$ d\'efinie comme
suit:
  \begin{equation}
{\boldsymbol\vf}_{\nu,k}^{(m)}(x) =
\begin{cases}
 J_n(\la_k  x)  & \text{si } 0\leq x \leq 1,\\
\frac{J_n(\la_k)}{J_{n-m}(\la_k)}x^{m}J_{n-m}(\frac{\la_k}{x}) & \text{si } x>1.
\end{cases}
\end{equation}
On v\'erifie que:
\begin{equation}\label{x18}
{\boldsymbol\vf}_{\nu,k}^{(m)}\in \mathcal{E}_{m},\,\,
\|\vf_{\nu,k}^{(m)}\|^2_{L^2,\infty}=\|{\boldsymbol\vf}_{\nu,k}^{(m)}\|^2_m\quad \text{et}\quad (f\otimes 1,
\vf_{\nu,k}^{(m)})_{L^2,\infty}=(f_\nu,
{\boldsymbol\vf}_{\nu,k}^{(m)})_m\, \forall \,\,f\otimes 1\in \mathcal{H}^{(m)}.
\end{equation}
 Le th\'eor\`eme \ref{x16} sera un corollaire
d'un  r\'esultat plus  pr\'ecis, \`a savoir le  le th\'eor\`eme suivant:
\begin{theorem}\label{x17}
En gardant les m\^emes notations et hypoth\`eses, la famille suivante:
\[
\bigl\{ {\boldsymbol\vf}_{\nu,k}^{(m)}|\, k\in \N \bigr\}\quad\bigl(\text{resp.}\,\bigl\{1,x,\ldots,x^m  \bigr\}\cup\bigl\{ {\boldsymbol\vf}_{\nu,k}^{(m)}|\, k\in \N \bigr\}\bigr),
\]
forme
 une base hilbertienne dans $\mathcal{E}_m$ lorsque $\nu\leq -1$ ou $\nu\geq m+1$ (resp. si $\nu \in \{0,1,\ldots,m\}$).
\end{theorem}

\begin{remarque}\label{z13}
Notons que le cas $m=0$ est \'equivalent \`a la th\'eorie classique des s\'eries de Fourier-Bessel. Par cela,
on veut dire que  la  famille $\{1,  J_\nu(\la r))|\,  J_\nu'(\la)=0\}$ (resp. $\{ J_\nu(\la r)|\,
J_\nu(\la)=0\}$) forme un syst\`eme total pour l'espace des fonctions $L^2$
 sur $[0,1]$ v\'erifiant la condition de Neumann  (resp. la condition de Dirichlet)
  (voir la preuve du \ref{x15}).
\end{remarque}
Afin d'\'etablir le th\'eor\`eme \ref{x17}, alors il suffira de d\'emontrer le  th\'eor\`eme ci-dessous:
\begin{theorem}\label{x26}
Soient $m\in \N_{\geq 1}$, $\nu\in \Z$ et  $p,q$ deux entiers $\geq |\nu|$. Il existe deux suites r\'eelles
$(l_k)_{k\in \N}$ et $(l'_k)_{k\in \N}$  telles que pour tout $f\in \mathcal{E}_m$, on a:
\begin{align*}
\delta_\nu(f,x^\nu)_m&=\sum_{k=1}^\infty   l'_k
\frac{(f,{\boldsymbol\vf}_{\nu,k}^{(m)})_{m}}{2\|{\boldsymbol\vf}_{\nu,k}^{(m)}\|_{m}^2}-\int_0^1\bigl(x^{2p+2m-\nu+1}
-x^{2q+2m-\nu+1}\bigr)f(\frac{1}{x})dx \\
\delta_\nu(f,x^\nu)_m&= \sum_{k=1}^\infty   l_k \frac{ (f
,{\boldsymbol\vf}_{\nu,k}^{(m)})_{m}}{2\|{\boldsymbol\vf}_{\nu,k}^{(m)}\|_{m}^2}
-\int_0^1\bigl(x^{2p+\nu+1}-x^{2q+\nu+1}\bigr)f(x)dx,
\end{align*}
o\`u $\delta_\nu=0$ si $\nu\leq -1$ ou $\nu\geq m+1$.
\end{theorem}

En effet, si  l'on fixe $m\in \N$ et $\nu\in \Z$, alors si
 $f\in \mathcal{E}_m$ tel que:
\[
 (f, x^k)_m=(f,{\boldsymbol\vf}_{{\nu,\la}}^{(m)} )_m=0\quad \forall\, k\in\{0,1,\ldots,m\},\;  \forall\, \la\in
 \mathcal{Z}_{m,\nu},
\]
alors par ce  dernier th\'eor\`eme \ref{x26}, nous d\'eduisons  que:
\begin{align*}
\int_0^1\bigl(x^{2p+2m-\nu+1}-x^{2q+2m-\nu+1}\bigr)f(\frac{1}{x})dx&=\int_0^1\bigl(x^{2p+\nu+1}-x^{2q+\nu+1}\bigr)
f(x)dx=0,
\end{align*}
et cela pour tous $p$ et $q$ assez grand (pour $\nu$ fix\'e). Un th\'eor\`eme classique d'analyse nous permet alors
de conclure que
les deux fonctions d\'efinies presque partout sur $[0,1]$: $x\mapsto f(\frac{1}{x})$ et $x\mapsto f(x)$ sont
 nulles presque partout sur $[0,1]$, ce qui est \'equivalent \`a:
 \[f=0,\]
 dans $\mathcal{E}_m$.\\

La suite a pour but d'\'etablir le th\'eor\`eme \ref{x26} ci-dessus.
Soit $f\in \mathcal{E}_{m}$ une fonction sur $\R$. Nous  consid\`erons la somme partielle suivante:
\[
S_n(x):= \sum_{k=1}^n \frac{({\boldsymbol\vf}_{\nu,k}^{(m)},f)_{m}}{\|{\boldsymbol\vf}_{\nu,k}^{(m)}\|^2_{m}}{\boldsymbol\vf}_{\nu,k}^{(m)}(x)
\quad
\forall\, x\in \R,\;n\in \N^\ast,
\]

On pose pour tous $0<t<1$ et $0<x<1$:
\begin{align*}
 \kappa_n^+(t,x)&:=\sum_{k=1}^n\frac{J_{\nu}(t\la_k)J_{\nu}(x\la_k)}{2\|{\boldsymbol\vf}_{\nu,k}^{(m)}\|_{m}^2},\\
\kappa_n^-(t,x)&:=\sum_{k=1}^n\frac{J_{\nu}(\la_k)}{J_{\nu-m}(\la_k)}\frac{J_{\nu}(x\la_k)J_{\nu-m}(\la_k t)}{\|{\boldsymbol\vf}_{\nu,k}^{(m)}\|^2_{m}},\\
\tau_n^+(t,x)&:=\sum_{k=1}^n\frac{J_{\nu}(\la_k)^2}{J_{\nu-m}(\la_k)^2}\frac{J_{\nu-m}(t\la_k)J_{\nu-m}(x\la_k)}{\|{\boldsymbol\vf}_{\nu,k}^{(m)}\|_{m}^2},\\
\tau_n^-(t,x)&:=\sum_{k=1}^n\frac{J_{\nu}(\la_k)}{J_{\nu-m}(\la_k)}\frac{J_{\nu-m}(\la_k x)J_{\nu}(\la_k t)}{\|{\boldsymbol\vf}_{\nu,k}^{(m)}\|_{m}^2}.
\end{align*}
On v\'erifie que
\[
 S_n(x)=\int_0^1 f(t)\kappa_n^+\bigl( t,x\bigr)  tdt+\int_1^\infty \frac{f(t)}{t^m}\kappa_n^-\bigl(\frac{1}{t},x\bigr)\frac{tdt}{t^4},\quad \text{si}\; 0\leq x\leq 1,
\]
et
\[
 S_n(x)=\int_0^1 f(t)x^m\tau_n^-\bigl(t,\frac{1}{x}\bigr)tdt+\int_1^\infty \frac{f(t)}{t^m}x^m\tau_n^-\bigl(\frac{1}{t},\frac{1}{x}\bigr)\frac{tdt}{t^4},\quad \text{si}\;  x\geq 1,
\]

Donc, il est naturel d'\'etudier les fonctions $\kappa^+_n,\kappa^-_n,\tau^+_n$ et $\tau_n^-$. Pour cela nous allons
 introduire les quatres fonctions m\'eromorphes sur $\C$ suivantes:
\begin{align*}
 t_{+,\nu}( w)&=\frac{J_{\nu}(tw)J_{\nu}(xw)J_{\nu-m}(w)}{J_{\nu}(w)L_{m,\nu}(w)},\\
t_{-,\nu}(w)&=\frac{J_{\nu}(xw)J_{\nu-m}(wt)}{L_{m,\nu}(w)},\\
 s_{+,\nu}(w)&=\frac{J_{\nu}(w)}{J_{\nu-m}(w)}\frac{J_{\nu-m}(xw)J_{\nu-m}(tw)}{L_{m,\nu}(w)},\\
s_{-,\nu}(w)&=\frac{J_{\nu-m}(xw)J_{\nu}(t w)}{L_{m,\nu}(w)}.
\end{align*}
o\`u $0<t<1$ et $0<x<1$.

Si l'on note par $A_n$ un r\'eel compris strictement entre $\la_n$ et $\la_{n+1}$ deux \'el\'ements
cons\'ecutifs de $\mathcal{Z}_{m,\nu}$. Alors nous avons  le th\'eor\`eme suivant:

\begin{theorem}\label{ttss}
 Soit $\nu\in \Z$.
\begin{enumerate}
\item
 Si $\nu\leq -1$ ou $\nu\geq m+1$, alors on a pour tout $n\gg1$:
 \begin{equation}\label{q1}
\frac{i}{2\pi }\int_{A_n-\infty i}^{A_n+\infty i}t_{+,\nu}(w)dw=\kappa_n^+(t,x)-\sum_{\substack{\al|J_{\nu}(\al)=0\\
0<\al<A_n}}\frac{J_{\nu}(t\al)J_{\nu}(x\al)}{J_{\nu+1}(\al)^2},
\end{equation}

\begin{equation}\label{q2}
 \frac{i}{2\pi}\int_{A_n-\infty i}^{A_n+\infty i}t_{-,\nu}(w)dw=\kappa_n^-(t,x),
 \end{equation}

 \begin{equation}\label{q3}
 \frac{i}{2\pi}\int_{A_n-\infty i}^{A_n+\infty i}s_{+,\nu}(w)dw
 =\tau_n^+(t,x)-\sum_{\substack{\gamma|J_{\nu-m}(\al)=0\\
0<\gamma<A_n}}\frac{J_{\nu-m}(t\gamma)J_{\nu-m}(x\gamma)}{J_{\nu-m+1}(\gamma)^2},
\end{equation}

\begin{equation}\label{q4}
 \frac{i}{2\pi}\int_{A_n-\infty i}^{A_n+\infty i}s_{-,\nu}(w)dw =\tau_n^-(t,x).
\end{equation}
\item
Si $0\leq \nu \leq m$, on a pour tout $n\gg1$:
\begin{align*}
\frac{i}{2\pi }\int_{A_n-\infty i}^{A_n+\infty i}t_{+,\nu}(w)dw&=-2\frac{(m-\nu+1)(\nu+1)}{(m+2)}t^\nu x^\nu+ \kappa_n^+(t,x)-\sum_{\substack{\al|J_{\nu}(\al)=0\\
0<\al<A_n}}\frac{J_{\nu}(t\al)J_{\nu}(x\al)}{J_{\nu+1}(\al)^2},\\
 \frac{i}{2\pi}\int_{A_n-\infty i}^{A_n+\infty i}t_{-,\nu}(w)dw&=-2\frac{(m-\nu+1)(\nu+1)}{(m+2)}t^{m-\nu} x^\nu+\kappa_n^-(t,x),\\
 \frac{i}{2\pi}\int_{A_n-\infty i}^{A_n+\infty i}s_{+,\nu}(w)dw&=-2\frac{(m-\nu+1)(\nu+1)}{(m+2)}t^{m-\nu} x^{m-\nu}+ \tau_n^+(t,x)-\sum_{\substack{\al|J_{\nu}(\al)=0\\
0<\al<A_n}}\frac{J_{\nu-m}(t\gamma)J_{\nu-m}(x\gamma)}{J_{\nu-m+1}(\gamma)^2}, \\
 \frac{i}{2\pi}\int_{A_n-\infty i}^{A_n+\infty i}s_{-,\nu}(w)dw &=-2\frac{(m-\nu+1)(\nu+1)}{(m+2)}t^\nu x^{m-\nu}+\tau_n^-(t,x).
\end{align*}





\end{enumerate}
\end{theorem}
Afin de d\'emontrer ce th\'eor\`eme \ref{ttss}, nous commen\c{c}ons par calculer les r\'esidus des fonctions $t_{+,\nu},t_{-,\nu},s_{+,\nu}
$ et $s_{-,\nu}$ aux pôles respectifs et nous repr\'esenterons  $\kappa^+(t,x),\kappa^-(t,x),\tau^+(t,x)$ et $\tau^-(t,x)
$ comme int\'egrales des quatre premi\`eres fonctions le long d'un chemin pr\'ecis. Nous aurons besoin de quelques
estimations sur $L_{m,\nu}$ du type:  \begin{equation}\label{LLLLL}
 \bigl|L_{m,\nu}(\theta w)\bigr|\leq\frac{c_\nu}{(\theta|w|)}\exp\bigl(2\theta \bigl|Im(w)\bigr|\bigr),\quad
 \bigl|L_{m,\nu}(w)\bigr|\geq\frac{c'_\nu}{|w|}\exp\bigl(2\bigl|Im(w)\bigr|\bigr),
\end{equation}
avec  $c_\nu$ et $c_\nu'$ deux constantes r\'eelles positives non nulles et $w$ appartient \`a un ouvert contenant la
droite verticale d'abscisse $A_n$ (avec $0<\theta<1$). Cela sera une cons\'equence  du lemme suivant:
\begin{lemma}
 On note par $I_\nu$ la fonction de Bessel modifi\'ee d'ordre $\nu$ et  on pose
\[
 G_\nu(z):=I_{\nu+1}(z)I_{\nu-m}(z)+I_\nu(z)I_{\nu-m-1}(z)\quad \forall\, z\in \C.
\]
\begin{enumerate}
\item
On a pour tout $z\in \C$ et tout $\nu\in \Z$
\begin{equation}
G_{-\nu}(z)=G_{m+\nu}(z),
\end{equation}
\begin{equation}\label{LLLLL1}
L_{m,\nu}(iz)=i^{2\nu-m-1} G_\nu(z).
\end{equation}

\item
Pour $z$ assez large avec $|\arg(z)|<\frac{\pi}{2}$
, on a
\begin{equation}\label{LLLLL2}
G_{\nu+m}(z)=\frac{e^{2z}}{2\pi z}\Bigl(1-\frac{4\nu^2+2m^2+4\nu m+2m+1}{2z}+O(\frac{1}{z^2})\Bigr)\quad \text{si}\; \nu\in \N^\ast,
\end{equation}

et
\begin{equation}\label{LLLLL3}
 G_\nu(z)=\frac{e^{2z}}{2\pi z}\Bigl(2-\frac{2\nu^2+m^2-2\nu m+m}{z}+O(\frac{1}{z^2})\Bigr)\quad \text{si}\;  \nu\in
  \{0,1,\ldots,m\}.
\end{equation}

\end{enumerate}

\end{lemma}

\begin{proof}
\leavevmode
 \begin{enumerate}
 \item
En utilisant les propri\'et\'es de recurrence  des fonctions de Bessel, nous avons:
\[
 \begin{split}
  G_\nu(z)&=I_{\nu+1}(z)I_{\nu-m}(z)+I_\nu(z)I_{\nu-m-1}(z)\\
&=I_{-\nu-1}(z)I_{-\nu+m}(z)+I_{-\nu}(z)I_{-\nu+m+1}(z),\quad \text{puisque} \, I_{-n}=I_n, \,
cf.\cite[9.6.6
]{Table2}\\
&=I_{(m-\nu)+1}(z)I_{(m-\nu)-m}(z)+I_{(m-\nu)}(z)I_{(m-\nu)-m-1}(z)\\
&=G_{m-\nu}(z).
 \end{split}
\]
Quant \`a l'identit\'e \ref{LLLLL1}, elle se d\'eduit ais\'ement de l'expression de $G_\nu$ et de la relation
$I_n(iz)=i^n I_n(z)$.
\item
Soit $n\in \N$ fix\'e. On a pour $z$ assez grand avec $|\arg(z)|<\frac{\pi}{2}$:
\begin{equation}\label{developpementz1}
I_n(z)=\frac{e^z}{\sqrt{2\pi z}}\bigl(1-\frac{4n^2-1}{8z}+O(\frac{1}{z^2}) \bigr)\quad   (\text{voir par exemple}\,  \cite[9.7.1]{Table2})
\end{equation}
Soit $\nu\in \N$.
 En utilisant
 \ref{developpementz1}, nous avons:
\begin{enumerate}
\item
\[
G_{\nu+m}(z)=\frac{e^{2z}}{2\pi z}\Bigl(1-\frac{4\nu^2+2m^2+4\nu m+2m+1}{2z}+O(\frac{1}{z^2})\Bigr)\quad \forall\,
|z|\gg 1,\, |\arg(z)|<\frac{\pi}{2}.
\]
\item
Si $ \nu\in \{0,1,\ldots,m \}$, on \'ecrit $
 G_\nu(z)=I_{\nu+1}(z)I_{m-\nu}+I_\nu(z)I_{m+1-\nu}(z)$
(on a utilis\'e que $I_{-n}=I_n$). Alors   on v\'erifie que:
\[
 G_\nu(z)=\frac{e^{2z}}{2\pi z}\Bigl(2-\frac{2\nu^2+m^2-2\nu m+m}{z}+O(\frac{1}{z^2})\Bigr)\quad \forall\,
|z|\gg 1,\, |\arg(z)|<\frac{\pi}{2}.\]
\end{enumerate}
\end{enumerate}
\end{proof}
Maintenant on peut d\'eduire les estimations \ref{LLLLL} en utilisant \ref{LLLLL1}, \ref{LLLLL2}, \ref{LLLLL3} et
\ref{Lparite}.

  Nous allons dresser la liste de tous  les r\'esidus possibles  des
fonctions $t_{+,\nu}$, $t_{-,\nu}$, $s_{+,\nu}$ et $s_{-,\nu}$ en leurs \'eventuels pôles en fonction de $\nu$ (voir $1.$,
$2.$ et $3.$ ci-dessous):
\begin{enumerate}
\item
Soit $\nu\in \Z$. Afin de  simplifier les notations,  on note par $\la$ un z\'ero positif de $L_{m,\nu}$,  par $\al$ un z\'ero positif non
nul de $J_\nu$ et par $\gamma$ un z\'ero positif non nul de $J_{\nu-m}$. On sait que ces z\'eros sont simples et on
v\'erifie que
 $L_{m,\nu}(\al)=-J_{\nu-m}(\al)J_{\nu}'(\al)$ et
 $L_{m,\nu}(\gamma)=-J_{\nu}(\gamma)J_{\nu-m}'(\gamma)$. Par cons\'equent:
 {\allowdisplaybreaks
\begin{equation}\label{x20}
\begin{split}
\mathrm{Res}_\la(t_{+,\nu})&=\frac{J_{\nu}(t\la)J_{\nu}(x\la)J_{\nu-m}(\la)}{J_{\nu}(\la)L_{m,\nu}'(\la)}
=\frac{J_{\nu}(t\la)J_{\nu}(x\la)}{2\|{\boldsymbol\vf}_{\nu,\la}^{(m)}\|_{m}^2} \quad(\text{on a utilis\'e
} \, \ref{deriveeLnorme} \, \text{et}\, \ref{x18}),\\
\mathrm{Res}_\al(t_{+,\nu})&=\frac{J_{\nu}(x\al)J_{\nu}(t\al)J_{\nu-m}(\al)}{J_{\nu}'(\al)L_{m,\nu}(\al)}
=-\frac{J_{\nu}(t\al)J_{\nu}(x\al)}{J_{\nu}'(\al)^2},\\
 \mathrm{Res}_\la\bigl( t_{-,\nu}
 \bigr)&=\frac{J_{\nu}(x\la)J_{\nu-m}(t\la)}{L_{m,\nu}'(\la)}=\frac{J_{\nu}(\la)}{J_{\nu-m}(\la)}
 \frac{J_{\nu}(x\la)J_{\nu-m}(\la t)}{\|{\boldsymbol\vf}^{(m)}_{\nu,\la}\|^2_{m}},\\
\mathrm{Res}_\la(s_{+,\nu})&=\frac{J_{\nu}(\la)}{J_{\nu-m}(\la)}\frac{J_{\nu-m}(t\la)J_{\nu-m}(x\la)}{L_{m,\nu}'(\la)}
=\frac{J_{\nu}(\la)^2}{J_{\nu-m}(\la)^2}\frac{J_{\nu-m}(t\la)J_{\nu-m}(x\la)}{\|{\boldsymbol\vf}^{(m)}_{\nu,\la}\|_{m}^2}\\
\mathrm{Res}_\gamma(s_{+,\nu})&=-\frac{J_{\nu}(\gamma)}{J_{\nu-m}'(\gamma)^2}
\frac{J_{\nu-m}(t\gamma)J_{\nu-m}(x\gamma)}{J_{\nu}(\gamma)}
=-\frac{J_{\nu-m}(t\gamma)J_{\nu-m}(x\gamma)}{J_{\nu-m+1}(\gamma)^2},\\
\mathrm{Res}_\la(s_{-,\nu})&=\frac{J_\nu(\la)}{J_{\nu-m}(\la)}\frac{J_{\nu-m}(\la x)J_{\nu}(\la t
 )}{L'_\nu(w)}=\frac{J_{\nu}(\la)}{J_{\nu-m}(\la)}\frac{J_{\nu-m}(\la x)J_{\nu}(\la
 t)}{\|{\boldsymbol\vf}^{(m)}_{\nu,\la}\|_{m}^2}.
\end{split}
\end{equation}
}

\item

Par contre ces fonctions ci-dessus peuvent, en fonction de $\nu$, avoir un pôle en $w=0$. Comme $J_p$ est d'ordre $|p|$ en z\'ero et en utilisant lemme \ref{zerosimple}, alors  on peut
  calculer explicitement l'ordre de ces fonctions en $w=0$ en fonction de $\nu$. En effet, lorsque $\nu\leq -1$ ou $\nu\geq m+1$, on a:
{\allowdisplaybreaks
\begin{align*}
\mathrm{Ord}_0( t_{+,\nu})&=(2\nu+\nu+m)-(\nu+2\nu+m-1)=1,\\
\mathrm{Ord}_0(t_{-,\nu})&=(2\nu +\nu+m)-(2\nu+m-1)=1,\\
\mathrm{Ord}_0 (s_{+,\nu})&=1,\\
\mathrm{Ord}_0(s_{-,\nu})&=1.
\end{align*}}
\item
Si $\nu\in \{0,1,\ldots,m\}$. Nous  montrons que  $t_{+,\nu}$, $t_{-,\nu}$, $s_{+,\nu}$ et $s_{-,\nu}$
  sont d'ordre $-1$ en $w=0$ pour cela on va
calculer leur r\'esidu en $w=0$. Les calculs suivants d\'ecoulent directement de  la formule \ref{ordreenzero} et
du lemme   \ref{zerosimple}. On \'etablit que:

\begin{equation}\label{x19}
\begin{split}
\mathrm{Res}_0(t_{+,\nu})&=2\frac{(m-\nu+1)(\nu+1)}{(m+2)}t^\nu x^\nu,\\
\mathrm{Res}_0(t_{-,\nu})&=2\frac{(m-\nu+1)(\nu+1)}{(m+2)}t^{m-\nu} x^\nu,\\
 \mathrm{Res}_0(s_{+,\nu})&=2\frac{(m-\nu+1)(\nu+1)}{(m+2)}t^{m-\nu} x^{m-\nu},\\
\mathrm{Res}_0(s_{-,\nu})&=2\frac{(m-\nu+1)(\nu+1)}{(m+2)}t^{\nu} x^{m-\nu}.\\
\end{split}
\end{equation}

\end{enumerate}

Maintenant on se propose de d\'emontrer le th\'eor\`eme \ref{ttss}. Soit $\nu \in \Z$. Commençons par exemple par $t_{+,\nu}$ (avec $\nu$ un entier quelconque).  Soit $A_n$ un r\'eel
positif strictement compris entre $\la_n$ et $\la_{n+1} $ deux z\'eros cons\'ecutifs de $L_{m,\nu}$, $B>0$ et $0<\eps\ll 1$.
  On consid\`ere  $R_n$ le chemin ferm\'e  dans $\C$
  suivant: $[iB,A_n+i B]\cup [A_n+iB,A_n-iB]\cup [-iB, A_n-iB]\cup
[-iB, -i\eps]\cup D_\eps \cup [i\eps,iB]$ o\`u $D_\eps$ est le demi-cercle de centre $w=0$ \`a
 gauche de l'axe imaginaire et de rayon $\eps$.
 On a par un simple calcul de r\'esidus et en utilisant \ref{x20}:
 \[
 \int_{R_n}t_{+,\nu}(w)dw=-\mathrm{Res}_0(t_{+,\nu})+\sum_{k=1}^n\frac{J_{\nu+m}(t\la_k)J_{\nu+m}(x\la_k)}{2\|{\boldsymbol\vf}_{\nu,k}^{(m)}\|_{m}^2}-\sum_{\substack{\al|J_{\nu+m}(\al)=0\\
0<\al<A_n}}\frac{J_{\nu+m}(t\al)J_{\nu+m}(x\al)}{J_{\nu+m+1}(\al)^2}.
 \]

On pose $w=u+iv$ avec $u,v\in \R$ et on suppose que $t,x>0$ avec $0<t+x<2$. On obtient, en utilisant \ref{LLLLL}
  et \ref{besselineqtriv}:
\begin{equation}\label{t+++}
 \begin{split}
  |t_{+,\nu}(w)|&=\Bigl|\frac{J_{\nu+m}(tw)J_{\nu+m}(xw)J_\nu(w)}{J_{\nu+m}(w)L_{m,\nu}(w)}\Bigr|\\
&\leq \frac{c_{+,\nu}}{\sqrt{tx}} \exp\bigl(-(2-(t+x))|v| \bigr),
  \end{split}
\end{equation}
pour tout $w$ dans un voisinage ouvert de la droite  verticale d'abscisse  $A_n$ ($c_{+,\nu}$ une constante r\'eelle). Par cons\'equent
\[
\lim_{B\mapsto \infty} \int_{-Bi}^{A_n-Bi}t_{+,\nu}(w)dw=\lim_{B\mapsto \infty} \int_{Bi}^{A_n+Bi}t_{+,\nu}(w)dw=0
\]
Quant \`a la contribution de $t_{+,\nu}$ sur la partie  $[-iB, -i\eps]\cup D_\eps \cup [i\eps,iB]
$,  nous utilisons
  le fait que $t_{+,\nu}$ est une fonction impaire $\bigl($cela r\'esulte de la relation \ref{Lparite} et que
 $J_p(-z)=(-1)^pJ_p(z),\; \forall\, p\in \N, \forall\, z\in \C$ $\bigr)$.
  Cela implique que l'int\'egrale le long du segment $[-Bi,-\eps i]\cup D_\eps\cup [i\eps ,iB]
   $  est donn\'ee par $\mathrm{Res}_0(t_{+,\nu})$. En combinant tout cela, nous obtenons:
\[
\frac{1}{2\pi i}\int_{A_n-\infty i}^{A_n+\infty i}t_{+,\nu}(w)dw=-\mathrm{Res}_0(t_{+,\nu})+\sum_{k=1}^n\frac{J_{\nu+m}(t\la_k)J_{\nu+m}(x\la_k)}{2\|{\boldsymbol\vf}_{\nu,k}^{(m)}\|_{m}^2}-\sum_{\substack{\al|J_{\nu+m}(\al)=0\\
0<\al<A_n}}\frac{J_{\nu+m}(t\al)J_{\nu+m}(x\al)}{J_{\nu+m+1}(\al)^2}.
\]
Et en utilisant \ref{x19} cela termine la preuve du th\'eor\`eme pour les formules en $t_{+,\nu}$ pour $\nu\in \Z$ quelconque.\\

De la m\^eme mani\`ere on peut \'etablir qu'il existe des constantes r\'eelles $c_{-,\nu},d_{+,\nu}$ et $d_{-,\nu}$ telles que:
\begin{align*}
  |t_{-,\nu}(w)|&\leq \frac{c_{-,\nu}}{\sqrt{tx}} \exp\bigl(-(2-(t+x))|v| \bigr),\\
  |s_{+,\nu}(w)|&\leq \frac{d_{+,\nu}}{\sqrt{tx}} \exp\bigl(-(2-(t+x))|v| \bigr),\\
  |s_{-,\nu}(w)|&\leq \frac{d_{-,\nu}}{\sqrt{tx}} \exp\bigl(-(2-(t+x))|v| \bigr),
\end{align*}
valables pour tout $w$ dans un domaine contenant la droite verticale d'abscisse $A_n$ et on suit le m\^eme raisonnement pour trouver les formules restantes. Ce qui termine la preuve du th\'eor\`eme \ref{ttss}.


\begin{lemma}
 Soient $n\in \N$ et $k\in \N$. Si $\frac{k-2}{4}\in \N$ alors il existe des r\'eels $b_{1,k},\ldots,b_{\frac{k-2}{4},k}$ tels que
\[
 \int_0^1t^{k+1+n}J_n(tw)dt=\frac{J_{n+1}(w)}{w}-k\frac{J_{n+2}(w)}{w^2}+\sum_{i=1}^{\frac{k-2}{4}}b_{i,k}\frac{J_{n+i+2}(w)}{w^{i+2}},\quad \forall\, w\in \C^\ast.
\]
\end{lemma}
\begin{proof}
En utilisant la formule \ref{B5},
on a par une int\'egration par parties:\footnote{Le signe $\int_0^w$ signifie l'int\'egration le long du segment $[0,w]$ dans le plan complexe.}
{\allowdisplaybreaks
\begin{align*}
\int_0^w t^{k+n+1}J_n(t)dt&=\bigl[t^{k+n+1}J_{n+1}(t) \bigr]_0^w-k\int_0^w t^{k+n}J_{n+1}(t)dt\\
&=w^{k+n+1}J_{n+1}(w)-k\bigl[t^{k+n}J_{n+2}(t) \bigr]_0^w+k(k-2)\int_0^w t^{k+n-1}J_{n+2}(t)dt\\
&=w^{k+n+1}J_{n+1}(w)-kw^{k+n}J_{n+2}(w)\\
&+k(k-2)\bigl[t^{k+n-1}J_{n+3}(t) \bigr]_0^w-k(k-2)(k-4)\int_0^w t^{k+n-2}J_{n+3}(t)dt\\
&=w^{k+n+1}J_{n+1}(w)+\sum_{i=0}^p (-1)^{i+1}k(k-2)(k-4)\cdots (k-2i)w^{k+n-i}J_{n+i+2}(w)\\
&+(-1)^{p+1}k(k-2)\cdots (k-2(i+1))\int_0^w t^{k+n-p-1}J_{n+p+2}(t)dt.
\end{align*}}
On s'arr\^ete lorsque $(k+n-p-1)-1=n+p+2$, c'est \`a dire lorsque $p=\frac{k-4}{2}$. On termine la preuve en notant que
$
\int_0^1 t^{k+n+1}J_n(tw)dt=\frac{1}{w^{k+n+2}}\int_0^w t^{k+n+1} J_n(t)dt$ pour tout $w\neq 0$.
\end{proof}

\begin{Corollaire}\label{x21}
Soient $p$ et $q$ deux entiers positifs sup\'erieurs \`a 2. Il existe des constantes r\'eelles $c_1,c_2,\ldots,c_{{}_{\max(p-2,q-2)}}$ telles que
\[\int_0^1\bigl(t^{2p+n+1}-t^{2q+n+1}\bigr)J_n(tw)dt=2\frac{q-p}{w^2}J_{n+2}(w)+\sum_{i=1}^{\max(p-2,q-2)}c_i \frac{J_{n+i+2}(w)}{w^{i+2}}.\]
\end{Corollaire}
\begin{proof}
C'est une cons\'equence imm\'ediate du lemme pr\'ec\'edent.
\end{proof}
En fonction de $\nu$, nous allons \'etablir que la fonction r\'eelle $x\mapsto x^{2p+\nu+1}-x^{2q+\nu+1}$ admet
un d\'eveloppement de type Bessel-Fourier dans le contexte de cet article. C'est l'objet
des th\'eor\`emes \ref{x22} et \ref{x23} qui  constitueront le noyau de la preuve du th\'eor\`eme \ref{x26}.
\begin{theorem}\label{x22}
Soit $\nu\in \Z$. Si $\nu\leq -1$ ou $\nu\geq m+1$ alors  pour $p,q\in \N$ assez grands, il existe des constantes r\'eelles $l_k$ et $l'_k$ pour tout $k\in \N^\ast$ telles  qu'on a convergence uniforme sur  $[0,1]$ des s\'eries suivantes:
\begin{enumerate}
\item
\begin{enumerate}
\item
\[
\sum_{k=1}^\infty   l_k
x^\frac{1}{2}\frac{J_{\nu}(x\la_k)}{2\|{\boldsymbol\vf}_{\nu,k}^{(m)}\|_{m}^2}=x^{2p+\nu+\frac{1}{2}
}-x^{2q+\nu+\frac{1}{2}},
\]
\item
\[
\sum_{k=1}^\infty   l'_k x^\frac{1}{2}\frac{J_{\nu}(x\la_k)}{2\|{\boldsymbol\vf}_{\nu,k}^{(m)}\|_{m}^2}=0,
\]
\end{enumerate}

\item
\begin{enumerate}
\item
\[
\sum_{k=1}^\infty l'_k \frac{J_{\nu}(\la_k)}{J_{\nu-m}(\la_k)}x^\frac{1}{2}\frac{J_{\nu-m}(x\la_k)}{\|{\boldsymbol\vf}_{\nu,k}^{(m)}\|_{m}^2}
=x^{2p+\nu-m+\frac{1}{2}}-x^{2q+\nu-m+\frac{1}{2}},
\]
\item
\[
\sum_{k=1}^\infty l_k \frac{J_{\nu}(\la_k)}{J_{\nu-m}(\la_k)}x^\frac{1}{2}\frac{J_{\nu-m}(x\la_k)}{\|{\boldsymbol\vf}_{\nu,k}^{(m)}\|_{m}^2}
=0.\]
\end{enumerate}
\end{enumerate}

\end{theorem}
\begin{proof}
Soit $\nu\in\Z$. On fixe $p$ et $q$ deux entiers  sup\'erieurs \`a $|\nu|$.
\begin{enumerate}
\item
\begin{enumerate}
\item
D'apr\`es  le corollaire \ref{x21},
 il existe des constantes r\'eelles $c_{i,\nu}$ telles que pour tout $w$ dans un ouvert ne contenant pas les pôles
 de $t_{+,\nu}$, on a
\begin{align*} \int_0^1\bigl(t^{2p+\nu+1}-t^{2q+\nu+1}\bigr)&t_{+,\nu}(w)dt=\int_0^1\bigl(t^{2p+\nu+1}-t^{2q+\nu+1}\bigr)J_{\nu}(tw)dt\frac{J_{\nu}(xw)J_{\nu-m}(w)}{J_{\nu}(w)L_{m,\nu}(w)},\\
 &=2\frac{q-p}{w^2}J_{\nu+2}(w)\frac{J_{\nu}(xw)J_{\nu-m}(w)}{J_{\nu}(w)L_{m,\nu}(w)}+\sum_{i=1}^{\max(p-2,q-2)}c_{i,\nu} \frac{J_{n+i+2}(w)}{w^{i+2}}\frac{J_{\nu}(xw)J_{\nu-m}(w)}{J_{\nu}(w)L_{m,\nu}(w)}.
\end{align*}
En utilisant   \ref{LLLLL}, il existe une constante $c_{\nu,++}$ telle que:
\begin{align*}
\Biggl|\int_{A_n-i\infty}^{A_n+i\infty}\frac{J_{n+i+2}(w)}{w^{i+2}}\frac{J_{\nu}(xw)J_{\nu-m}(w)}{J_{\nu}(w)L_{m,\nu}(w)}dw\Biggr|&\leq \int_{A_n-i\infty}^{A_n-i\infty}\frac{c_{\nu,++}}{\sqrt{x}|w|^{i+2}}\exp\bigl(-(1-x)|v| \bigr)|dw|\\
&\leq \frac{c_{\nu,++}}{\sqrt{x}}\int_{-\infty}^\infty \frac{1}{(A_n^2+v^2)}dv,\quad \text{puisque} \, 0<x<1\\
&=\frac{\pi c_{\nu,++}}{\sqrt{x}A_n}.
\end{align*}
Par suite, il existe une constante $C_\nu$ telle que
\begin{equation}\label{w1}
\Biggl|\int_{A_n-i\infty}^{A_n+i\infty}\int_0^1\bigl(t^{2p+\nu+1}-t^{2q+\nu+1}\bigr)x^\frac{1}{2}
t_{+,\nu}(w)dt\Biggr|\leq \frac{C_\nu}{A_n},\quad \forall\, n\in \N^\ast, \, x\in [0,1[.
\end{equation}
 Notons que
$\int_{A_n-i\infty}^{A_n+i\infty}\int_0^1\bigl(t^{2p+\nu+1}-t^{2q+\nu+1}\bigr)t_{+,\nu}(w)dt
=\int_0^1\bigl(t^{2p+\nu+1}-t^{2q+\nu+1}\bigr)\int_{A_n-i\infty}^{A_n+i\infty} t_{+,\nu}(w)dt.$ En effet, pour
$x\in ]0,1[$ fix\'e, on a $t_{+,\nu}$ consid\'er\'ee comme une fonction en $t$ et $w$ est absolument
int\'egrable (d'apr\`es \ref{t+++}).

On suppose que $\nu\leq -1$ ou $\nu\geq m+1$.
En utilisant  \ref{q1} du th\'eor\`eme \ref{ttss}, \ref{w1} devient:
\begin{equation}\label{w2}
 \Bigl|
\int_0^1\bigl(t^{2p+\nu+1}-t^{2q+\nu+1}\bigr)x^{\frac{1}{2}}\kappa_n^+(t,x)dt-\sum_{\substack{\al|J_{\nu}(\al)=0\\
0<\al<A_n}}
\int_0^1\bigl(t^{2p+\nu+1}-t^{2q+\nu+1}\bigr)x^{\frac{1}{2}}\frac{J_{\nu}(t\al)J_{\nu}(x\al)}{J_{\nu+1}(\al)^2}
\Bigr|\leq \frac{C_\nu}{A_n},
\end{equation}
 pour tout $x\in [0,1]$ et $n\in \N$.\\

 Montrons que  la somme \[
\sum_{\substack{\al|J_{\nu}(\al)=0\\
0<\al<A_n}}
\int_0^1\bigl(t^{2p+\nu+1}-t^{2q+\nu+1}\bigr)x^{\frac{1}{2}}\frac{J_{\nu}(t\al)J_{\nu}(x\al)}{J_{\nu+1}(\al)^2},\]
converge uniform\'ement vers  la fonction
 $x\mapsto
x^{2p+\nu+\frac{1}{2}}-x^{2q+\nu+\frac{1}{2}}$ sur $[0,1]$ lorsque $n$ tend vers l'infini.
En effet, on a d'une part  la fonction $x\mapsto x^{2p+\nu}-x^{2p+\nu}$ est d\'erivable sur $]0,1]$
et s'annule en $x=1$ alors (d'apr\`es \cite[\S 18.26]{Watson}) son d\'eveloppement en s\'erie de Fourier-Bessel converge
uniform\'ement vers cette fonction sur un intervalle de la forme $[1-\delta,1]$ avec $\delta>0$ quelconque.
  D'autre part,  le r\'esultat du  (\cite[p. 617]{Watson} voir premier paragraphe)  nous fournit la convergence sur
  $[0,1-\delta]$, avec $\delta>0$  arbitraire. En combinant cela avec \ref{w2}, on d\'eduit que:
  \begin{equation}\label{w3}
\lim_{n\mapsto \infty}
\int_0^1\bigl(t^{2p+\nu+1}-t^{2q+\nu+1}\bigr)x^{\frac{1}{2}}\kappa_n^+(t,x)dt=x^{2p+\nu+\frac{1}{2}}
-x^{2p+\nu+\frac{1}{2}},
\end{equation}
  uniform\'ement sur $[0,1]$.

  Si l'on pose
  $l_k:=\int_0^1\bigl(t^{2p+\nu+1}-t^{2q+\nu+1}\bigr)J_{\nu}(t\la_k)dt$. Alors, l'assertion pr\'ec\'edente
  devient:
\[
\sum_{k=1}^\infty   l_k x^\frac{1}{2}\frac{J_{\nu}(x\la_k)}{2\|{\boldsymbol\vf}_{\nu,k}^{(m)}\|_{m}^2}=x^{2p+\nu+\frac{1}{2}
}-x^{2q+\nu+\frac{1}{2}},
\]
uniform\'ement sur $[0,1]$.
\item
Si l'on consid\`ere  $t_{-,\nu}$  et en utilisant le m\^eme raisonnement pr\'ec\'edent et \ref{q2}, on obtient
avec le m\^eme raisonnement pr\'ec\'edent que:
 \begin{equation}
 \lim_{n\mapsto \infty}
\int_0^1\bigl(t^{2p+\nu+1}-t^{2q+\nu+1}\bigr)x^{\frac{1}{2}}\kappa_n^-(t,x)dt=0,
\end{equation}
uniform\'ement sur $[0,1]$. En posant
$l_k':=\frac{J_\nu(\la_k)}{J_{m-\nu}(\la_k)}\int_0^1\bigl(t^{2p+\nu+1}-t^{2q+\nu+1}\bigr)J_{m-\nu}(t\la_k)dt$, la
derni\`ere assertion devient:
\[\sum_{k=1}^\infty   l'_k x^\frac{1}{2}\frac{J_{\nu}(x\la_k)}{2\|{\boldsymbol\vf}_{\nu,k}^{(m)}\|_{m}^2}=0,\quad \forall\, x\in [0,1].\]
 uniform\'ement
 sur $[0,1]$.
\end{enumerate}
\item

\begin{enumerate}
\item Par application du \ref{x21}, on montre aussi que:
\begin{align*} \int_0^1\bigl(t^{2p+\nu+1}-t^{2q+\nu+1}\bigr)&s_{+,\nu}(w)dt=\int_0^1\bigl(t^{2p+\nu-m+1}-t^{2q+\nu-m+1}\bigr)J_{\nu-m}(tw) dt \frac{J_\nu(w)}{J_{\nu-m}(w)}\frac{J_{\nu-m}(xw)J_{\nu-m}(w)}{J_{\nu-m}(w)L_{m,\nu}(w)},\\ &=2\frac{q-p}{w^2}J_{\nu-m+2}(w)\frac{J_\nu(w)}{J_{\nu-m}(w)}\frac{J_{\nu-m}(xw)J_{\nu-m}(w)}{J_{\nu-m}(w)L_{m,\nu}(w)}\\
&+\sum_{i=1}^{\max(p-2,q-2)}c_{i,\nu-m} \frac{J_{n+i+2}(w)}{w^{i+2}}\frac{J_\nu(w)}{J_{\nu-m}(w)}\frac{J_{\nu-m}(xw)J_{\nu-m}(w)}{J_{\nu-m}(w)L_{m,\nu}(w)}.\\
\end{align*}
Et comme avant, on peut trouver $C_\nu'$ une constante r\'eelle telle que:
\[
\Biggl|\int_{A_n-i\infty}^{A_n+i\infty}\int_0^1\bigl(t^{2p+\nu-m+1}-t^{2q+\nu-m+1}\bigr)s_{+,\nu}(w)dt\Biggr|\leq \frac{C_\nu}{\sqrt{x}A_n},\quad \forall\, n\in \N^\ast, \,\forall\, x\in ]0,1[.
\]
Cette derni\`ere in\'egalit\'e combin\'ee avec \ref{q3},  nous permet  de d\'eduire que:
\begin{align*}
\lim_{
n\mapsto \infty}\int_0^1\bigl(t^{2p+\nu-m+1}-t^{2q+\nu-m+1}\bigr)&x^\frac{1}{2}\tau_n^+(t,x)dt\\
&=\sum_{\substack{\gamma|J_{\nu-m}(\gamma)=0\\
0<\gamma}}\int_0^1\bigl(t^{2p+\nu-m+1}-t^{2q+\nu-m+1}\bigr)
J_{\nu-m}(t\gamma)dt x^\frac{1}{2}\frac{J_{\nu-m}(x\gamma)}{J_{\nu-m+1}(\gamma)^2}dt\\
&=x^{2p+\nu-m+\frac{1}{2}}-x^{2q+\nu-m+\frac{1}{2}},
\end{align*}
uniform\'ement sur $[0,1]$. En d'autres termes, la somme suivante converge uniform\'ement vers $x^{2p+\nu-m+\frac{1}{2}}-x^{2q+\nu-m+\frac{1}{2}}$ sur $[0,1]$:
\begin{align*}
\sum_{k=1}^\infty l'_k \frac{J_{\nu}(\la_k)}{J_{\nu-m}(\la_k)}x^\frac{1}{2}\frac{J_{\nu-m}(x\la_k)}{\|{\boldsymbol\vf}_{\nu,k}^{(m)}\|_{m}^2}
=x^{2p+\nu-m+\frac{1}{2}}-x^{2q+\nu-m+\frac{1}{2}},\quad \forall\, x\in [0,1].
\end{align*}
o\`u $l'_k=\frac{J_\nu(\la_k)}{J_{\nu-m}(\la_k)}\int_0^1\bigl(t^{2p+m-\nu+1}-t^{2q+m-\nu+1}\bigr)J_{m-\nu}(t\la_k)dt$.

\item Pour la derni\`ere formule, on consid\`ere $s_{-,\nu}$ et on montre comme avant que (en utilisant \ref{q4}):
\[
\sum_{k=1}^\infty l_k \frac{J_{\nu}(\la_k)}{J_{\nu-m}(\la_k)}\frac{J_{\nu-m}(x\la_k)}{\|{\boldsymbol\vf}_{\nu,k}^{(m)}\|_{m}^2}
=0, \quad \forall\, x\in [0,1].
\]
avec convergence uniforme sur $[0,1]$.
\end{enumerate}
\end{enumerate}
\end{proof}

Le th\'eor\`eme suivant traite le cas restant, c-\`a-d $\nu\in \{0,1,\ldots,m\}$:
\begin{theorem}\label{x23}
Pour tout $\nu\in \{0,1,\ldots,m\}$, il existe $\delta_\nu\in \R^\ast$ tel que pour tout  $x\in [0,1]$:

\begin{enumerate}
\item
\begin{enumerate}
\item
\[
\delta_\nu x^{\nu+\frac{1}{2}}= \sum_{k=1}^\infty   l_k
x^\frac{1}{2}\frac{J_{\nu}(x\la_k)}{2\|{\boldsymbol\vf}_{\nu,k}^{(m)}\|_{m}^2}-\bigl(x^{2p+\nu+\frac{1}{2}}-x^{2q+\nu
+\frac{1}{2}}\bigr),
\]
\item
\[
\delta_{m-\nu}
x^{\nu+\frac{1}{2}}= \sum_{k=1}^\infty   l'_k
x^\frac{1}{2}\frac{J_{\nu}(x\la_k)}{2\|{\boldsymbol\vf}_{\nu,k}^{(m)}\|_{m}^2}.
\]
\end{enumerate}

\item
\begin{enumerate}

\item
\[
\delta_{m-\nu} x^{m-\nu+\frac{1}{2}}= \sum_{k=1}^\infty   l_k'
x^\frac{1}{2}\frac{J_\nu(\la_k)}{J_{\nu-m}(\la_k)}
\frac{J_{\nu-m}(x\la_k)}{2\|{\boldsymbol\vf}_{\nu,k}^{(m)}\|_{m}^2}-\bigl(x^{2p+m-\nu+\frac{1}{2}}
-x^{2q+m-\nu+\frac{1}{2}}\bigr),
\]
\item
\[
\delta_\nu x^{m-\nu+\frac{1}{2}}= \sum_{k=1}^\infty   l_k
x^\frac{1}{2}\frac{J_\nu(\la_k)}{J_{\nu-m}(\la_k)}\frac{J_{\nu-m}(x\la_k)}{2\|{\boldsymbol\vf}_{\nu,k}^{(m)}\|_{m}^2},
\]
\end{enumerate}
\end{enumerate}
En plus, on a convergence uniforme de ces sommes sur $[0,1]$. (Les constantes $l_k$ et $l_k'$ sont les m\^emes
qu'au th\'eor\`eme \ref{x22}.)
\end{theorem}
\begin{proof}

Soit  $ \nu\in \{ 0,1,\ldots,m \}$.  En utilisant (2.) du th\'eor\`eme \ref{ttss}, alors la preuve de ce th\'eor\`eme suit  le m\^eme raisonnement fait pour \'etablir  th\'eor\`eme
 \ref{x22}. On trouve:
{\allowdisplaybreaks
\begin{align*}
2\frac{(\nu+1)(m-\nu+1)(q-p)}{(m+2)(p+\nu+1)(q+\nu+1)} x^{\nu+\frac{1}{2}}&= \sum_{k=1}^\infty   l_k
x^\frac{1}{2}\frac{J_{\nu}(x\la_k)}{2\|{\boldsymbol\vf}_{\nu,k}^{(m)}\|_{m}^2}-\bigl(x^{2p+\nu+\frac{1}{2}}-x^{2q+\nu
+\frac{1}{2}}\bigr),\\
2\frac{(\nu+1)(m-\nu+1)(q-p)}{(m+2)(p+m-\nu+1)(q+m-\nu+1)} x^{\nu+\frac{1}{2}}&= \sum_{k=1}^\infty   l'_k
x^\frac{1}{2}\frac{J_{\nu}(x\la_k)}{2\|{\boldsymbol\vf}_{\nu,k}^{(m)}\|_{m}^2},\\
2\frac{(\nu+1)(m-\nu+1)(q-p)}{(m+2)(p+m-\nu+1)(q+m-\nu+1)} x^{m-\nu+\frac{1}{2}}&= \sum_{k=1}^\infty   l'_k
x^\frac{1}{2}\frac{J_\nu(\la_k)}{J_{\nu-m}(\la_k)}
\frac{J_{\nu-m}(x\la_k)}{2\|{\boldsymbol\vf}_{\nu,k}^{(m)}\|_{m}^2}-\bigl(x^{2p+m-\nu+\frac{1}{2}}
-x^{2q+m-\nu+\frac{1}{2}}\bigr),\\
2\frac{(\nu+1)(m-\nu+1)(q-p)}{(m+2)(p+\nu+1)(q+\nu+1)} x^{m-\nu+\frac{1}{2}}&= \sum_{k=1}^\infty   l_k
x^\frac{1}{2}\frac{J_\nu(\la_k)}{J_{\nu-m}(\la_k)}\frac{J_{\nu-m}(x\la_k)}{2\|{\boldsymbol\vf}_{\nu,k}^{(m)}\|_{m}^2},\\
\end{align*}
}
avec convergence uniforme  sur $[0,1]$. On pose $\delta_\nu=2\frac{(\nu+1)(m-\nu+1)(q-p)}{(m+2)(p+\nu+1)(q+\nu+1)},\; \forall\,\nu \in \{0,1,\ldots,m\}$.\\
\end{proof}

Maintenant, on se propose d'\'etablir le th\'eor\`eme \ref{x26}. Soit $\nu\in \Z$. Si
  $\nu\leq -1$ ou $\nu\geq m+1$ on pose  $\delta_\nu=0$, sinon $\delta_\nu$ est le r\'eel dans \ref{x23}.

Soit $f\in \mathcal{E}_m$. D'apr\`es \ref{x22} et \ref{x23}, nous avons montr\'e qu'il existe des constantes
r\'eelles
$\delta_\nu$ qui ne d\'ependent pas de $f$ v\'erifiant $\delta_\nu=0$ si $\nu\leq -1$ ou $\nu\geq m+1$ telles que:

D'apr\`es 1.(a) et 2.(b) du th\'eor\`eme \ref{x22} (resp. 1.(a) et 2.(b) du th\'eor\`eme
 \ref{x23}), on a pour tout $\nu\leq -1$ ou $\nu\geq m+1$ (resp. pour tout  $\nu\in \{0,1,\ldots,m\}$):
\begin{equation}\label{q5}
\begin{split}
\delta_\nu \int_0^1 x^{\nu+1}f(x)dx&= \sum_{k=1}^\infty   l_k \frac{\int_0^1 f(x)J_{\nu}(x\la_k)xdx}{2\|{\boldsymbol\vf}_{\nu,k}^{(m)}\|_{m}^2}-\int_0^1\bigl(x^{2p+\nu+1}-x^{2q+\nu+1}\bigr)f(x)dx,\\
\delta_\nu \int_0^1 x^{2m-\nu+1}f(\frac{1}{x})dx&= \sum_{k=1}^\infty   l_k \frac{J_\nu(\la_k)}{J_{\nu-m}(\la_k)}\frac{\int_0^1 x^mf(\frac{1}{x})J_{\nu-m}(x\la_k)xdx}{2\|{\boldsymbol\vf}_{\nu,k}^{(m)}\|_{m}^2},
\end{split}
\end{equation}
D'apr\`es 1.(b) et 2.(a) du th\'eor\`eme \ref{x22} (resp. 1.(b) et 2.(a) du th\'eor\`eme
 \ref{x23}), on a pour tout $\nu\leq -1$ ou $\nu\geq m+1$ (resp. pour tout  $\nu\in \{0,1,\ldots,m\}$):

\begin{equation}\label{q6}
\begin{split}
\delta_{m-\nu}\int_0^1 x^{2m-\nu+1}f(\frac{1}{x})dx&= \sum_{k=1}^\infty   l'_k \frac{J_\nu(\la_k)}{J_{\nu-m}(\la_k)}\frac{\int_0^1 x^m f(\frac{1}{x})J_{\nu-m}(x\la_k)xdx}{2\|{\boldsymbol\vf}_{\nu,k}^{(m)}\|_{m}^2}-\int_0^1\bigl(x^{2p+2m-\nu+1}
-x^{2q+2m-\nu+1}\bigr)f(\frac{1}{x})dx,\\
\delta_{m-\nu} \int_0^1 x^{\nu+1}f(x)dx&= \sum_{k=1}^\infty   l'_k \frac{\int_0^1 f(x)J_{\nu}(x\la_k)xdx}{2\|{\boldsymbol\vf}_{\nu,k}^{(m)}\|_{m}^2},
\end{split}
\end{equation}
(Rappelons que si $\nu\in \{0,1,\ldots,m\}$, on a
$(f ,x^\nu)_m= \int_0^1 x^{\nu+1}f(x)dx+\int_0^1 x^{2m-\nu+1}f(x)dx$. Par contre lorsque
 $\nu\leq -1$ ou $\geq m+1$,
on  fait la convention suivante $(f,x^\nu)_m:=0$).
En additionnant les deux \'equations du \ref{q5} (resp \ref{q6}), on obtient:

\begin{align*}
\delta_\nu(f,x^\nu)_m&= \sum_{k=1}^\infty   l_k \frac{ (f
,{\boldsymbol\vf}_{\nu,k}^{(m)})_{m}}{2\|{\boldsymbol\vf}_{\nu,k}^{(m)}\|_{m}^2}-\int_0^1\bigl(x^{2p+\nu+1}
-x^{2q+\nu+1}\bigr)f(x)dx,\\
\delta_{\nu-m}(f,x^\nu)_m&=\sum_{k=1}^\infty   l'_k
\frac{(f,{\boldsymbol\vf}_{\nu,k}^{(m)})_{m}}{2\|{\boldsymbol\vf}_{\nu,k}^{(m)}\|_{m}^2}
-\int_0^1\bigl(x^{2p+2m-\nu+1}-x^{2q+2m-\nu+1}\bigr)f(\frac{1}{x})dx.
\end{align*}
Ce termine la preuve du th\'eor\`eme \ref{x26}.

\section{Annexe}

On rappelle le r\'esultat suivant:
\begin{proposition}\label{Green}$($Formule de Green$)$ Soit $X$ une vari\'et\'e complexe de dimension $d$ et $A\subset X$ un ouvert relativement compact tel que $\overline{A}$ soit une sous-vari\'et\'e r\'eelle \`a coins de $X$. Soient $f$ et $g$ deux formes diff\'erentielles de classes $\mathcal{C}^2$ au voisinage de $\overline{A} $ de bidegr\'es $(p,p)$ et $(q,q)$ telles que $p+q=d-1$. On a:
{{}
\[
 \int_{A}(fdd^c g-g dd^c f)=\int_{\partial A} (fd^cg-g d^cf)
\]}
\end{proposition}
\begin{proof}
 Voir par exemple \cite{DemaillyLivre}.
\end{proof}

On rappelle qu'un ensemble $\{ \vf_n\}_{n\in \N}$ de vecteurs orthogonaux deux \`a deux dans $\mathcal{H}$, un espace hilbertien, est dit total si et seulement si l'ensemble $S$ form\'e des combinaisons lin\'eaires finies $\sum_{k=1}^n a_k \vf_k$ des $\vf_k$ est dense dans $\mathcal{H}$.
\begin{proposition}\label{base}
 Une suite orthogonale $(\phi_j)_{j\in \N}$ de vecteurs d'un espace de Hilbert $H$ est  totale si et seulement si
le vecteur nul est l'unique vecteur orthogonal \`a tous les $\phi_j$.
\end{proposition}

\begin{proof}
 Voir par exemple \cite[th\'eor\`eme 12, p.366]{Birkhoff}.
\end{proof}

La proposition suivante donne une description d'une extension maximale autoadjointe associ\'ee \`a un op\'erateur positif admettant
une famille de vecteurs propres formant un syst\`eme total.
\begin{lemma}\label{extensionMaximale}

Soit $H$ un espace de Hilbert avec produit hermitien not\'e $(\cdot,\cdot)$.

Soit $\Delta: D\subset H\lra H$ un op\'erateur lin\'eaire, o\`u $D$ est un sous-espace lin\'eaire dense de $H$. Soit
$(\phi_k)_{k\in\N}$ une base orthonormale de $H$, et on suppose qu'il existe $0\leq \la_0\leq \la_1\leq \ldots$
une suite croissante de r\'eels positifs tels que:
\[
\bigl(\phi_k,\Delta\psi \bigr)=\la_k(\phi_k,\psi)\quad \forall\, \psi\in D,\;\forall\, k\in \N.
\]

On pose $H_2:=\bigl\{\psi=\sum_{k=0}^\infty a_k\phi_k\in H\bigl| a_k\in \C,\; \forall k\in \C, \sum_{k=0}^\infty \la_k^2|a_k|^2<\infty \bigr\}$ et soit $Q$ l'op\'erateur lin\'eaire d\'efini sur $H_2$ en posant:
\[
Q(\psi)=\sum_{k=0}^\infty \la_k a_k\phi_k,
\]
pour $\psi=\sum_{k=0}^\infty  a_k\phi_k\in H_2$. \\

Si $D\subset H_2$, alors $Q$ est une extension maximale autoadjointe de $\Delta$ dans $H_2$.
\end{lemma}
\begin{proof}

V\'erifions d'abord que
$Q$ est une extension de $\Delta$, soit $\psi\in D$ et montrons que $Q(\psi)=\Delta(\psi)$. Il existe deux suites de nombres
complexes $(a_k)_{k\in \N}$ et $(b_k)_{k\in \N}$ telles que $\psi=\sum_{k=0}^\infty a_k \phi_k$ et $\Delta\psi=\sum_{k=0}^\infty
b_k \phi_k$. Donc, $\overline{b_k}=(\phi_k,\Delta\psi)=\la_k(\phi_k,\psi)=\la_k\overline{a_k}$, par suite
$\Delta\psi=\sum_{k=0}^\infty \la_k a_k \psi_k=Q(\psi)$.

\begin{align*}
(Q(\psi),\psi')=\sum_{k=0}^\infty \la_k a_k\overline{a}_k'= (\psi,Q(\psi')), \quad \forall\,\psi,\psi'\in H_2.
\end{align*}
donc $Q$ est autoadjoint.

Maintenant, soit $T:D'\subset H\lra H$ est  un op\'erateur
lin\'eaire autoadjoint qui \'etend $Q$ (c'est \`a dire $H_2\subset D'$ et $T_{|D}=\Delta$). Soit $\psi\in D'$, il existe deux suites de
nombres complexes $(a_k)_{k\in _N}$ et $(b_k)_{k\in _N}$ telles que $\psi=\sum_{k=0}^\infty a_k\phi_k $ et $T\psi=\sum_{k\in \N}b_k
\phi_k$. Comme $T$ est autoadjoint, alors $(Tw,\psi)=(w,T\psi)$ pour tout $w\in D'$. En particulier,
$(T\psi_k,\psi)=(\psi_k,T\psi)$ pour tout $k\in \N$. On en tire que $\la_k \overline{a_k}=\overline{b_k}$. Par suite,
$\sum_{k\in\N}\la_k^2|a_k|^2=\sum_{k\in \N}|b_k|^2<\infty$, donc $\psi\in H_2$. On a donc
montr\'e que $D'=H_2$.

\end{proof}

 \subsection{Rappels sur la th\'eorie des fonctions de Bessel}\label{paragraphe}


Quelques propri\'et\'es de fonctions de Bessel utilis\'ees dans cet article. Dans cette introduction \`a la th\'eorie des fonctions de Bessel, la principale r\'ef\'erence sera le chapitre 17 du \cite{Whittaker}, on s'int\'eressera \`a une sous-classe de fonctions de Bessel \`a savoir les fonctions de Bessel d'ordre un entier.

Pour tout $z\in \C$ fix\'e, la fonction:
{{} \[
t\mapsto e^{\frac{1}{2}z(t-\frac{1}{t})},\quad t\neq 0,
\]}

admet un d\'eveloppement en  s\'erie de Laurent en $t$. Soit $n\in \Z$. Par d\'efinition, la fonction de Bessel d'ordre $n$ est la fonction qui \`a tout $z\in \C$, associe  le coefficient de $t^n$ dans ce d\'eveloppement, on le note par $J_n(z)$.
\'enonçons quelques propri\'et\'es de la fonction $J_n$:
\begin{itemize}
\item[$\bullet$] $J_n$ est une fonction analytique sur $\C$ telle que $J_{-n}=(-1)^nJ_n$ et
{{} \begin{equation}\label{ordreenzero}
J_n(z)=\sum_{r=0}^\infty \frac{(-1)^r(\frac{1}{2}z)^{n+2r}}{\Gamma(r+1)\Gamma(n+r+1)}, \quad\text{si}\quad n\in \N
\end{equation}}
\item[$\bullet$] $J_n$ est une solution de l'\'equation diff\'erentielle lin\'eaire suivante:
{{} \begin{equation}\label{besselequation}
\frac{d^2 y}{d^2 z}+\frac{1}{z}\frac{dy}{dz}+\bigl(1-\frac{n^2}{z^2} \bigr)y=0
\end{equation}}

On a les relations de r\'ecurrence suivantes:

\item[$\bullet$]
\begin{equation}\label{B1}
\frac{d}{dz}(z^{-n}J_n(z))=-z^{-n}J_{n+1}(z).
\end{equation}

\item[$\bullet$]
\begin{equation}\label{B5}
\frac{d}{dz}(z^n J_n(z))=z^n J_{n-1}(z)
\end{equation}

\item[$\bullet$] Les z\'eros de $J_n$ sont r\'eels.\\

\item[$\bullet$] Pour tous $a\neq b$, on a:
\begin{equation}\label{intTTT}
\begin{split}
\int_0^1xJ_n(ax)J_n(bx)dx&=\frac{1}{b^2-a^2}\Bigl(aJ_n(b)J_n'(a)-bJ_n(a)J_n'(b)  \Bigr),
\end{split}
\end{equation}
et
\begin{equation}\label{encoreeq}
\int_0^1x J_n(ax)^2dx=\frac{1}{2}\bigl(J_n'(a)^2+(1-\frac{n^2}{a^2})J_n(a)^2  \bigr).
\end{equation}
voir par exemple \cite[p. 381, 18.]{Whittaker}
\end{itemize}

Soient $Z_m=\{j_{m,1}<j_{m,2}<\ldots\}$ l'ensemble des  z\'eros positifs non nuls de $J_m$ ordonn\'es par ordre croissant, voir par exemple \cite[\S 15]{Watson}.  Soient $D_m:=\{c_{m,1}<c_{m,2}<\ldots\}$ les z\'eros positifs non nuls de $J'_m$ ordonn\'es par ordre croissant. Comme les z\'eros de $J_m$ sont simples, voir par exemple \cite[\S 15.21]{Watson}, alors
{{}
{{}
\begin{equation}\label{z\'eros}
 Z_m\cap D_m=\emptyset.
\end{equation}}}

\bibliographystyle{plain}
\bibliography{biblio}

\vspace{1cm}

\begin{center}
{\sffamily \noindent National Center for Theoretical Sciences, (Taipei Office)\\
 National Taiwan University, Taipei 106, Taiwan}\\

 {e-mail}: {hajli@math.jussieu.fr}

\end{center}

\end{document}